\documentclass[11pt,a4paper,reqno]{amsart} 
\usepackage{anysize}
\marginsize{3.5cm}{3.5cm}{2.2cm}{2.2cm}

\usepackage{hyperref} % [hypertex]
\hypersetup{
    colorlinks=true,       
    linkcolor=red,          
    citecolor=blue,        
    filecolor=magenta,      
    urlcolor=cyan           
}
\usepackage{standalone}
\usepackage[english]{babel}
\usepackage{amsmath}
\usepackage{amsfonts,amssymb}
\usepackage{enumerate}
\usepackage{mathrsfs}
\usepackage{amsthm}
\usepackage{bbm}
\usepackage[pagewise,displaymath, mathlines]{lineno}
\usepackage{tikz}
\usepackage{tikz}
\usetikzlibrary{intersections,decorations.markings}
\usetikzlibrary{calc}

\newfont{\bb}{msbm10 at 12pt}
\newfont{\bbt}{msbm10 at 9pt}
\def\r{\mathbb R}
\def\d{\mathbb D}
\def\n{\mathbb N}
\def\z{\mathbb Z}
\def\t{\mathbb T}
\def\c{\hbox{\bb C}}

\def\b{\hbox{B}}

\def\fr{\mathscr F}
\def\sr{\mathscr S}
\def\dr{\mathscr D}
\def\w{\mathscr W}
\def\er{\mathscr E}

\DeclareMathOperator{\RE}{Re}
\DeclareMathOperator{\supp}{supp}
\DeclareMathOperator{\D}{\langle D \rangle}
\DeclareMathOperator{\DIV}{div}

\DeclareMathOperator{\op}{Op}

\newcommand{\norm}[1]{\left\Vert #1 \right\Vert}
\newcommand{\abs}[1]{\left\vert #1 \right\vert}
\newcommand{\set}[1]{\left\{#1\right\}}
\newcommand{\eps}{\epsilon}

\usepackage{hyperref}
\hypersetup{
    colorlinks=true,
    linkcolor=red,
    filecolor=red,      
    urlcolor=blue,
    citecolor=red,
    pdftitle={Positive regularity results on the flow map of periodic quasi-linear dispersive Burgers type equations and the Gravity-Capillary equations},
    bookmarks=true,
    pdfpagemode=FullScreen,
}

\numberwithin{equation} {section}

\theoremstyle{plain}\newtheorem{lemma}{Lemma}[section]
\theoremstyle{plain}\newtheorem{proposition}{Proposition}[section]
\theoremstyle{plain}\newtheorem{theorem}{Theorem}[section]
\theoremstyle{plain}\newtheorem{definition}{Definition}[section]
\theoremstyle{plain}\newtheorem{remark}{Remark}[section]
\theoremstyle{plain}\newtheorem{corollary}{Corollary}[section]
\theoremstyle{plain}
\theoremstyle{plain}\newtheorem{notation}{Notation}[section]
\theoremstyle{plain}\newtheorem{definition-notation}{Definition-Notation}[section]
\theoremstyle{plain}\newtheorem{definition-proposition}{Definition-Proposition}[section]
\numberwithin{equation}{section}

\newcommand{\keyword}[1]{\textbf{\textit{Keywords---}} \textbf{#1}}

\title{Regularity results on the flow maps of periodic dispersive Burgers type equations and the Gravity-Capillary equations}%: \[\dps \partial_t u+Re(u)\partial_xu+i\D^\alpha u=0,\]
\author{Ayman Rimah Said}

\def\notina[#1]#2{\begingroup\def\thefootnote{\fnsymbol{footnote}}\footnote[#1]{#2}\endgroup}

 \def\Xint#1{\mathchoice
      {\XXint\displaystyle\textstyle{#1}}%
      {\XXint\textstyle\scriptstyle{#1}}%
      {\XXint\scriptstyle\scriptscriptstyle{#1}}%
      {\XXint\scriptscriptstyle\scriptscriptstyle{#1}}%
      \!\int}
   \def\XXint#1#2#3{{\setbox0=\hbox{$#1{#2#3}{\int}$}
        \vcenter{\hbox{$#2#3$}}\kern-.5\wd0}}
   
   \def\dashint{\Xint-}
\begin{document}
\iffalse
\begin{titlepage}
   \begin{center}
       \vspace*{5cm}

       \textbf{Regularity results on the flow maps of periodic dispersive Burgers type equations and the Gravity-Capillary equations}

 \vspace{1.5cm}
	
       \textbf{Ayman Rimah Said}

       \text{Department of Mathematics, Duke University} 
             \vspace{1.5cm}
       \subsection*{Acknowledgement}
I would like to express my sincere gratitude to my thesis advisor Thomas Alazard. 
   \section*{Declarations}
   \subsection*{Conflict of interest}
   The author states that there is no conflict of interest.
   \subsection*{Funding}
   No funding was received to assist with the preparation of this manuscript.
   \subsection*{Financial interests}
   The author has no relevant financial or non-financial interests to disclose.
   \subsection*{Data availability}
   This manuscript has no associated data.
   \end{center}
\end{titlepage}
\fi
%\notina[0]{ PhD student at l'IMO, Paris Sud University and Centre Borelli, ENS Paris Saclay. email: \url{aymanrimah@gmail.com}.}

\begin{abstract}
In the first part of this paper we prove that the flow associated to a dispersive Burgers equation with a non local term
 of the form $\abs{D}^{\alpha-1} \partial_x  u$, $\alpha \in [1,+\infty[$ is Lipschitz from bounded sets of $H^s_0(\t;\r)$ to $C^0([0,T],H^{s-(2-\alpha)^+}_0(\t;\r))$ for $T>0$ and $s>\lceil \frac{\alpha}{\alpha-1}\rceil-\frac{1}{2}$, where $H^s_0$ are the Sobolev spaces of functions with $0$ mean value, proving that the result obtained in \cite{Ayman19} is optimal on the torus. The proof relies on a paradifferential generalization of a complex Cole-Hopf gauge transformation introduced by T.Tao in \cite{Tao04} for the Benjamin-Ono equation. 
 
 For this we prove a generalization of the Baker-Campbell-Hausdorff formula for flows of hyperbolic paradifferential equations and prove the stability of the class of paradifferential operators modulo more regular remainders, under conjugation by such flows. For this we prove a new characterization of paradifferential operators in the spirit of Beals \cite{Beals 77}.
 
 In the second part of this paper we use a paradifferential version of the previous method to prove that a re-normalization of the flow of the one dimensional periodic gravity capillary equation is Lipschitz from bounded sets of $H^s$ to $C^0([0,T],H^{s-\frac{1}{2}})$ for $T>0$ and $s>3+\frac{1}{2}$. This proves that the result obtained in \cite{Ayman19} is optimal for the water waves system.
 
 \keyword{Flow map, Regularity, Quasi-linear, nonlinear Burgers type dispersive equations, Water Waves system, Gravity-Capillary equations, Cole-Hopf Gauge transform.}
\end{abstract}

\maketitle

\vspace{-5mm}

\tableofcontents

\vspace{-7mm}

\section{Introduction}
In our study of the quasi-linearity of the water waves system in \cite{Ayman19} we studied the flow map regularity for some model nonlinear dispersive equations of the form:
\begin{equation} \label{pos reg_intro_DBeq}
\partial_t u+u\partial_x u+\abs{D}^{\alpha-1} \partial_x  u=0 \text{ on $\d$,}
\end{equation}
where $\d=\r  \text{ or }  \t $, $\alpha \in [0,2[$ and $\abs{D}$ is the Fourier multiplier with symbol $\abs{\xi}$. We proved that they are quasi-linear. We based our work on the following distinction between semi-linearity and quasi-linearity given in \cite{Saut02}:
\begin{itemize}
\item A partial differential equation is said to be semi-linear if its flow map is regular (at least $C^1$).
\item A partial differential equation is said to be quasi-linear if its flow map is not Lipschitz.
\end{itemize}

More precisely we proved that:
\begin{itemize}
\item the flow map associated to \eqref{pos reg_intro_DBeq} fails to be uniformly continuous from bounded sets of $H^s(\d)$ to $C^0([0,T],H^s(\d))$ for $T>0$ and $s>2+\frac{1}{2}$. 
\end{itemize}

The drawback of this test of quasi-linearity, where we only analyse the uniform continuity of the flow map, is that it does not show the effect of the dispersive term. The natural question is then to ask if one can know $\alpha$ exactly by having a more refined analysis of the regularity of the flow map. 

For this we can start by noticing that independently of $\alpha$ the flow map is Lipschitz from bounded sets of $H^s(\d)$ to $C^0([0,T],H^{s-1}(\d))$ and ask: can the space $H^{s-1}(\d)$ be replaced by $H^{s-\mu}(\d)$ with $\mu<1$ depending on $\alpha$? Again in \cite{Ayman19} we proved that the best $\mu$ one can hope for is $\mu=1-(\alpha-1)^+$, where $a^+:=\max(a,0)$, more precisely we showed that: 
\begin{itemize}
\item the flow map cannot be Lipschitz from bounded sets of $H^s(\d)$ to\\ $C^0([0,T],H^{s-1+(\alpha-1)^+ +\eps}(\d))$ for $\eps>0$.
\end{itemize}

Looking at the literature to assess the optimality of the result, first in \cite{Saut13} the equation \eqref{pos reg_intro_DBeq} is actually shown to be quasi-linear for $\alpha \in [0,3[$ and becomes semi-linear for $\alpha=3$, that is the Korteweg-de Vries equation, when $\d=\r$ suggesting that our results are sub-optimal. Then when $\d=\t$, in \cite{Molinet08}, for the case $\alpha=2$ and the Benjamin--Ono equation, the flow map is shown to be Lipschitz (and even has analytic regularity) on bounded subsets of $H^s_0$ the (Sobolev spaces of functions with $0$ mean value). Which suggests that our results could be optimal but with a subtlety in the low frequencies.

 The aim of the current paper is to prove that the results obtained in \cite{Ayman19} are optimal on the torus and for the full periodic water waves system with surface tension, that is the gravity capillary equation. 

%In Appendix \ref{pos reg_section gauge transf on R}, we look to the problem on $\r$ and use the same Gauge transform to show that the lack of regularity obtained in \cite{Saut13} for $\alpha \geq 2$ is essentially due to the lack of control of the $L^1$ norm by $\r$ based Sobolev norms, which is essentially a low frequency problem.

Before we give the main results of this paper, it is important to place the question of the flow map regularity in the vast and rich literature that studies equations of the form \eqref{pos reg_intro_DBeq}, for a comprehensive and complete overview of those equations and their link to other problems coming from mechanical fluids and dispersive non linear equations in physics we refer to Saut's \cite{Saut13,Saut18}. Beyond the starting point of hyperbolic local well-posedness of \eqref{pos reg_intro_DBeq} in $H^s, s>1+\frac{1}{2}$, three natural questions are: (1) is the problem globally well-posed or is there blow up in finite time? (2) What is the smallest $s$ for which we still have local well-posedness? (3) Is the continuous dependence on the initial datum optimal, that is, is the equation semi/quasi-linear? The first two questions are usually closely connected due to the existence of conservation laws and form some of the most important problems in PDE today. Understanding them demands a better understanding of the equation beyond the basic hyperbolic structure. For the equations analysed in this paper, this means that a refined understanding of the interaction between the dispersive and nonlinear transport terms is needed to answer those questions. For $\alpha\leq 1 $ and $\alpha\geq 2$ the problem for equation \eqref{pos reg_intro_DBeq} is now very well understood. Indeed for $\alpha\leq 1$ the equation is known to exhibit finite time blow up and does so through a wave breaking scenario \cite{Castro10,Hur12,Hur17,Hur18,Pasqualotto21,Saut20'}. For $\alpha \geq 2$ the equation is globally well-posed and the optimal threshold $s_c$ is known \cite{Tao04,Koch03,Linares14,Ionescu07,Molinet12,Ifrim17}, for the special cases where \eqref{pos reg_intro_DBeq} is integrable that is the Benjamin--Ono equation ($\alpha=2$) and the KdV equation ($\alpha=3$) the equation is even better understood due to the remarkable construction of Birkhoff coordinates \cite{Gerard20,Kappler06,Killip19}. For $1<\alpha<2$, the global Cauchy problem is not as well understood, a numerical study was carried out by Klein and Saut in \cite{Saut14} and they conjectured that there is blow up for $1<\alpha\leq \frac{3}{2}$ and that the equations are globally well-posed for $\alpha>\frac{3}{2}$. To the best of the author's knowledge the only progress towards answering this conjecture was given in \cite{Molinet18} where the authors proved global well posedness for $\alpha>1+\frac{6}{7}$.  One of the main goals of the current work is to show that a refined understanding of the third question sheds some light on some interactions between the dispersive and nonlinear terms. This understanding could later be used to give answers to the first two questions. Indeed determining $\alpha$ here through the understanding of the exact regularity of the flow map required us to construct a generalized Baker--Campbell--Hausdorff formula for hyperbolic paradifferential equations. We use those tools developed here in a subsequent work \cite{Ayman20'} to answer global Cauchy type questions. 

\subsection{On the torus} 
We show that the flow map associated to \eqref{pos reg_intro_DBeq} is Lipschitz from bounded sets of $H^s_0(\t)$ to $C^0([0,T],H^{s-1+(\alpha-1)^+}_0(\t))$. We begin by recalling a classical result in the literature (\cite{Metivier08,Saut13}).
\begin{theorem}\label{pos reg_intro_th standard CP DB}
Consider three real numbers $\alpha\in [0,2[$, $s\in]2+\frac{1}{2},+\infty[$, $r>0$ and $u_0 \in H^{s}(\d)$. Then there exists $C_s>0$ such that for $0<T< \frac{C_s}{r+\norm{\partial_x u_0}_{L^\infty(\d)}}$ and all $v_0$ in the ball $\b(u_0,r)\subset  H^{s}(\d)$ there exists a unique $v\in C^0([0,T],H^s(\d))$ solving the Cauchy problem:
\begin{equation}\label{pos reg_intro_th standard CP DB_eq mod on t}
\begin{cases} 
\partial_t v+v\partial_xv+\abs{D}^{\alpha-1} \partial_xv=0 \\
v(0,\cdot)=v_0(\cdot).
\end{cases}
\end{equation}
Moreover, for all $\mu \in [0,s],\ \exists C_\mu \in \r_+$ such that:
\begin{equation}
 \forall t\in [0,T],\ \norm{v(t)}_{H^{\mu}(\d)} \leq e^{C_\mu \norm{\partial_x v}_{L^1([0,T],L^\infty(\d))}}  \norm{v_0}_{H^{\mu}(\d)}. 
\end{equation}
Taking $v_0 \in \b(u_0,r)$, and assuming moreover that $u_0 \in H^{s+1}(\d)$ then:
\begin{align}
 \forall t\in [0,T],\norm{(u-v)(t)}_{H^{s}(\d)} &\leq e^{C_s( \norm{\partial_x (u,v)}_{L^1([0,t],L^\infty(\d))}+C_s \norm{u}_{L^1([0,t],H^{s+1}(\d))})}  \norm{u_0-v_0}_{H^{s}(\d)},
\end{align}
where $u$ is the solution emanating from $u_0$.
\end{theorem}
In \cite{Ayman19} we studied the regularity of the flow map and proved the following. 
\begin{theorem}[\cite{Ayman19}]\label{pos reg_intro_th CP DB geo quasilinear} 
Consider three real numbers $\alpha\in [0,2[$, $s\in ]2+\frac{1}{2},+\infty[$, $r>0$ and $u_0 \in H^{s}(\t;\r)$. Let $C_s$ be given by Theorem \ref{pos reg_intro_th standard CP DB} and  $0<T< \frac{C_s}{r+\norm{\partial_x u_0}_{L^\infty(\d)}}$.
\begin{itemize}
 \item Then the flow map associated to the Cauchy problem \eqref{pos reg_intro_th standard CP DB_eq mod on t}:
	\begin{align*}
										\b(u_0,r) \rightarrow &C^0([0,T],H^s(\t;\r))\\
										v_0 \mapsto &v
										\end{align*}
				is continuous but not uniformly continuous.
\item Moreover for all $\eps>0$ the flow map:
	\begin{align*}
										\b(u_0,r) \rightarrow &C^0([0,T],H^{s-1+(\alpha-1)^+ +\eps}(\t;\r))\\
										v_0 \mapsto &v
										\end{align*}
				is not Lipschitz.
\end{itemize}
\end{theorem}				
Here we prove that these results are essentially optimal on the torus, more precisely we prove the following theorem.
\begin{theorem}\label{pos reg_intro_th CP DB exact flow reg}
Consider three real numbers $\alpha\in ]1,2[$, $s\in ]\lceil \frac{\alpha}{\alpha-1}\rceil-\frac{1}{2},+\infty[$, $r>0$ and $u_0 \in H^{s}_0(\t;\r)$. Let $C_s$ be given by Theorem \ref{pos reg_intro_th standard CP DB} and  $0<T< \frac{C_s}{r+\norm{\partial_x u_0}_{L^\infty(\d)}}$.
Then the flow map associated to the Cauchy problem \eqref{pos reg_intro_th standard CP DB_eq mod on t}
	\begin{align*}
										\b(u_0,r)\cap H_0^s(\t;\r) \rightarrow &C^0([0,T],H^{s-(2-\alpha)^+}_0(\t;\r))\\
										v_0 \mapsto &v
										\end{align*}
				is Lipschitz, more precisely we have the a priori estimate for $t\in [0,T]$
\[
\norm{(u-v)(t)}_{H^{s-(2-\alpha)^+}(\d)} \leq e^{Cte^{Ct\norm{(u,v)}_{L^\infty_tW_x^{\lceil \frac{\alpha}{\alpha-1}\rceil-1,\infty}}}\norm{v_0}_{H^{s}}} \norm{u_0-v_0}_{H^{s-(2-\alpha)^+}(\d)}.
\]	
\end{theorem}

Several remarks are in order.
\begin{enumerate}
\item As a corollary of Theorem \ref{pos reg_intro_th CP DB exact flow reg} we prove in Section \ref{pos reg_study of mod pbm_subsec proof cor DB exact flow reg} the following.
\begin{corollary}\label{pos reg_intro_cor CP DB exact flow reg}
Consider three real numbers $\alpha\in ]1,2[$, $s> \lceil \frac{\alpha}{\alpha-1}\rceil-\frac{1}{2}$, $ r>0,$ and $u_0 \in H^{s}(\t;\r)$. Let $C_s$ be given by Theorem \ref{pos reg_intro_th standard CP DB} and  $0<T< \frac{C_s}{r+\norm{\partial_x u_0}_{L^\infty(\d)}}$.
\begin{itemize}
\item Then the flow map associated to the Cauchy problem \eqref{pos reg_intro_th standard CP DB_eq mod on t}:
	\begin{align*}
										\b(u_0,r) \rightarrow &C^0([0,T],H^s(\t;\r))\\
										v_0 \mapsto &v
										\end{align*}
				is continuous but not uniformly continuous.
\item For all $\eps>0$ the flow map:
	\begin{align*}
										\b(u_0,r) \rightarrow &C^0([0,T],H^{s-1 +\eps}(\t;\r))\\
										v_0 \mapsto &v
										\end{align*}
				is not $C^1$.
\end{itemize}
\end{corollary}

\item The case $\alpha=\frac{3}{2}$ is closely related to the system obtained after reduction and para-linearization of the periodic water waves system in dimension 1 obtained in \cite{Alazard11} Proposition 3.3 by  Alazard, Burq and Zuily, which we will treat in the second part of this paper.

\item The case $\alpha=2$ and the Benjamin-Ono equation on the circle was obtained by Molinet in \cite{Molinet08}. Though Molinet's result extends to the Cauchy problem on $L^2(\t)$ and only studied the flow map regularity for data with $0$ mean value. 

\end{enumerate}

The main language and techniques used in this article is that of paraproducts, paracomposition paradifferential operators and paradifferential calculus for which a rigorous review is given in Appendix \ref{paracomposition_Notions of microlocal analysis_Paradifferential Calculus}. We give an intuitive interpretation of those concepts in the following paragraph so that the reader unfamiliar with this language can get a good grasp of the statements without having to go through Appendix \ref{paracomposition_Notions of microlocal analysis_Paradifferential Calculus} first.
 
\subsubsection*{\textbf{$\bullet$ Paraproducts, paracomposition and Paradifferential operators}} 
For the sake of this discussion let us pretend that $\partial_x$ is left-invertible with a choice of $\partial_x^{-1}$ that acts continuously from $H^s$ to $H^{s+1}$. We follow here analogous ideas to the ones presented by Shnirelman in \cite{Shnirelman05}. One way to define the paraproduct of two functions $f,g\in H^s$ with $s$ sufficiently large is: we differentiate $fg$ $k$ times, using the Leibniz formula, and then restore the function $fg$ by the $k$-th power of $\partial_x^{-1}$:
 \begin{align*}
 fg&=\partial_x^{-k}\partial_x^{k}(fg)\\
 	&=\partial_x^{-k}\big(g\partial_x^k f+k\partial_x g\partial_x^{k-1} f+\dots+k\partial_x f\partial_x^{k-1} g+g\partial_x^k f  \big)\\
 	&=T_g f+T_fg+R,
 \end{align*}
 where,
 \[T_gf=\partial_x^{-k}\big(g\partial_x^k f\big), \ \ T_fg=\partial_x^{-k}\big(f\partial_x^k g\big),\]
 and $R$ is the sum of all remaining terms. The key observation is that if $s>\frac{1}{2}+k$, then $g \mapsto T_fg$ is a continuous operator in $H^s$ for $f  \in H^{s-k}$. The remainder $R$ is a continuous bilinear operator from $H^s$ to $H^{s+1}$. The operator $T_fg$ is called the paraproduct of $g$ and $f$ and can be interpreted as follows. The term $T_fg$ takes into play high frequencies of $g$ compared to those of $f$ and demands more regularity in $g\in H^s$ than $f \in H^{s-k}$ thus the term $T_fg$ bears the "singularities" brought on by $g$ in the product $fg$. Symmetrically $T_gf$ bears the "singularities" brought on by $f$ in the product $fg$ and the remainder $R$ is a smoother function ($H^{s+1}$) and does not contribute to the main singularities of the product. Notice that this definition uses a "general" heuristic from PDE that is the worst terms are the highest order terms (ones involving the highest order of differentiation). Now to make such a definition rigorous, we quantify this frequency comparison. The starting point is the product formula:
 \[
 fg(x)=\frac{1}{(2\pi)^2}\int\limits_{\r}\int\limits_{\r}e^{ix\cdot(\xi_1+\xi_2)}\fr(f)(\xi_1)\fr(g)(\xi_2)d\xi_1d\xi_2.
 \]
 Now if for some parameters $B>1,b>0$ one defines a cut-off function:
 \[
\psi^{B,b}(\eta,\xi)=0 \text{ when }
\abs{\xi}< B\abs{\eta}+b,
\text{ and }
\psi^{B,b}(\eta,\xi)=1 \text{ when } \abs{\xi}>B\abs{\eta}+b+1,
\]
then one can rigorously define the paraproduct as
\[
T^{B,b}_{g}f(x)=T_{g}f(x)=\frac{1}{(2\pi)^2}\int\limits_{\r}\int\limits_{\r}\psi^{B,b}(\xi_2,\xi_1)e^{ix\cdot(\xi_1+\xi_2)}\fr(f)(\xi_1)\fr(g)(\xi_2)d\xi_1d\xi_2.
\]

 To get a good intuition of a paradifferential operator $T_p$ with symbol $p\in \Gamma^\beta_\rho$, as a first gross approximation, one can think of $T_p$ as the composition of a paraproduct $T_f$ with Fourier multiplier $m(D)$, that is:
 \[
 T_p\approx T_f m(D), \text{ with } f\in W^{\rho,\infty} \text{ and }m \text{ is of order }\beta.
 \]
 Indeed following Coifman and Meyer's symbol reduction Proposition $5$ of \cite{Coifman78}, one can show that linear combinations of composition of a paraproduct with a Fourier multiplier are dense in the space of paradifferential operators.
 
Finally for the paracomposition operation we again work with $f \in H^s$ and $g \in C^s$ with $s$ large and consider the composition of two functions $f\circ g$ which bears the singularities of both $f$ and $g$, and our goal is to separate them. We proceed as before by differentiating $f \circ g$ $k$ times, using the Fa\'a di Bruno's formula, and then restore the function $fg$ by the $k$-th power of $\partial_x^{-1}$:
\begin{align*}
 f \circ g&=\partial_x^{-k} \partial_x^{k} (f \circ g)\\
 	&=\partial_x^{-k}\big((\partial_x^k f\circ g)\cdot(\partial_x g)^k 
 	+\dots+(\partial_x f\circ g)\cdot\partial_x^k g  \big)\\
 	&=g^*f+T_{\partial_x f \circ g}g+R,
 \end{align*}
 where,
 \[g^*f=\partial_x^{-k}\big((\partial_x^k f\circ g)\cdot(\partial_x g)^k \big) \text{ is the paracomposition of $f$ by $g$}\]
 and $R$ is the sum of all remaining terms. Again the key observation is that if $s>\frac{1}{2}+k$, then $f \mapsto g^*f$ is a continuous operator in $H^s$ for $g  \in C^{s-k}$. Thus this term bears essentially the singularities of $f$ in $f\circ g$. As before $T_{\partial_x f \circ g}g$ bears essentially the singularities of $g$ in $f\circ g$. The remainder $R$ is a continuous bilinear operator from $H^s$ to $H^{s+1}$. Thus we have separated the singularities of the composition $f\circ g$.\\

\subsection{The periodic gravity capillary equation}
 We follow here the presentation in \cite{Alazard11}, \cite{Alazard16} and \cite{Alazard18}.
\subsubsection{Assumptions on the domain}
 We consider a domain with free boundary, of the form:
 \[\set{(t,x,y) \in [0,T]\times \r \times \r:(x,y)\in \Omega_t},\]
 where $\Omega_t$ is the domain located between a free surface:
 \[\Sigma_t=\set{(x,y)\in \r\times \r:y=\eta(t,x)}\]
 and a given (general) bottom denoted by $\Gamma = \partial \Omega_t \setminus \Sigma_t$. More precisely we assume that initially $(t=0)$ we have the hypothesis \eqref{pos reg_intro_WW_asump dom_$H_t$} given by:
\begin{itemize}
 \item The domain $\Omega_t$ is the intersection of the half space, denoted by $\Omega_{1,t}$, located below the free surface $\Sigma_t$,
 \[\Omega_{1,t}=\set{(x,y)\in \r\times \r:y<\eta(t,x)}\tag{\textbf{$H_t$}}\label{pos reg_intro_WW_asump dom_$H_t$}\]
 and an open set $\Omega_2 \subset \r^{1+1}$ such that $\Omega_2$ contains a fixed strip around $\Sigma_t$, which means that there exists $h>0$ such that,
 \[\set{(x,y)\in \r \times \r: \eta(t,x)-h\leq y \leq \eta(t,x)}\subset \Omega_2.\tag{\textbf{$H_t$}}\label{pos reg_intro_WW_asump dom_$H_t$}\]
 We shall assume that the domain $\Omega_2$ (and hence the domain $\Omega_t=\Omega_{1,t} \cap \Omega_2$) is connected.
 \end{itemize}
\subsubsection{The equations}
 We consider an incompressible inviscid liquid, having unit density. The equations of motion are given by the Euler system of the velocity field $v$: 
 \begin{equation}\label{pos reg_intro_WW_equations_eq Euler}
 \begin{cases}
 \partial_t v+v\cdot \nabla v+\nabla P=-ge_y \\
  \DIV v=0
\end{cases} 
  \text{ in }  \Omega_t,
 \end{equation}
 where $-ge_y$ is the acceleration of gravity $(g>0)$ and where the pressure term $P$ can be recovered from the velocity by solving an elliptic equation. The problem is then coupled with 
 the boundary conditions:
 \begin{align}\label{pos reg_intro_WW_equations_cond on bound eq 1}
 \begin{cases}
 v\cdot n=0 &  \text{on }  \Gamma, \\
 \partial_t \eta=\sqrt{1+(\partial_x \eta)^2}v \cdot \nu &  \text{on }  \Sigma_t,\\
 P=-\kappa H(\eta)  &  \text{on }  \Sigma_t,
 \end{cases}
 \end{align}
 where $n$ and $\nu$ are the exterior normals to the bottom $\Gamma$ and the free surface $\Sigma_t$, $\kappa$ is the surface tension and $H(\eta) $ is the mean curvature of the free surface:
 \[
 H(\eta)=\partial_x \bigg(\frac{\partial_x \eta}{\sqrt{1+(\partial_x \eta)^2}}\bigg)
 .\]

  We are interested in the case with surface tension and take $\kappa=1$. The first condition in \eqref{pos reg_intro_WW_equations_cond on bound eq 1} expresses the fact that the particles in contact with the rigid bottom remain in contact with it. As no hypothesis is made on the regularity of $\Gamma,$ this condition makes sense in a weak variational meaning due to the hypothesis \eqref{pos reg_intro_WW_asump dom_$H_t$}, for more details on this we refer to Section 2 in \cite{Alazard11}.\\
  
  The fluid motion is supposed to be irrotational and $\Omega_t$ is supposed to be simply connected thus the velocity $v$ field derives
   from some potential $\phi$ that is $v=\nabla \phi$ and:
\[\begin{cases}
\Delta \phi=0  \text{ in }  \Omega,\\  \partial_n \phi=0  \text{ on } \Gamma.
\end{cases}
\]

The boundary condition on $\phi$ becomes:
 \begin{align}\label{pos reg_intro_WW_equations_cond on bound eq 2}
 \begin{cases}
 \partial_n \phi =0 &  \text{on } \Gamma, \\
 \partial_t \eta=\partial_y \phi -\partial_x \eta \partial_x  \phi & \text{on } \Sigma_t,\\
 \partial_t \phi=-g\eta+ H(\eta)-\frac{1}{2} \abs{\nabla_{x,y} \phi}^2 &  \text{on }  \Sigma_t.
 \end{cases}
 \end{align}

Following Zakharov \cite{Zakharov68} and Craig-Sulem \cite{Craig93} we reduce the analysis to a system on the free surface $\Sigma_t$. If $\psi$
is defined by 
\[\psi(t,x)=\phi(t,x,\eta(t,x)),\]
then $\phi$ is the unique variational solution of
\[\Delta \phi =0  \text{ in } \Omega_t, \  \phi_{|y=\eta}=\psi, \   \partial_n \phi =0  \text{ on } \Gamma.\]
Define the Dirichlet-Neumann operator by 
\begin{align*}
(G(\eta)\psi)(t,x)&=\sqrt{1+(\partial_x \eta)^2}\partial_n \phi_{|y=\eta}(t,x)\\
&=(\partial_y \phi)(t,x,\eta(t,x))-\partial_x \eta(t,x) \partial_x  \phi(t,x,\eta(t,x)).
\end{align*}
For the case with rough bottom we refer to \cite{Alazard09},  \cite{Alazard11} and \cite{Alazard16} for the well-posedness of the variational problem and the Dirichlet-Neumann operator.
Now $(\eta,\psi)$ (see for example \cite{Craig93}) solves:
  \begin{align}\label{pos reg_intro_WW_equations_WW syst}
 \partial_t \eta&=G(\eta)\psi ,\\
 \partial_t \psi&=-g\eta+ H(\eta)+\frac{1}{2} (\partial_x  \psi)^2  +\frac{1}{2}\frac{\partial_x  \eta \partial_x  \psi+G(\eta)\psi}{1+(\partial_x \eta)^2}.\nonumber
 \end{align}
The system is completed with initial data
\[\eta(0,\cdot)=\eta_{in}, \ \psi(0,\cdot)=\psi_{in}.\]
We consider the case when $\eta$, $\psi$ are $2\pi$-periodic in the  space variable $x$. 
\subsubsection{Flow map regularity}
In \cite{Alazard11} and \cite{Alazard16},  Alazard, Burq, and Zuily perform a paralinearization and symmetrization of the the water waves system that takes the form:
\[
\partial_tu+ T_V.\nabla u+i T_{\gamma}u=f,
\]
where $\gamma$ is an elliptic symbol of order $\frac{3}{2}$ which closely resembles the model problem we presented on $\t$ but with an extra non-linearity in $\gamma$. The paralinearization and symmetrization of the system was used to prove the well-posedness of the Cauchy problem in the optimal threshold $s>2+\frac{1}{2}$ in which the velocity field $v$ is Lipschitz. We will complete this and our result in \cite{Ayman19} by giving the precise regularity of the flow map. First we recall some previously known results on the Cauchy problem from \cite{Alazard16, Alazard11}.
\begin{theorem}[From \cite{Alazard16,Alazard11}]\label{pos reg_intro_WW_flow map reg_Th CP Alazard}
Consider two real numbers $r>0$,
 $s \in]2+ \frac{1}{2},+\infty[$ and $(\eta_0,\psi_0)\in H^{s+\frac{1}{2}}(\t) \times H^{s}(\t) $ such that, 
 $$\forall (\eta'_0,\psi'_0) \in \b((\eta_0,\psi_0),r)\subset H^{s+\frac{1}{2}}(\t) \times H^{s}(\t)$$
  the assumption $\eqref{pos reg_intro_WW_asump dom_$H_t$}_{t=0}$ is satisfied. Then there exists $T>0$ such that the Cauchy problem \eqref{pos reg_intro_WW_equations_WW syst} with initial data $(\eta'_0,\psi'_0)\in \b((\eta_0,\psi_0),r)$ has a unique solution
\[(\eta',\psi') \in C^0([0,T];H^{s+\frac{1}{2}}(\t) \times H^{s}(\t) )\]
and such that the assumption \eqref{pos reg_intro_WW_asump dom_$H_t$} is satisfied for $t \in [0,T]$. Moreover the flow map $(\eta'_0,\psi'_0)\mapsto (\eta',\psi') $ is continuous. 
\end{theorem}
In \cite{Ayman19} we completed this by the following.
\begin{theorem}\label{pos reg_intro_WW_flow map reg_Th CP geo quasi}
Consider two real numbers $r>0$,
 $s \in]2+ \frac{1}{2},+\infty[$ and $(\eta_0,\psi_0)\in H^{s+\frac{1}{2}}(\t) \times H^{s}(\t) $ such that, 
 $$\forall (\eta'_0,\psi'_0) \in \b((\eta_0,\psi_0),r)\subset H^{s+\frac{1}{2}}(\t) \times H^{s}(\t)$$
  the assumption $\eqref{pos reg_intro_WW_asump dom_$H_t$}_{t=0}$ is satisfied.
  
  Then for all $ R>0$ the flow map associated to the Cauchy problem \eqref{pos reg_intro_WW_equations_WW syst}: 
	\begin{align*}
										\b(0,R)  \rightarrow &C^0([0,T],H^{s+\frac{1}{2}}(\t) \times H^{s}(\t))\\
										(\eta'_0,\psi'_0) \mapsto &(\eta',\psi')
										\end{align*}
				is not uniformly continuous.
								
And at least a loss of $\frac{1}{2}$ derivative is necessary to have Lipschitz control over the flow map, that is for all $\eps'>0$ the flow map
	\begin{align*}
										\b(0,R)  \rightarrow &C^0([0,T],H^{s+\eps'}(\t) \times H^{s-\frac{1}{2}+\eps'}(\t))\\
										(\eta'_0,\psi'_0) \mapsto &(\eta',\psi')
										\end{align*}

				is not Lipschitz.
\end{theorem}
Here it is shown that those results are sufficient after suitable re-normalization of the flow map. 
\begin{theorem}\label{pos reg_intro_WW_flow map reg_Th CP exact reg}
Consider two real numbers $r>0$,
 $s \in]3+ \frac{1}{2},+\infty[$ and $(\eta_0,\psi_0)\in H^{s+\frac{1}{2}}(\t) \times H^{s}(\t) $ such that, 
 $$\forall (\eta'_0,\psi'_0) \in \b((\eta_0,\psi_0),r)\subset H^{s+\frac{1}{2}}(\t) \times H^{s}(\t)$$
  the assumption $\eqref{pos reg_intro_WW_asump dom_$H_t$}_{t=0}$ is satisfied. Define $(\eta,\psi)$ and $(\eta',\psi')$ as the solutions to the Cauchy problem \eqref{pos reg_intro_WW_equations_WW syst} on $[0,T], \ T>0$.
  Define the following change of variables:
  \begin{equation}\label{pos reg_intro_WW_flow map reg_Th CP exact reg_renorm flow eq}
  \chi(t,x)=\int\limits_0^x \frac{1}{\sqrt{1+(\partial_x \eta(t,y))^2}}dy-\int\limits_0^t\int\limits_{\Sigma_s} \bigg[\frac{1}{1+(\partial_x \eta(s,y))^2}+\partial_x \phi\bigg] d\Sigma_s 
  %&=\int\limits_0^x \frac{1}{\sqrt{1+(\partial_x \eta(s,y))^2}}dyds-\int\limits_0^t\int\limits_0^{2\pi} \bigg[\frac{1}{1+(\partial_x \eta(s,y))^2}+\partial_x \phi\bigg] \sqrt{1+(\partial_x \eta(s,y))^2} dsdy,\nonumber
  \end{equation}
where $d\Sigma_t$ is the surface measure on $\Sigma_t$ and $\chi'$ is defined analogously from $(\eta',\psi')$. Then for $r$ sufficiently small and $t\in [0,T]$ we have:
\begin{align}
&\norm{(\eta,\psi)^*-(\eta',\psi')^{*'}(t,\cdot)}_{H^s\times H^{s-\frac{1}{2}}}\nonumber\\
&\leq C(\norm{(\eta_0,\psi_0,\eta'_0,\psi'_0)}_{H^{s+\frac{1}{2}} \times H^{s}})
\norm{(\eta_0,\psi_0)^*-(\eta'_0,\psi'_0)^{*'}}_{H^s\times H^{s-\frac{1}{2}}},
\end{align}	
where ${}^*$ and $^{*'}$ are the paracomposition by $\chi$ and $\chi'$, which we recall it's definition in \ref{paracomposition_section Paracomposition}.
\end{theorem}
%\begin{remark}\label{rem:paracompostion}The paraproduct is a way to decompose the product of two functions $uv$ into 3 terms: $T_uv$ whose Sobolev regularity is given by $v$ if $u$ is bounded, symmetrically the term $T_v u$ and a residual term $R(u,v)$ whose Sobolev regularity is the sum of that of $u$ and $v$ minus $\frac{d}{2}$. The paracomposition operation is the analogue for the composition of two functions $u\circ \chi $ into three terms: $\chi^*u$ whose Sobolev regularity is given by $u$ if $\chi$ is Lipschitz continuous, $T_{Du\circ \chi}\chi$ and a residual term that possess higher regularity.\end{remark}

 The time integral in the re-normalization \eqref{pos reg_intro_WW_flow map reg_Th CP exact reg_renorm flow eq} is to ensure that the mean value of the transport term vanishes. This re-normalization is used here to compensate the non-linearity in the dispersive term $T_\gamma$ of order $\frac{3}{2}$ .

\subsection{Strategy of the proof}
For Theorem \ref{pos reg_intro_th CP DB exact flow reg}, we first work on $H^s_0$ and the main idea is to conjugate \eqref{pos reg_intro_DBeq} to a semi-linear dispersive equation of the form:
\[\partial_tw+\abs{D}^{\alpha-1}\partial_xw=Ru, \]
where $R$ is continuous from $H^s$ to itself. For the viscous Burgers equation such a result is obtained by the Cole-Hopf transformation that reduces the problem to a one dimensional heat equation. In \cite{Tao04}, T.Tao used a complex version of the Cole-Hopf transformation to reduce the problem on the Benjamin-Ono equation to a one dimensional Schr\"odinger type equation, this idea was extensively used to lower the regularity needed for the well-posedness of the Cauchy problem as in Molinet's work in \cite{Molinet08}. A generalized pseudodifferential form of this transformation was used in \cite{Alazard15} to reduce the one dimensional water waves system to a one dimensional semi-linear Schr\"odinger type system.\\

Formally if we follow the same lines of those previous papers, the transformation we will have to use is a pseudodifferential transformation of the form:
\begin{equation}\label{pos reg_intro_sketch of the proof_gauge}
\begin{cases}
w=\op(a)u,\\
a=e^{\frac{1}{i\alpha}\xi \abs{\xi}^{1-\alpha}U},
\end{cases}
\end{equation}
where $U$ is a real valued periodic primitive of $u$ that exists because $u$ has mean value $0$ and $\op(a)$ is the pseudo-differential operator wit symbol $a$. 

The main problem is that such an operator belongs to a H\"ormander symbol class of the form $S^0_{\alpha-1,2-\alpha}$, see Remark \ref{rem:hormander symbold classes} for a formal definition, for $\alpha=\frac{3}{2}$ this becomes $S^0_{\frac{1}{2},\frac{1}{2}}$ which is a ``bad" symbol class with no general symbolic calculus rules. Thus we have to treat this transformation with care. 

The idea here is inspired by the particular form of the formal computation, we express the desired operator as the time one of a flow map associated to a hyperbolic equation, that is $a=e^{iT_p}$ where $(e^{i\tau T_p})_{\tau \in \r}$ is defined as the group generated by the paradifferential operator $iT_p$ where $p$ is a real valued symbol of order smaller than $1$. This is inspired by previous results of Alazard, Baldi and P.G\'erard \cite{Gerard10}.

Take a different operator $T_b$. The main new idea is to apply a Baker-Campbell-Hausdorff formula. Formally this allows one to express $e^{i\tau T_p}T_b e^{-i\tau T_p}$ as a series of successive Lie derivatives $[iT_{p},\cdots,[iT_{p},T_b]$. The same kind of computations go for $[e^{i\tau T_p},T_b]$. The convergence of such a series is a non trivial problem, equivalent to solving a linear ODE in the Fr\'echet space of paradifferential symbol classes $\Gamma^m_{+\infty}$ defined in Appendix \ref{paracomposition_Notions of microlocal analysis_Paradifferential Calculus}. Such an ODE is generally not well posed and to solve such a problem one usually has to look at a Nash-Moser type scheme. Though in our case we have an explicit ODE that can be solved locally with loss of derivative inspired by H\"ormander's \cite{Hormander90} and Beals in \cite{Beals 77}, we prove the existence of a symbol $b^\tau_p$ such that $e^{i\tau T_p}T_b e^{-i\tau T_p}=\op(b^\tau_p)$, moreover $b^\tau_p$ is shown to have an asymptotic expansion given by the Baker-Campbell-Hausdorff formula.
The use of paradifferential operators is the key here, as in H\"ormander's \cite{Hormander90}, because the continuity of paradifferential operators given by Theorem \ref{paracomposition_Notions of microlocal analysis_Paradifferential Calculus} insures that we do not need to control an infinite number of semi-norms as would have been the case for pseudodifferential operators.

Finally the adequate choice of $T_p$ eliminates the transport term of order 1 and gives a term of order $2-\alpha$ which is enough to get the desired estimate.

Passing from $H^s_0$ to $H^s$ we use the following gauge transform: 
\[\tilde{u}(t,x)=u(t,x-t\dashint u_0)-\dashint u_0, \text{ where }\dashint u_0=\frac{1}{\abs{\t}}\int\limits_\t u_0,\]
which we will prove is continuous on $H^s$ but not uniformly continuous and $C^1$ only from $H^s$ to $H^{s-1}$.

For the Gravity-Capillary equation the problem is more delicate. Indeed the model problems we study are for the paralinearized and symmetrized system, though the change of variable from the original system to the paralinearized and symmetrized one is known to be Lipschitz on $H^s$ for $s>2+\frac{1}{2}$ (\cite{Alazard16,Alazard11}). Thus the problem is reduced to the study of the flow map regularity of an equation of the form
\[
\partial_tu+ T_V\cdot \partial_x u+i T_{\gamma}u=f.
\]
In the same spirit as \cite{Alazard15,Alazard18} we perform a para-change of variable, that is we  para-compose with $\chi$ defined by \eqref{pos reg_intro_WW_flow map reg_Th CP exact reg_renorm flow eq}, to get:
\[
\partial_t[\chi^*u]+ T_{W}\cdot\partial_x \left(\chi^*u\right)+i T_{\abs{\xi}^{\frac{3}{2}}}\chi^*u=f, \text{ with }  W=\frac{V\circ \chi}{\partial_x \chi} \text{ and } \int\limits_\t W=0.
\]
We then proceed exactly as for Equation \eqref{pos reg_intro_DBeq} (with the $0$ mean value hypothesis insured by the choice of $\chi$).

\begin{itemize}
\item Transformation \eqref{pos reg_intro_sketch of the proof_gauge}, in which we use a primitive of the solution is called a gauge transform in the literature. 

\item As for the Cole-Hopf transformation, this gauge transform \eqref{pos reg_intro_sketch of the proof_gauge} is essentially one dimensional. 

\item It is interesting to note that the gauge transformation can be iterated to eliminate the term of order $2-\alpha$ and get at the step of order k a remainder of order $k+1-k \alpha$ which is an improvement at each step as $\alpha>1$. Choosing $k=\lceil \frac{1}{1-\alpha} \rceil$, we get a residual term that is bounded from $H^s$ to $H^s$ when one pays the ``price" of working in high enough regularity that is $s>1+\frac{1}{\alpha-1}$. In \cite{Ayman20'} we use this iteration to prove that for $2<\alpha<3$,  the paradifferential version of \eqref{pos reg_intro_DBeq} can be transformed to a semi-linear equation with a regularizing remainder, that is:
\[
\partial_t u+T_{u}\partial_x u+\partial_x\abs{D}^{\alpha-1}u=0\Rightarrow \partial_t Au+\partial_x\abs{D}^{\alpha-1}Au=R(u),
\]
where the operator norm of $R(u)$ is controlled by $\norm{u}_{L^\infty([0,T], C_*^{2-\alpha})}$.

\end{itemize}

\subsection{Acknowledgement}
I would like to express my sincere gratitude to my thesis advisor Thomas Alazard. I would also like to thank the referees for their valuable input that greatly improved the manuscript.
\section{Study of the model problems}
\subsection{Proof of Theorem \ref{pos reg_intro_th CP DB exact flow reg}, the estimates on $H^s_0$}\label{pos reg_study of mod pbm_subsec proof thm DB exact flow reg}

We keep the notations of Theorem \ref{pos reg_intro_th CP DB exact flow reg}, fixing $u_0\in H^s_0(\t;\r)$ and $r>0$ and taking: $$v_0,w_0 \in \b(u_0,r) \subset H^s_0(\t;\r).$$ As the mean value is conserved by the flow of \eqref{pos reg_intro_th standard CP DB_eq mod on t} we consider the solutions $u,v,w \in C^0([0,T];H^s_0(\t;\r))$ to \eqref{pos reg_intro_th standard CP DB_eq mod on t} with initial data $u_0,v_0,w_0$ and on a uniform small interval $[0,T]$.

 The main goal of the proof is to show the following estimate:
\begin{equation} \label{pos reg_study of mod pbm_subsec proof thm DB exact flow reg_key diff est}
\norm{v(t,\cdot)-w(t,\cdot)}_{H^{s-(2-\alpha)^+}}\leq e^{Cte^{Ct\norm{(v,w)}_{L^\infty_tW^{\lceil \frac{\alpha}{\alpha-1}\rceil-1,\infty}_x}}\norm{w_0}_{H^{s}}}\norm{v_0-w_0}_{H^{s-(2-\alpha)^+}}.
\end{equation}
%with the following tame control,\begin{equation}\label{pos reg_study of mod pbm_subsec proof thm DB exact flow reg_key diff est tame}C(\norm{(v_0,w_0)}_{H^s})\leq C'(\norm{(v_0,w_0)}_{H^{s-(2-\alpha)^+}})[\norm{(v_0,w_0)}_{H^s}+1],\end{equation}where $C$ and $C'$ are non decreasing positive functions.

The final simplification we make in this paragraph is that given the well-posedness of the Cauchy problem in $H^s$, and the density of $H^{+\infty}$ in $H^s$, it suffice to prove \eqref{pos reg_study of mod pbm_subsec proof thm DB exact flow reg_key diff est} for $v_0,w_0 \in H^{+\infty}$, which henceforth we will suppose. 

We start by applying the paralinearization Theorem \ref{paracomposition_Notions of microlocal analysis_Paradifferential Calculus_symbolic calculus para precised} to the term $u\partial_x v$ to get: 
\begin{equation}\label{pos reg_study of mod pbm_subsec proof thm DB exact flow reg_paralinearization DB}
\begin{cases} 
\partial_t v+T_{vi\xi}v+T_{i\abs{\xi}^{\alpha-1}\xi} v=R_0(v)v, \\
v(0,\cdot)=v_0(\cdot),
\end{cases}
\end{equation}
where
\[
R_0(v)\cdot =-T_{\partial_x v}\cdot -R(v,\partial_x \cdot) +\left(T_{i\xi \abs{\xi}^{\alpha-1}}-\partial_x\abs{D}^{\alpha-1}\right)\cdot.
\]
Now we reduce $H^{s-(2-\alpha)^+}$ estimates to $L^2$ ones by defining $f_1=\D^{s-(2-\alpha)^+} v$. Commuting $\D^{s-(2-\alpha)^+}$ with \eqref{pos reg_study of mod pbm_subsec proof thm DB exact flow reg_paralinearization DB}, using the symbolic calculus rules of Theorem \ref{paracomposition_Notions of microlocal analysis_Paradifferential Calculus_symbolic calculus para precised}, we get that:
\begin{equation}\label{pos reg_study of mod pbm_subsec proof thm DB exact flow reg_redc to L^2  para DB}
\begin{cases} 
\partial_t f_1+T_{vi\xi}f_1+T_{i\abs{\xi}^{\alpha-1}\xi} f_1=R_1(v)f_1 \\
f_1(0,\cdot)=\D^{s-(2-\alpha)^+}v_0(\cdot),
\end{cases}
\end{equation}
where 
\[
R_1(v)\cdot=\left[\D^{s-(2-\alpha)^+},T_{vi\xi}\right]\cdot+\D^{s-(2-\alpha)^+}R_0(v)\cdot.
\]
We define analogously $g_1=\D^{s-(2-\alpha)^+}w$ and notice that by definition:
\[\norm{f_1-g_1}_{L^2}=\norm{v-w}_{H^{s-(2-\alpha)^+}}, \] 
thus the problem is reduced to getting $L^2$ estimates on $f_1-g_1$.

Here we give the full proof using estimates that will be proved in Section \ref{BCH formula}.
\subsubsection{Gauge transform and Energy estimate}\label{pos reg_study of mod pbm_subsec proof thm DB exact flow reg_subsec gauge trfm}
The goal of this section is to find an operator $A$ such that
\[
\partial_t [A f_1]+T_{i\abs{\xi}^{\alpha-1}\xi} Af_1+AT_{vi\xi}f_1+[A,T_{i\abs{\xi}^{\alpha-1}\xi}]f_1=(\partial_t A) f_1+A R_1(f_1)f_1, 
\]
and $AT_{vi\xi}+[A,T_{i\abs{\xi}^{\alpha-1}\xi}]$ is a hyperbolic operator of order $(2-\alpha)^+<1$.\\
If we define $V=\partial_x^{-1}v$ which is the periodic zero mean value primitive of $v$, then
\[\hat{V}(0)=0 \text{ and } \hat{V}(\xi)=\frac{\hat{v}(\xi)}{i\xi}, \text{ for }  \xi \in \z^*,\]
and we define analogously $W$ from $w$. Then a formal computation shows that one can choose $A=T_{e^{\frac{1}{i\alpha}\xi \abs{\xi}^{1-\alpha}V}}\in S^0_{\alpha-1,2-\alpha}(\t\times \z)$ which is a symbol class with no general symbolic calculus rules. Here we will define $A$ differently \footnote{Similar ideas were used in Appendix C of \cite{Alazard15} to get estimates on a change of variable operator which are still in the usual symbol classes $S^m_{1,0}$, the difficulty here being that we are no longer in those symbol classes.}.

$A=e^{iT_{p_v}}$, that is, it is  defined as the time one of the flow map generated by $T_{ip_v}$ with  
$$p_v=-\frac{1}{\alpha}\xi \abs{\xi}^{1-\alpha}V\in \Gamma_{\lceil \frac{\alpha}{\alpha-1}\rceil}^{2-\alpha}(\t),$$
which is well defined by Proposition \ref{BCH formula_def para hyperbolic flow}. We define analogously $e^{iT_{p_w}}$ and $p_w$ from $w$.
Now introduce:
\begin{align}\label{pos reg_study of mod pbm_subsec proof thm DB exact flow reg_subsec gauge trfm_eq1}
f_2=e^{iT_{p_v}}f_1,\ \
g_2=e^{iT_{p_w}}g_1.
\end{align}
%The study of the symbolic calculus associated to this very specific form of symbols is given by Proposition \ref{BCH formula_def para hyperbolic flow} and the change of variable \eqref{pos reg_study of mod pbm_subsec proof thm DB exact flow reg_subsec gauge trfm_eq1} is Lipschitz from $L^2$ to $L^2$ but under $H^{(2-\alpha)^+}$ control on $(f_2,g_2)$. 
As $e^{-iT_{p_v}}$ and $e^{-iT_{p_w}}$ are the time $-1$ generated by the flow map $p_v,p_w$ respectively we write:
\begin{align}\label{eq:proof model esimate Gauge}
\norm{f_1-g_1}_{L^2}&=\norm{e^{-iT_{p_v}}f_2-e^{-iT_{p_w}} g_2}_{L^2}\nonumber\\
&\leq \norm{e^{-iT_{p_v}}[f_2-g_2]}_{L^2}+\norm{(e^{-iT_{p_v}}-e^{-iT_{p_w}})g_2}_{L^2}.\nonumber
\intertext{Applying estimate $(1)$ of Proposition \ref{BCH formula_def para hyperbolic flow} and estimate \eqref{BCH formula_def para hyperbolic flow_est diff flow}:}
\norm{f_1-g_1}_{L^2}&\leq e^{C\norm{v}_{L^\infty}} \norm{f_2-g_2}_{L^2}
+e^{C\norm{(v,w)}_{L^\infty}}\norm{V-W}_{L^\infty}\norm{g_2}_{H^{(2-\alpha)^+}}.
\end{align}
%where $C$ verifies the estimate \ref{pos reg_study of mod pbm_subsec proof thm DB exact flow reg_key diff est tame}. As $e^{-iT_{p_v}}$ and $e^{-iT_{p_w}}$ are the time $-1$ generated by the flow map $p_v,p_w$ respectively which is well defined by Proposition \ref{BCH formula_def para hyperbolic flow}. We get by symmetry:\[\norm{f_1-g_1}_{L^2} \leq C(  \norm{(v,w)}_{H^s})\norm{f_2-g_2}_{L^2},\]thus,\[C^{-1}(  \norm{(v,w)}_{H^s})\norm{f_1-g_1}_{L^2}\leq\norm{f_2-g_2}_{L^2}\leq C(  \norm{(v,w)}_{H^s})\norm{f_1-g_1}_{L^2}, \]
The goal now is getting an $L^2$ estimates on $f_2-g_2$. To get the equations on $f_2$ and $g_2$ we commute $e^{iT_{p_v}}$ and $e^{iT_{p_w}}$ with \eqref{pos reg_study of mod pbm_subsec proof thm DB exact flow reg_redc to L^2  para DB}, we make the computations for $f_2$, those for $g_2$ are obtained by symmetry:
\[
e^{iT_{p_v}}\partial_t f_1+e^{iT_{p_v}}T_{vi\xi}f_1+e^{iT_{p_v}}T_{i\abs{\xi}^{\alpha-1}\xi}f_1=e^{iT_{p_v}} R_1(v)f_1, \text{ thus, }
\]
\[
\partial_t \big(e^{iT_{p_v}}f_1\big)+T_{i\abs{\xi}^{\alpha-1}\xi}e^{iT_{p_v}}f_1+(e^{iT_{p_v}}T_{vi\xi}-[T_{i\abs{\xi}^{\alpha-1}\xi},e^{iT_{p_v}}])f_1+[e^{iT_{p_v}},\partial_t]f_1=e^{iT_{p_v}} R_1(v)f_1.
\]
By definition of $p_v$ and Proposition \ref{BCH formula_prop diff Atau param} we have:
$$\partial_\xi(\xi \abs{\xi}^{\alpha-1})\partial_x p_v
	=v\xi \text{ and } [e^{iT_{p_v}},\partial_t] =-e^{iT_{p_v}} \int\limits_0^1 e^{-irT_{p_v}} T_{i\partial_t p_v} e^{irT_{p_v}}dr.$$
Thus
\begin{equation}\label{Eq-aux1}
\partial_t f_2+T_{i\abs{\xi}^{\alpha-1}\xi}f_2
=R_2(v) f_2+e^{iT_{p_v}} R_1(v)e^{-iT_{p_v}}f_2,
\end{equation}
where 
\[
R_2(v)\cdot=-(e^{iT_{p_v}}T_{vi\xi}-[T_{i\abs{\xi}^{\alpha-1}\xi},e^{iT_{p_v}}])e^{-iT_{p_v}}\cdot+[e^{iT_{p_v}},\partial_t]e^{-iT_{p_v}}\cdot.
\]
In Corollary \ref{BCH formula_cor reste gaue trans} we show that we have the estimates 
\[
\norm{\RE(R_2(v))}_{L^2 \rightarrow L^2}\leq e^{\norm{v}_{L^\infty}}\norm{v}_{W^{1,\infty}},\]
\[
\norm{ [R_2(v)-R_2(w)]g_2}_{L^2}\leq  e^{C\norm{(v,w)}_{W^{\lceil \frac{\alpha}{\alpha-1}\rceil-1,\infty}_x}} \norm{v-w}_{W^{\lceil \frac{\alpha}{\alpha-1}\rceil-1,\infty}_x} \norm{g_2}_{H^{(2-\alpha)^+}},
\]
and
\begin{multline*}
\norm{ [e^{iT_{p_v}} R_1(v)e^{-iT_{p_v}}-e^{iT_{p_w}} R_1(w)e^{-iT_{p_v}}]g_2}_{L^2}\\
\leq e^{C\norm{(v,w)}_{W^{\lceil \frac{\alpha}{\alpha-1}\rceil-1,\infty}_x}} \norm{v-w}_{W^{\lceil \frac{\alpha}{\alpha-1}\rceil-1,\infty}_x} \norm{g_2}_{H^{(2-\alpha)^+}}.
\end{multline*}
We get analogously on $g_2$,
\begin{equation}\label{pos reg_study of mod pbm_subsec proof thm DB exact flow reg_subsec gauge trfm_eq3}
\partial_t g_2+T_{i\abs{\xi}^{\alpha-1}\xi}g_2
=R_2(w) g_2+e^{iT_{p_w}} R_1(w)e^{-iT_{p_w}}g_2.
\end{equation}
Taking the difference between \eqref{Eq-aux1} and \eqref{pos reg_study of mod pbm_subsec proof thm DB exact flow reg_subsec gauge trfm_eq3}
%\begin{multline*}\partial_t (f_2-g_2)+T_{i\abs{\xi}^{\alpha-1}\xi}(f_2-g_2)=[R_2(w)-R_2(v)-e^{iT_{p_v}} R_1(v)e^{-iT_{p_v}}+e^{iT_{p_w}} R_1(w)e^{-iT_{p_w}}]g_2.\end{multline*}
and preforming an energy estimate gives for $0\leq t\leq T$:
\begin{align*}
\norm{(f_2-g_2)(t,\cdot)}_{L^2}&\leq e^{C\norm{e^{C\norm{(v,w)}_{W^{\lceil \frac{\alpha}{\alpha-1}\rceil-1,\infty}_x}}\norm{g_2}_{H^{(2-\alpha)^+}}}_{L^1_t}}\norm{(f_2-g_2)(0,\cdot)}_{L^2}\\
&\leq e^{Cte^{Ct\norm{(v,w)}_{L^\infty_tW^{\lceil \frac{\alpha}{\alpha-1}\rceil-1,\infty}_x}}\norm{w_0}_{H^{s}}}\norm{(f_2-g_2)(0,\cdot)}_{L^2},
\end{align*}
which injected back in \eqref{eq:proof model esimate Gauge} concludes the proof.
\subsection{Proof of Corollary \ref{pos reg_intro_cor CP DB exact flow reg}, the estimates on $H^s$}\label{pos reg_study of mod pbm_subsec proof cor DB exact flow reg}
The starting point is noticing that the mean value is preserved by \eqref{pos reg_intro_th standard CP DB_eq mod on t} and by doing the change of unknowns:
\begin{equation}
\begin{cases}\label{pos reg_study of mod pbm_subsec proof cor DB exact flow reg_low freq trans}
\tilde{u}(t,x)=u(t,x-t\dashint u_0)-\dashint u_0\\
\tilde{v}(t,x)=v(t,x-t\dashint v_0)-\dashint v_0
\end{cases},
\end{equation}
where $\dashint u_0=\frac{1}{2\pi}\int\limits_\t u_0$ is the mean value. We can reduce the Cauchy problem for general data to ones with $0$ mean value by verifying that $\tilde{u},\tilde{v} \in H_0^s$ still solve \eqref{pos reg_intro_th standard CP DB_eq mod on t}. Thus the main goal is to prove that the change of variable \eqref{pos reg_study of mod pbm_subsec proof cor DB exact flow reg_low freq trans} is not regular. More precisely we will show that there exists a positive constant $C$ and two sequences $(u^\lambda_{\eps})$ and $(v^\lambda_{\eps})$ solutions of \ref{pos reg_intro_th standard CP DB_eq mod on t} in $C^0([0,1],H^s(\t))$ such that for every $t\leq T$, where $T$ is a uniform small time,
\[
\sup_{\lambda,\eps} \norm{u^\lambda_{\eps}}_{H^s(\t)(t,\cdot)}+
\norm{v^\lambda_{\eps}(t,\cdot)}_{H^s(\t)}\leq C,
\]
$(u^\lambda_{\eps,\tau})$ and $(v^\lambda_{\eps,\tau})$ satisfy initially:
\[
\lim_{\substack{\lambda \rightarrow +\infty \\ \eps\rightarrow 0}} \norm{u^\lambda_{\eps}(0,\cdot)-v^\lambda_{\eps}(0,\cdot)}_{H^s(\t)}=0,
\]
but,
\[
\liminf_{\substack{\lambda \rightarrow +\infty \\ \eps\rightarrow 0}} \norm{u^\lambda_{\eps}(t,\cdot)-v^\lambda_{\eps}(t,\cdot)}_{H^s(\t)}\geq c>0.
\]
Which proves the non uniform continuity.
 Considering a weaker control norm we want to get, for all $\delta>0$ and for $t>0$:
 \[
\liminf_{\substack{\lambda \rightarrow +\infty \\ \eps \rightarrow 0}} \frac{\norm{u^\lambda_{\eps}(t,\cdot)-v^\lambda_{\eps}(t,\cdot)}_{H^{s-1+\delta}(\t)}}{\norm{u^\lambda_{\eps}(0,\cdot)-v^\lambda_{\eps}(0,\cdot)}_{H^s(\t)}}=+\infty.
\]
\subsubsection{Definition of the Ansatz}
Take $\omega \in C_0^{\infty}(\t)$ such that for $x\in[0,2\pi]$: $$\omega(x)=1  \text{ if } \abs{x}\leq \frac{1}{2}, \ \omega(x)=0 \text{ if } \abs{x} \geq 1.$$
 Let $(\lambda ,\eps)$ be two positive real sequences such that:
\begin{equation} \label{pos reg_study of mod pbm_subsec proof cor DB exact flow reg_def ansarz_eq1} %%%%%%%%%ref%%%%%%%%%%
  \lambda  \rightarrow + \infty, \  \eps  \rightarrow  0, \ \lambda \eps  \rightarrow  + \infty.
\end{equation}
Put for $x\in[0,2\pi]$, 
\begin{equation}\label{eq:addhoc reviewer}
u_0(x)=\lambda ^{\frac{1}{2}-s}\omega(\lambda x), \ v _0(x)=u_0(x)+ \eps  \omega(x), 
\end{equation}
and extend $u^0$ and $v^0$ periodically.
The main trick here will be to use the time reversibility of equation \eqref{pos reg_intro_th standard CP DB_eq mod on t} by defining $\tilde{u},\tilde{v}$ as the solution of \eqref{pos reg_intro_th standard CP DB_eq mod on t} with data fixed at time $t>0$ given by
\begin{equation}
\begin{cases}
\tilde{u}(t,x)=u_0-\dashint u_0\\
\tilde{v}(t,x)=v_0-\dashint v_0
\end{cases},
\end{equation}
 where $t\leq t_0$ is chosen small enough for the equations to be well-posed. Finally, define $u$ and $v$ by \eqref{pos reg_study of mod pbm_subsec proof cor DB exact flow reg_low freq trans}.
\subsubsection{Main estimates}
First the estimates at time $0$, for $\lceil\frac{\alpha}{\alpha-1}\rceil-\frac{3}{2}\leq s-1\leq \nu \leq s$:
\begin{align}
\norm{u(0,x)-v(0,x)}_{H^\nu}&=\norm{\tilde{u}(0,x)-\tilde{v}(0,x)+\dashint u_0-\dashint v_0}_{H^\nu}\nonumber
\intertext{By the estimate \eqref{pos reg_study of mod pbm_subsec proof thm DB exact flow reg_key diff est} and the Cauchy-Schwartz inequality,}
\norm{u(0,x)-v(0,x)}_{H^\nu}
&\leq Ce^{C\lambda^{\nu-s+(2-\alpha)^+}}\eps.
\end{align}
Now the estimates at a fixed time $t>0$, by construction:
\begin{align*}
\norm{u(t,x)-v(t,x)}_{H^\nu}&=\norm{u_0(x+t\dashint u_0)-v_0(x+t\dashint v_0)}_{H^\nu}\\
&=\norm{u_0(x+t\dashint u_0)-u_0(x+t\dashint v_0)}_{H^\nu} +O_{H^\nu}(\eps)
\end{align*}
Now by hypothesis $\lambda\eps \rightarrow+\infty$ and $ t\dashint \omega>0$, thus $u_0(\cdot+t\dashint u_0)$ and $u_0(\cdot+t\dashint v_0)$ have disjoint supports, thus
\begin{align}
\norm{u(t,x)-v(t,x)}_{H^\nu}&= \norm{u_0(x+t\dashint u_0)}_{H^\nu}+\norm{u_0(x+t\dashint v_0)}_{H^\nu} +O_{H^\nu}(\eps)\nonumber\\
&=C\lambda^{\nu-s}+O_{H^\nu}(\eps).
\end{align}

Now to conclude the proof we differentiate the cases:
\begin{itemize}
\item in the case of non uniform continuity we take $\eps$ such that $\eps e^{C\lambda^{(2-\alpha)^+}}\rightarrow 0$ and apply the previous estimates with $\nu=s$.
 \item In the case of non Lipschitz control we take $\eps$ such that $\lambda^{-1+\delta}\eps^{-1}\rightarrow +\infty$ and apply the previous estimates with $\nu=s-1+\delta$.
 \end{itemize}
\section{Flow map regularity for the periodic gravity capillary equation}
\subsection{Prerequisites from the Cauchy problem}
We start by recalling the a priori estimates given by Proposition $5.2$ of \cite{Alazard11} combined with the results of \cite{Alazard16}. We keep the notations of Theorem \ref{pos reg_intro_WW_flow map reg_Th CP exact reg}.
\begin{proposition} \label{pos reg_flow GWW_CP Alazard} (From \cite{Alazard11} and \cite{Alazard16}) Consider a real number $s> 2+\frac{1}{2}$. Then there exists a non decreasing function C such that, for all $T \in ]0,1]$ and all solution $(\eta,\psi)$ of \eqref{pos reg_intro_WW_equations_WW syst} such that:
\[(\eta,\psi) \in C^0([0,T];H^{s+\frac{1}{2}}(\t)\times H^{s}(\t) ) \text{ and \eqref{pos reg_intro_WW_asump dom_$H_t$} is verified for $t\in [0,T]$,}\]  
we have: \[\norm{(\eta,\psi)}_{L^\infty(0,T;H^{s+\frac{1}{2}}\times H^{s})} \leq C({(\eta_0,\psi_0)}_{H^{s+\frac{1}{2}}\times H^{s}})+TC(\norm{(\eta,\psi)}_{L^\infty(0,T;H^{s+\frac{1}{2}}\times H^{s})}). \]
\end{proposition}
The proof will relies on the para-linearized and symmetrized version of \eqref{pos reg_intro_WW_equations_WW syst} given by Proposition $4.8$ and corollary $4.9$ of \cite{Alazard11} which are valid on $\t$ as shown in \cite{Alazard16}. Before we recall this, for clarity as in \cite{Alazard11}, we introduce a special class of operators $\Sigma^m \subset \Gamma^m_0$ given by:
\begin{definition}(From \cite[\S 4]{Alazard11})
Given $m \in \r$, $\Sigma^m$ denotes the class of symbols $a$ of the form
\[a=a^{(m)}+a^{(m-1)},\]
with,
\[a^{(m)}=F(\partial_x \eta (t,x),\xi),\]
\[a^{(m-1)}=\sum_{\abs{k}=2}G_\alpha(\partial_x \eta (t,x),\xi)\partial^k_x \eta(t,x),\]
such that 
\begin{enumerate}
\item $T_a$ maps real valued functions to real-valued functions;
\item F is of class $C^\infty$ real valued function of $(\zeta,\xi) \in \r \times (\z \setminus 0),$ homogeneous of order m in $\xi$; and such that there exists a continuous function $K=K(\zeta)>0$ such that
\[F(\zeta,\xi)\geq K(\zeta)\abs{\xi}^m,\]   
for all $(\zeta,\xi)\in \r\times (\z \setminus 0)$; 
\item $G_\alpha$ is a $C^\infty$ complex valued function of $(\zeta,\xi)\in \r\times (\z \setminus 0)$, homogeneous of order $m-1$ in $\xi$.
\end{enumerate}
\end{definition}
$\Sigma^m$ enjoys all the usual symbolic calculus properties in the sense of Proposition \ref{paracomposition_Notions of microlocal analysis_Paradifferential Calculus_symbolic calculus para precised} modulo acceptable remainders that we define by the following:
\begin{definition-notation}(From \cite[Def 4.2]{Alazard11})
Let $m \in \r$ and consider two families of operators of order m,
\[\set{A(t): t \in [0,T]}, \ \ \ \set{B(t):t \in [0,T]}.\]
We shall say that $A \sim B$ if $A-B$ is of order $m -\frac{3}{2}$ and satisfies the following estimate: for all $\mu \in \r$, there exists a continuous function C such that for all $t \in [0,T]$,
\[\norm{A(t)-B(t)}_{H^{\mu}\rightarrow H^{\mu -m+\frac{3}{2}}} \leq C(\norm{\eta(t)}_{H^{s+\frac{1}{2}}}).\]
\end{definition-notation}
In the next proposition we recall the different symbols that appear in the para-linearization and symmetrization of the water waves equations.
\begin{proposition}\label{pos reg_flow GWW_prop para Alazard}(From \cite[\S 4.2]{Alazard11})%%%%%%%%%%%%ref%%%%%%%%%
We work under the hypothesis of Proposition \ref{pos reg_flow GWW_CP Alazard}.
Put 
\[\lambda=\lambda^{(1)}+\lambda^{(0)}, \  l=l^{(2)}+l^{(1)} \text{ with,}\]
\begin{align}\label{pos reg_flow GWW_prop para Alazard_eq1}%%%%%%%%%%%%ref%%%%%%%%%
&\begin{cases}
\lambda^{(1)}=\abs{\xi},\\
\lambda^{(0)}=\frac{1+|\partial_x\eta|^2}{2\abs{\xi}}\set{\partial_x \bigg(\alpha^{(1)}\partial_x \eta\bigg)+i\frac{\xi}{\abs{\xi}} \partial_x \alpha^{(1)} },\\
\alpha^{(1)}=\frac{1}{\sqrt{1+|\partial_x \eta|^2}}\bigg( \abs{\xi}+i\partial_x \eta  \xi \bigg).
\end{cases}
\\
&\begin{cases}
l^{(2)}=(1+|\partial_x \eta|^2)^{-\frac{3}{2}}\xi^2,\\
l^{(1)}=-\frac{i}{2}(\partial_x \cdot \partial_\xi)l^{(2)}.\\
\end{cases}
\end{align}
Now let $q\in \Sigma^0, p \in \Sigma^{\frac{1}{2}}, \gamma \in \Sigma^{\frac{3}{2}}$ be defined by
\begin{align*}
q&=(1+|\partial_x \eta|^2)^{-\frac{1}{2}},\\
p&=\underbrace{(1+|\partial_x \eta|^2)^{-\frac{5}{4}}\abs{\xi}^{\frac{1}{2}}}_{:=p^{(\frac{1}{2})}}+p^{(-\frac{1}{2})},\\
\gamma&=\underbrace{\sqrt{l^{(2)}\lambda^{(1)}}}_{:=\gamma^{(\frac{3}{2})}}+\underbrace{\sqrt{\frac{l^{(2)}}{\lambda^{(1)}}}\frac{\RE \lambda^{(0)}}{2}}_{:=\gamma^{(\frac{1}{2})}}-\frac{i}{2}(\partial_\xi \cdot \partial_x)\sqrt{l^{(2)}\lambda^{(1)}},\\
p^{(-\frac{1}{2})}&=\frac{1}{\gamma^{(\frac{3}{2})}}\set{ql^{(1)}-\gamma^{(\frac{1}{2})}p^{(\frac{1}{2})}+i\partial_\xi \gamma^{(\frac{3}{2})}\cdot \partial_x p^{(\frac{1}{2})}}.
\end{align*}
Then
\[T_qT_\lambda \sim T_\gamma T_q, \  T_q T_l \sim T_\gamma T_p, \ T_\gamma \sim (T_\gamma)^*, \]
where $(T_\gamma)^*$ is the adjoint of $T_\gamma$.
\end{proposition}
Now we can write the paralinearization and symmetrization of the equations \eqref{pos reg_intro_WW_equations_WW syst} after a change of variable:
\begin{corollary}\label{pos reg_flow GWW_cor para sys Alazard}(From \cite[Corollary 4.9]{Alazard11})%%%%%%%%%%%%ref%%%%%%%%%
Under the hypothesis of Proposition \ref{pos reg_flow GWW_CP Alazard}, introduce the unknowns\footnote{U is commonly called the "good" unknown of Alinhac. Introduced by Alazard-Metivier in \cite{Alazard09}, following earlier works by Lannes in \cite{{Lannes2005}}.}
\[U=\psi-T_B \eta, \ \Phi_1= T_p \eta  \text{ and }  \Phi_2=T_q U,\]
where we recall,
\[
\begin{cases}
B=(\partial_y \phi)_{|y=\eta}=\frac{\partial_x \eta \cdot \partial_x \psi +G(\eta) \psi}{1+(\partial_x \eta)^2},\\
V=(\partial_x \phi)_{|y=\eta}=\partial_x \psi -B\partial_x \eta.
\end{cases}\]
Then $\Phi_1, \Phi_2 \in C^0([0,T];H^s(\t))$ and 
\begin{equation}\label{pos reg_flow GWW_cor para sys Alazard_eq1}
\begin{cases}
\partial_t \Phi_1 +T_V \times \partial_x \Phi_1-T_\gamma \Phi_2=f_1,\\
\partial_t \Phi_2+ T_V \times \partial_x \Phi_2+ T_\gamma \Phi_1=f_2,
\end{cases}
\end{equation}
with $f_1,f_2 \in L^\infty(0,T;H^s(\t)),$ and $f_1,f_2$ have $C^1$ dependence on $(\Phi_1, \Phi_2)$ verifying:
\[
\norm{(f_1,f_2)}_{L^\infty(0,T;H^s(\t))} \leq C(\norm{(\eta,\psi)}_{L^\infty(0,T;H^{s+\frac{1}{2}}\times H^s(\t))}).
\]
\end{corollary}
\subsection{Proof of Theorem \ref{pos reg_intro_WW_flow map reg_Th CP exact reg}}
Corollary \ref{pos reg_flow GWW_cor para sys Alazard} shows that the paralinearization and symmetrization of the equations \eqref{pos reg_intro_WW_equations_WW syst} are of the form of the equations treated in Theorem \ref{pos reg_intro_th CP DB exact flow reg}, so the proof will follow the same lines but with more care in treating the non-linearity in the dispersive term.

We keep the notations of Theorem \ref{pos reg_intro_WW_flow map reg_Th CP exact reg}, fixing $(\eta_0,\psi_0)\in H^{s+\frac{1}{2}}(\t) \times H^{s}(\t)$ and $r>0$. We begin by taking $(\tilde{\eta}_0,\tilde{\psi}_0)\in \b((\eta_0,\psi_0),r) \subset H^{s+\frac{1}{2}}(\t) \times H^{s}(\t)$ and consider the solutions $(\eta,\psi),(\tilde{\eta},\tilde{\psi})  \in C^0([0,T];H^{s+\frac{1}{2}}(\t) \times H^{s}(\t))$ to \eqref{pos reg_intro_th standard CP DB_eq mod on t} with initial data $(\eta_0,\psi_0),(\tilde{\eta}_0,\tilde{\psi}_0)$, on a uniform small interval $[0,T]$ where the hypothesis \eqref{pos reg_intro_WW_asump dom_$H_t$} is also supposed to be verified. Define the following change of variables:
  \begin{align}\label{pos reg_flow GWW_proof_renorm flow def}
  \chi(t,x)&=\int\limits_0^x \frac{1}{\sqrt{1+(\partial_x \eta(t,y))^2}}dy-\int\limits_0^t\int\limits_{\Sigma_s} \bigg[\frac{1}{1+(\partial_x \eta(s,y))^2}+\partial_x \phi\bigg] d\Sigma_s ds \nonumber\\
  &=\int\limits_0^x \sqrt{1+(\partial_x \eta(t,y))^2}dy-\int\limits_0^t\int\limits_0^{2\pi} \bigg[\frac{1}{1+(\partial_x \eta)^2}+\partial_x \phi\bigg] \sqrt{1+(\partial_x \eta)^2} dyds
  \end{align}
and $\tilde{\chi}$ is defined analogously from $(\tilde{\eta},\tilde{\psi})$.

 The main goal of the proof is to show the following estimate:
\begin{align} \label{pos reg_flow GWW_proof_Diff key est}
&\norm{(\eta,\psi)^* (t,\cdot)-(\tilde{\eta},\tilde{\psi})^{\tilde{*}}(t,\cdot)}_{H^s\times H^{s-\frac{1}{2}}}\nonumber\\
&\leq C\bigg(\norm{(\eta_0,\psi_0,\tilde{\eta}_0,\tilde{\psi}_0)}_{H^{s+\frac{1}{2}}\times H^{s}}\bigg)\norm{(\eta_0,\psi_0)^*-(\tilde{\eta}_0,\tilde{\psi}_0)^{\tilde{*}}}_{H^s\times H^{s-\frac{1}{2}}},
\end{align}
where ${}^*$ and ${}^{\tilde{*}}$ are the paracomposition operators defined by $\chi$ and $\tilde{\chi}$ respectively. We recall that the definition of the paracomposition operator is given in Section \ref{paracomposition_section Paracomposition}.

Put $\Phi=(\Phi_1,\Phi_2)$ the unknowns obtained from  $(\eta,\psi)$ after paralinearization and symmetrization of the equations as in Corollary \ref{pos reg_flow GWW_cor para sys Alazard}. Define analogously $\tilde{\Phi}=(\tilde{\phi_1},\tilde{\phi_2})$  from $(\tilde{\eta},\tilde{\psi})$. Let us notice that, in order to prove \eqref{pos reg_flow GWW_proof_Diff key est}, it suffice to get estimates on $\Phi-\tilde{\Phi}$. Indeed we write 
\[
\begin{cases}
\Phi^*_1=T_{p^*}\eta^*\\
\Phi^*_2=T_{q^*}U^*
\end{cases} \text{ and } \begin{cases}
\eta^*=T_{\frac{1}{p^*}}\Phi^*_1-(T_{\frac{1}{p^*}}T_{p^*}-Id)\eta^*\\
U^*=T_{\frac{1}{q^*}}\Phi^*_2-(T_{\frac{1}{q^*}}T_{q^*}-Id)U^*
\end{cases}.
\]
then by the ellipticity of the symbols $p$ and $q$ combined with the immediate $L^2$ estimates( as $s>2+\frac{1}{2}$) we have:
\begin{align} \label{pos reg_flow GWW_proof_Diff est 2}
&\norm{(\eta,\psi)^* (t,\cdot)-(\tilde{\eta},\tilde{\psi})^{\tilde{*}}(t,\cdot)}_{H^s\times H^{s-\frac{1}{2}}}\nonumber\\
&\leq C\bigg(\norm{(\eta_0,\psi_0,\tilde{\eta}_0,\tilde{\psi}_0)}_{H^{s+\frac{1}{2}}\times H^{s}}\bigg)\norm{\Phi^*(t,\cdot)-\tilde{\Phi}^{\tilde{*}}(t,\cdot)}_{H^{s-\frac{1}{2}}\times H^{s-\frac{1}{2}}},\\
&\norm{\Phi^*(t,\cdot)-\tilde{\Phi}^{\tilde{*}}(t,\cdot)}_{H^{s-\frac{1}{2}}\times H^{s-\frac{1}{2}}}\nonumber\\
&\leq C\bigg(\norm{(\eta_0,\psi_0,\tilde{\eta}_0,\tilde{\psi}_0)}_{H^{s+\frac{1}{2}}\times H^{s}}\bigg)\norm{(\eta,\psi)^* (t,\cdot)-(\tilde{\eta},\tilde{\psi})^{\tilde{*}}(t,\cdot)}_{H^s\times H^{s-\frac{1}{2}}}.
\end{align}
\subsubsection{Gauge transform}\label{pos reg_flow GWW_proof_gauge}
Again, as $s>2+\frac{1}{2}$ we have an immediate $L^2$ estimates on $\Phi-\tilde{\Phi}$, thus we only need to get $\dot{H}^{s-\frac{1}{2}}\times \dot{H}^{s-\frac{1}{2}}$ estimates. 
Let us start by by writing $\Phi=\Phi_1+i\Phi_2$ in equation \eqref{pos reg_flow GWW_cor para sys Alazard_eq1}:
\begin{equation} \label{pos reg_flow GWW_proof_gauge_eq1}
\partial_t \Phi+T_V \cdot \partial_x \Phi+iT_\gamma \Phi=R_1(\Phi)\Phi,
\end{equation}
Where $R_1$ verifies
\[
\begin{cases}
\norm{R_1(\Phi)}_{H^{s-\frac{1}{2}}\rightarrow H^{s-\frac{1}{2}}}\leq C\bigg(\norm{(\eta_0,\psi_0,\tilde{\eta}_0,\tilde{\psi}_0)}_{H^{s+\frac{1}{2}}\times H^{s}}\bigg),\\
\norm{ [R_1(\Phi)-R_1(\tilde{\Phi})]\tilde{\Phi}}_{H^{s-\frac{1}{2}}}\leq C\bigg(\norm{(\eta_0,\psi_0,\tilde{\eta}_0,\tilde{\psi}_0)}_{H^{s+\frac{1}{2}}\times H^{s}}\bigg)\norm{\Phi-\tilde{\Phi}}_{H^{s-\frac{1}{2}}}.
\end{cases}
\] 
Indeed the first estimate on $R_1$ is a rephrasing of Corollary \ref{pos reg_flow GWW_cor para sys Alazard}. The difference estimates on $R_1$ follow from the fact that the change of variable to paralinearize and symmetrize the water waves system involves only taking paraproducts and Fourier multipliers which is a Lipschitz operation under sufficient regularity which is the case here given the hypothesis $s>2+\frac{1}{2}$. This is immediate for all of the terms in \eqref{pos reg_intro_WW_equations_WW syst} except for the Dirichlet-Neumann operator for which this follows from Proposition $3.14$ of \cite{Alazard11}.

The next step is to preform the change of variable by $\chi$, by Theorem \ref{paracomposition_section Paracomposition_subsec paracomp on the euclidean space_theorem paralinearisation of composition} we get:
\begin{equation} \label{pos reg_flow GWW_proof_gauge_eq1}
\partial_t \Phi^*+T_W \cdot \partial_x \Phi+iT_{\abs{\xi}^{\frac{3}{2}}} \Phi^*=R'_1(\Phi^*)\Phi^*,\text{ with }  W=\frac{V\circ \chi}{\partial_x \chi} \text{ and } \int\limits_\t W=0.
\end{equation}
Where $R'_1=(R_1)^*+T_{i\abs{\xi}^{\frac{3}{2}}}-T_{i\gamma^*}$ where $(R_1)^*$ and $T_{i\gamma^*}$ are the pull-back by $\chi$, Then $R_1'$ verifies
\[
\norm{ [R'_1(\Phi^*)-R_1(\tilde{\Phi}^{\tilde{*}})]\tilde{\Phi}^{\tilde{*}}}_{H^{s-\frac{1}{2}}}\leq C\bigg(\norm{(\eta_0,\psi_0,\tilde{\eta}_0,\tilde{\psi}_0)}_{H^{s+\frac{1}{2}}\times H^{s}}\bigg)\norm{\Phi^*-\tilde{\Phi}^{\tilde{*}}}_{H^{s-\frac{1}{2}}}.
\] 
We get the same equation on $\tilde{\Phi}^{\tilde{*}}$ by symmetry. The difference estimates on $R'_1$ follows from the fact that $R'_1$ is the pullbacks of $R_1$ by the paracomposition of $\chi$ and the structure of $R_1$ noted above.

Introduce the following gauge transform $e^{iT_{p_\Phi}}$ as the time one of the flow map defined by Propositions \ref{BCH formula_def para hyperbolic flow} with 
$$p_\Phi=\frac{2}{3} \abs{\xi}^{\frac{1}{2}}\partial_x^{-1}W\in \Gamma_2^{2-\alpha}(\t),$$
and put,
\begin{equation}\label{pos reg_flow GWW_proof_gauge_eq2}
\theta=e^{iT_{p_\Phi}}\Phi^*.
\end{equation}
 We define analogously $e^{iT_{p_{\tilde{\Phi}}}}$ and $\tilde{\theta}$ from $\tilde{\Phi}^{\tilde{*}}$.
  From Proposition \ref{BCH formula_def para hyperbolic flow} the change of variable \eqref{pos reg_flow GWW_proof_gauge_eq2} is Lipschitz from $H^{s-\frac{1}{2}}$ to $H^{s-\frac{1}{2}}$ but under $H^{s}$ control on $(\Phi,\tilde{\Phi})$ which is equivalent by Theorem \ref{paracomposition_Notions of microlocal analysis_Paradifferential Calculus_symbolic calculus para precised} to a control on $\norm{(\eta_0,\psi_0,\tilde{\eta}_0,\tilde{\psi}_0)}_{H^{s+\frac{1}{2}}\times H^{s}}$. We have:
\begin{align} \label{pos reg_flow GWW_proof_gauge_eq3}
&\norm{\Phi^* (t,\cdot)-\tilde{\Phi}^{\tilde{*}}(t,\cdot)}_{\dot{H}^{s-\frac{1}{2}}}\leq C\bigg(\norm{(\eta_0,\psi_0,\tilde{\eta}_0,\tilde{\psi}_0)}_{H^{s+\frac{1}{2}}\times H^{s}}\bigg)\norm{\theta(t,\cdot)-\tilde{\theta}(t,\cdot)}_{\dot{H}^{s-\frac{1}{2}}},\\
&\norm{\theta(t,\cdot)-\tilde{\theta}(t,\cdot)}_{\dot{H}^{s-\frac{1}{2}}}\leq C\bigg(\norm{(\eta_0,\psi_0,\tilde{\eta}_0,\tilde{\psi}_0)}_{H^{s+\frac{1}{2}}\times H^{s}}\bigg)\norm{\Phi^* (t,\cdot)-\tilde{\Phi}^{\tilde{*}}(t,\cdot)}_{\dot{H}^{s-\frac{1}{2}}}.
\end{align}

To get the equations on $\theta$ and $\tilde{\theta}$ we commute $e^{iT_{p_\Phi}}$ and $e^{iT_{p_{\tilde{\Phi}}}}$ with \eqref{pos reg_flow GWW_proof_gauge_eq1}, we make the computations for $\theta$, those for $\tilde{\theta}$ are obtained by symmetry:
\[
e^{iT_{p_\Phi}}\partial_t \Phi^*+e^{iT_{p_\Phi}} T_W \cdot \partial_x \Phi^*+ie^{iT_{p_\Phi}} T_{\abs{\xi}^{\frac{3}{2}}} \Phi=e^{iT_{p_\Phi}} R'_1(\Phi^*)\Phi^*
\]
$$
\partial_t e^{iT_{p_\Phi}}\Phi^*+iT_{\abs{\xi}^{\frac{3}{2}}} e^{iT_{p_\Phi}}\Phi^*+(e^{iT_{p_\Phi}} T_W \partial_x-[iT_{\abs{\xi}^{\frac{3}{2}}},e^{iT_{p_\Phi}}])\Phi^*-(\partial_t e^{iT_{p_\Phi}} )\Phi^*=e^{iT_{p_\Phi}}R'_1(\Phi^*)\Phi^*
$$
By the definition of $p_\Phi$ and Proposition \ref{BCH formula_prop diff Atau param} we have:
$$\partial_\xi(\abs{\xi}^\frac{3}{2})\partial_x p_\Phi
	=  W\xi \text{ and } \partial_t e^{iT_{p_\Phi}} =e^{iT_{p_\Phi}} \int\limits_0^1e^{-irT_{p_\Phi}} T_{i\partial_t p_\Phi} e^{irT_{p_\Phi}}dr.
	$$
thus by Corollary \ref{BCH formula_cor reste gaue trans} we get:
\begin{equation}\label{pos reg_flow GWW_proof_gauge_eq4}
\partial_t \theta+iT_{\abs{\xi}^{\frac{3}{2}}}  \theta
=R_2(\theta)\theta+e^{iT_{p_\Phi}} R_1(\Phi^*)e^{-iT_{p_\Phi}}\Phi^*, \text{ with }
\end{equation}
\[
R_2(\theta)\cdot =-(e^{iT_{p_\Phi}} T_W \partial_x-[iT_{\abs{\xi}^{\frac{3}{2}}},e^{iT_{p_\Phi}}])e^{-iT_{p_\Phi}}\cdot+(\partial_t e^{iT_{p_\Phi}} )e^{-iT_{p_\Phi}}\cdot.
\]
Now as in the case of the model problem above by Corollary \ref{BCH formula_cor reste gaue trans} $R_2$ and $e^{iT_{p_\Phi}} R_1(\Phi)e^{-iT_{p_\Phi}}$ and as $s>3+\frac{1}{2}$ they verify: 
\begin{align*}
&\norm{\RE(R_2(\theta))}_{H^{s-\frac{1}{2}} \rightarrow H^{s-\frac{1}{2}}}\leq C\bigg(\norm{(\eta_0,\psi_0,\tilde{\eta}_0,\tilde{\psi}_0)}_{H^{s+\frac{1}{2}}\times H^{s}}\bigg),\\
&\norm{ [R_2(\theta)-R_2(\tilde{\theta})]\tilde{\theta}}_{H^{s-\frac{1}{2}}}\leq C\bigg(\norm{(\eta_0,\psi_0,\tilde{\eta}_0,\tilde{\psi}_0)}_{H^{s+\frac{1}{2}}\times H^{s}}\bigg)\norm{\theta(t,\cdot)-\tilde{\theta}(t,\cdot)}_{H^{s-\frac{1}{2}}},
\intertext{and,}
&\norm{ [e^{iT_{p_\Phi}} R_1'(\Phi^*)e^{-iT_{p_\Phi}}-e^{iT_{p_{\tilde{\Phi}}}} R_1'(\tilde{\Phi}^*)e^{-iT_{p_{\tilde{\Phi}}}}]\tilde{\theta}}_{H^{s-\frac{1}{2}}} \\ &\leq C\bigg(\norm{(\eta_0,\psi_0,\tilde{\eta}_0,\tilde{\psi}_0)}_{H^{s+\frac{1}{2}}\times H^{s}}\bigg)\norm{\theta(t,\cdot)-\tilde{\theta}(t,\cdot)}_{H^{s-\frac{1}{2}}}.
\end{align*}

Thus we have succeeded to eliminate the term $T_V \cdot \partial_x $ of order $1$ in \eqref{pos reg_flow GWW_proof_gauge_eq1} and got a term of order $\frac{1}{2}$. The result then follows from a standard energy estimate.

\section{Baker-Campbell-Hausdorff formula: composition and commutator estimates}\label{BCH formula}
We will start by giving the propositions defining the operators used in the gauge transforms and the symbolic calculus associated to them. From those propositions we will deduce the direct estimates used in Sections \ref{pos reg_study of mod pbm_subsec proof thm DB exact flow reg_subsec gauge trfm} and \ref{pos reg_flow GWW_proof_gauge}.
\begin{notation}
To compute the conjugation and commutation of operators with a flow map, we introduce Lie derivatives, i.e commutators. More precisely we introduce the following notations for commutation between operators:
\[
\mathfrak{L}^0_a b=b, \ \mathfrak{L}_a b=[a,b]=a\circ b -b \circ a,\
\mathfrak{L}^2_a b=[a,[a,b]] , \ \mathfrak{L}^k_a b=\underbrace{[a,[\cdots,[a,}_{\text{k times}}b]]\cdots].
\] 

In the following proposition the variable $t\in [0,T]$ is the generic time variable that appeared in the previous section and a new variable $\tau \in \r$ will be used and they should not be confused. Finally by abuse of notation $\sr(\d)$ will designate $C^\infty(\t)$ in the case $\t=\d$.
\end{notation}
We start with the proposition defining the flow map and its standard properties. 
\begin{proposition}\label{BCH formula_def para hyperbolic flow}
Consider two real numbers $\delta < 1$, $s\in \r$ and a real valued symbol $p\in \Gamma^{\delta}_1(\d)$.
The following linear hyperbolic equation is globally well-posed:
\begin{equation}\label{BCH formula_def para hyperbolic flow_eq CP}
\begin{cases} 
\partial_\tau h-iT_{p} h=0, \\
h(0,\cdot)=h_0(\cdot)\in H^s(\d).
\end{cases}
\end{equation}
For $\tau \in \r$, define $e^{i \tau T_{p}}$ as the flow map associated to \ref{BCH formula_def para hyperbolic flow_eq CP} i.e: 
\begin{align}\label{BCH formula_def para hyperbolic flow_eq def flow}
e^{i \tau T_{p}}:&H^s(\d) \rightarrow H^s(\d)\nonumber\\
&h_0 \mapsto h(\tau,\cdot).
\end{align}
Then for $\tau \in \r$ we have, 
\begin{enumerate}
\item $e^{i \tau T_{p}}\in \mathscr{L}(H^s(\d))$ and
$$ \norm{e^{i \tau T_{p}}}_{H^s \rightarrow H^s}\leq e^{C\abs{\tau}  M_1^\delta(p)}. $$
\item $$iT_{p}\circ e^{i \tau T_{p}}=e^{i \tau T_{p}} \circ i T_{p}, \ e^{i(\tau+\tau')T_p}= e^{i \tau T_{p}}e^{i \tau' T_{p}}.$$
\item $e^{i \tau T_{p}}$ is invertible and,
 $$(e^{i \tau T_{p}})^{-1}=e^{-i \tau T_{p}}.$$ 
 Moreover,
  $$(e^{i \tau T_{p}})^*=e^{-i \tau (T_p)^*}=e^{-i \tau T_{p}}+R,$$ 
  where $R$ is a $\delta-1$ regularizing operator and $e^{i \tau (T_p)^*}$ is the flow generated by the Cauchy problem:
  \begin{equation}\label{BCH formula_def para hyperbolic flow_CP def flow conjugate}
\begin{cases} 
\partial_\tau h-i(T_p)^* h=0, \\
h(0,\cdot)=h_0(\cdot)\in H^s(\d).
\end{cases}
\end{equation}
\item Taking a real valued symbol $\tilde{p}\in \Gamma^{\delta}_1(\d)$ we have:
\begin{equation}\label{BCH formula_def para hyperbolic flow_est diff flow}
   \norm{[e^{i \tau T_{p}}-e^{i \tau T_{\tilde{p}}}] h_0}_{H^s }\leq C\abs{\tau} e^{C\abs{\tau}  M^\delta_1(p,\tilde{p}) }M^\delta_0(p-\tilde{p})\norm{h_0}_{H^{s+\delta}}.
   \end{equation}
\end{enumerate}
\end{proposition}
\begin{proof}
To prove point $(1)$ we commute $\D^s$ with the equation and make an energy estimate to get:
\[
\frac{d}{dt}\norm{h(t,\cdot)}_{H^s}^2\leq \left[\norm{[\D^s,iT_p]}_{H^s\rightarrow H^s}+\norm{T_p-\left(T_p\right)^*}_{H^s\rightarrow H^s}\right]\norm{h(t,\cdot)}_{H^s}^2,
\]
From Appendix \ref{paracomposition_Notions of microlocal analysis_Paradifferential Calculus} we have as $\delta \leq 1$ and $p\in \Gamma^{\delta}_1(\d)$
\[
\norm{[\D^s,iT_p]}_{H^s\rightarrow H^s}\leq  K M_1^\delta(p),
\]
moreover $p$ being real valued we have
\[
\norm{T_p-\left(T_p\right)^*}_{H^s\rightarrow H^s}\leq  K M_1^\delta(p).
\]
The desired estimate then follows from the Gronwall lemma.
The identities in point $(2)$ and the inverse and adjoint identities in point $(3)$ are the standard algebraic identities for semi-groups. As for the residual term $R$ in point $(3)$ it is given explicitly by
\[
R=i\int\limits_0^\tau e^{i(r-\tau)T_p}[T_p-(T_p)^*]e^{-ir(T_p)^*}dr.
\]
and $H^s\rightarrow H^{s+\delta-1}$ estimate again follows from Appendix \ref{paracomposition_Notions of microlocal analysis_Paradifferential Calculus}. Point $(4)$ comes by writing: 
$$
\partial_\tau [e^{i \tau T_{p}}-e^{i \tau T_{\tilde{p}}}]h_0-iT_{p} [e^{i \tau T_{p}}-e^{i \tau T_{\tilde{p}}}]h_0=iT_{p-\tilde{p}}e^{i \tau T_{\tilde{p}}} h_0,
$$
and making the usual energy estimate.
\end{proof}

The hypothesis $p\in \Gamma^\delta_1$ can be relaxed to $p\in \Gamma^\delta_\delta$, which is the minimal hypothesis needed to ensure well posedness in Sobolev spaces of \eqref{BCH formula_def para hyperbolic flow_eq CP}.

At the moment the only bounds we obtained on $e^{i \tau T_{p}}$ are the continuity bounds on Sobolev spaces, in order to study it's symbol and the  symbol of conjugated operators we need to transfer those continuity bounds to estimates on the symbol's seminorms. This was first done by Beals in \cite{Beals 77} for pseudodifferential operators in the class $S^m_{\rho,\rho}$ with $\rho<1,m\in \r$. The following lemma gives explicitly the key estimate adapted from \cite{Beals 77} given by \eqref{BCH formula_def para hyperbolic flow_lem Beals seminorm estimates adapted_est in x}, and we give one new estimate \eqref{BCH formula_def para hyperbolic flow_lem Beals seminorm estimates adapted_est in xi} that can then be directly applied to the paradifferential setting.

\begin{lemma}\label{BCH formula_def para hyperbolic flow_lem Beals seminorm estimates adapted}
Consider an operator $A$ continuous from $\sr(\d)$ to $\sr'(\d)$ and let $a\in \sr'(\d\times \hat{\d})$ be the unique symbol associated to A (cf \cite{Bony 13} for the uniqueness), that is if you let $K$ be the kernel associated to $A$ then
\[
\text{ for }u,v \in \sr(\d),\ (Au,v)=K(u\otimes v) \text{ and } a(x,\xi)=\fr_{y\rightarrow \xi}K(x,x-y).
\]
\begin{itemize}
\item If $A$ is continuous from $H^m$ to $L^2$, with $m\in  \r$, and $[\frac{1}{i}\frac{d}{dx},A]$ is continuous from $H^{m+\delta}$ to $L^2$ with $\delta\in (-\infty,1)$, then $(1+\abs{\xi})^{-m}a(x,\xi)\in L^\infty_{x,\xi}(\d \times \hat{\d})$ and we have the estimate:
\begin{equation}\label{BCH formula_def para hyperbolic flow_lem Beals seminorm estimates adapted_est in x}
\norm{(1+\abs{\xi})^{-m}a}_{L^\infty_{x,\xi}}\leq C_m \bigg[\norm{A}_{H^m \rightarrow L^2}+\norm{\bigg[\frac{1}{i}\frac{d}{dx},A\bigg]}_{H^{m+\delta} \rightarrow L^2}\bigg].
\end{equation}
\item If $A$ is continuous from $H^m$ to $L^2$, with $m\in  \r$, and $[ix,A]$ is continuous from from $H^{m-\rho}$ to $L^2$ with $\rho \geq 0$, then $(1+\abs{\xi})^{-m}a(x,\xi)\in L^\infty_{x,\xi}(\d \times \hat{\d})$ and we have the estimate:
\begin{equation}\label{BCH formula_def para hyperbolic flow_lem Beals seminorm estimates adapted_est in xi}
\norm{(1+\abs{\xi})^{-m}a}_{L^\infty_{x,\xi}}\leq C_m [\norm{A}_{H^m \rightarrow L^2}+\norm{[ix,A]}_{H^{m-\rho} \rightarrow L^2}].
\end{equation}
\end{itemize}
\end{lemma}

\begin{proof}
First without loss of generality through a standard mollification argument we work with $a \in \sr(\d\times \hat{\d})$. We study the cases on $\t$ and $\r$ separately.
\paragraph*{\bf Operators defined on $\t$} The first key observation is the following:
\begin{equation}\label{BCH formula_def para hyperbolic flow_lem Beals seminorm estimates adapted_proof_eq symbol iden}
(x,\xi)\in \t \times \z, \ e^{-i x\cdot \xi}A e^{ix\cdot\xi}=a(x,\xi),
\end{equation}
which one can write as $e^{ix\cdot\xi}\in L^2_x(\t)$. Thus taking $L^2$ norms in $x$ we get:
\begin{equation}\label{BCH formula_def para hyperbolic flow_lem Beals seminorm estimates adapted_proof_eq1}
\norm{a(\cdot,\xi)}_{L^2_x} \leq \norm{A}_{H^m \rightarrow L^2} \norm{e^{ix\cdot\xi}}_{H^m_x}\leq C_m  \norm{A}_{H^m \rightarrow L^2} (1+\abs{\xi})^{m}.
\end{equation}
Now to get the analogue of \eqref{BCH formula_def para hyperbolic flow_lem Beals seminorm estimates adapted_proof_eq1} but in the $\xi$ variable we observe that the continuity hypothesis reads for $(u,v)\in \sr$:
\[
 \abs{\int\limits_{\t\times \z} e^{ix\cdot \xi} a(x,\xi)(1+\abs{\xi})^{-m}\fr(v)(\xi)u(x)dx d\xi}\leq \norm{A}_{H^m \rightarrow L^2} \norm{\fr(v)}_{L^2_\xi}\norm{u}_{L^2_x},
\]
%which we rewrite as:
%\[ \abs{\int\limits_{\z} \fr(v)(\xi)\bigg[ \int\limits_\t e^{ix\cdot \xi} (1+\abs{\xi})^{-m}a(x,\xi)u(x)dx\bigg] d\xi}\leq \norm{A}_{H^m \rightarrow L^2} \norm{\fr(v)}_{L^2_\xi}\norm{u}_{L^2_x},\]
Thus for all $u\in L^2(\t)$
\[
\norm{\int\limits_\t e^{ix\cdot \xi} (1+\abs{\xi})^{-m}a(x,\xi)u(x)dx}_{ L^2_\xi} \leq C_m  \norm{A}_{H^m \rightarrow L^2}\norm{u}_{L^2_x}.
\]
Fixing $x_0\in \t$ and choosing $u(x)=\mathbbm{1}_{x_0\pm \epsilon}$, for $\epsilon$ sufficiently small we get by continuity of $a$ the analogous estimate to \eqref{BCH formula_def para hyperbolic flow_lem Beals seminorm estimates adapted_proof_eq1}:
\begin{equation}\label{BCH formula_def para hyperbolic flow_lem Beals seminorm estimates adapted_proof_eq2}
\norm{(1+\abs{\xi})^{-m}a(x,\xi)}_{L^\infty_x L^2_\xi} \leq C_m  \norm{A}_{H^m \rightarrow L^2}.
\end{equation}
The second key observation is:
\begin{equation}\label{BCH formula_def para hyperbolic flow_lem Beals seminorm estimates adapted_proof_eq Beals symbol identity 2}
e^{-i x\cdot \xi}\bigg[\frac{1}{i}\frac{d}{dx},A\bigg] e^{ix\cdot\xi}=\partial_x a(x,\xi), \ \ e^{-i x\cdot \xi}[ix,A] e^{ix\cdot\xi}=\partial_\xi a(x,\xi).
\end{equation}
Making the previous computations again with $\partial_x a$ and $\partial_\xi a$ instead of $a$ we get:
\begin{equation}\label{BCH formula_def para hyperbolic flow_lem Beals seminorm estimates adapted_proof_eq3}
\norm{\partial_x a(\cdot,\xi)}_{L^2_x}\leq C_m (1+\abs{\xi})^{m+\delta}\norm{\bigg[\frac{1}{i}\frac{d}{dx},A\bigg]}_{H^{m+\delta} \rightarrow L^2},
\end{equation}
and as $\rho\geq 0$:
\begin{equation}\label{BCH formula_def para hyperbolic flow_lem Beals seminorm estimates adapted_proof_eq4}
\norm{\partial_\xi[(1+\abs{\xi})^{-m}a(x,\xi)]}_{L^\infty_x L^2_\xi} \leq C_m  [\norm{A}_{H^m \rightarrow L^2}+\norm{[ix,A]}_{H^{m-\rho} \rightarrow L^2}].
\end{equation}

By the Sobolev embedding \eqref{BCH formula_def para hyperbolic flow_lem Beals seminorm estimates adapted_proof_eq2} and \eqref{BCH formula_def para hyperbolic flow_lem Beals seminorm estimates adapted_proof_eq4} give the desired result \eqref{BCH formula_def para hyperbolic flow_lem Beals seminorm estimates adapted_est in xi} on the torus.

To get \eqref{BCH formula_def para hyperbolic flow_lem Beals seminorm estimates adapted_est in x} we introduce as in \cite{Beals 77}: $$b(x,\xi,\xi_0)=a\bigg(\frac{x}{(1+\abs{\xi_0})^{\delta}},(1+\abs{\xi_0})^{\delta}\xi\bigg).$$ 
Thus \eqref{BCH formula_def para hyperbolic flow_lem Beals seminorm estimates adapted_proof_eq1} becomes
\[
\norm{b(\cdot,\xi)}_{L^2_x}\leq C_m  \norm{A}_{H^m \rightarrow L^2} (1+\abs{\xi})^{m}.
\]
As $\delta<1$ we have that $(1+\abs{\xi})^{\delta}\sim(1+\abs{\xi_0})^{\delta}$ for $\abs{\xi-\xi_0}\leq c (1+\abs{\xi_0})^{\delta}$ for some fixed $c> 0$. Thus \eqref{BCH formula_def para hyperbolic flow_lem Beals seminorm estimates adapted_proof_eq3} becomes
\[
\norm{\partial_x b(\cdot,\xi)}_{L^2_x}\leq C_m \frac{(1+\abs{\xi})^{m+\delta}}{(1+\abs{\xi_0})^{\delta}}\norm{\bigg[\frac{1}{i}\frac{d}{dx},A\bigg]}_{H^{m+\delta} \rightarrow L^2}\leq C_m \norm{\bigg[\frac{1}{i}\frac{d}{dx},A\bigg]}_{H^{m+\delta} \rightarrow L^2},
\]
Considering $b$ as a function of $(x,\xi)$ on $\t \times \b(\xi_0,c)$ by the Sobolev embedding we get:
\begin{equation}\label{BCH formula_def para hyperbolic flow_lem Beals seminorm estimates adapted_proof_eq5}
\norm{ b(x,\xi)}_{L^\infty_x}\leq C_m  \bigg[\norm{A}_{H^m \rightarrow L^2} +\norm{\bigg[\frac{1}{i}\frac{d}{dx},A\bigg]}_{H^{m+\delta} \rightarrow L^2}\bigg](1+\abs{\xi})^{m},
\end{equation}
which transferred back to $a$ give the desired result \eqref{BCH formula_def para hyperbolic flow_lem Beals seminorm estimates adapted_est in x} on the torus.
\paragraph*{\bf Operators defined on $\r$} The main problem we face on $\r$ when adapting the previous proof is we can no longer use $e^{ix\cdot\xi}$ as a test function as it no longer belongs to $L^2(\r)$. One way to get over this was given by Beals in \cite{Beals 77}, we choose $g$ in $\sr(\r)$ such that $g(0)=1,$ $\fr(g)$ is supported in $\set{\abs{\xi}\leq 1}$ and $g(x)=g(-x)$. Let $g_x(y)=g(y-x)$ and compute for $u \in \sr$:
\[
u(x)=u(x)g_x(x)=\int\limits_{\r  \times \r} e^{i(x-y)\cdot \xi}g_x(y)u(y)dy d\xi=\int\limits_{\r  \times \r} e^{-iy\cdot \xi}e^{ix}g_y(x)u(y)dy d\xi.
\]
We now compute an analogue of \eqref{BCH formula_def para hyperbolic flow_lem Beals seminorm estimates adapted_proof_eq symbol iden}:
\begin{align*}
Au(x)&=\int\limits_{\r  \times \r} e^{-iy\cdot \xi}A\big(e^{ix}g_y\big)(x)u(y)dy d\xi\\
&=\int\limits_{\r  \times \r} e^{i(x-y)\cdot \xi}a_0(x,y,\xi)u(y)dy d\xi\\
&=\int\limits_{\r } e^{ix\cdot \xi}a(x,\xi)\fr(u)(\xi) d\xi,
\end{align*}
where,
\[
a_0(x,y,\xi)=e^{-ix\cdot \xi}A\big(e^{ix\cdot \xi}g_y\big)(x),
\]
and,
\[
a(x,\xi)=\int\limits_{\r  \times \r}e^{i(x-y)\cdot (\eta-\xi)}a_0(x,y,\eta)d\eta dy.
\]
Applying the same arguments as in the periodic case we get:
\begin{equation}\label{BCH formula_def para hyperbolic flow_lem Beals seminorm estimates adapted_proof_eq6}
\norm{(1+\abs{\xi})^{-m}\partial^k_y a_0}_{L^\infty_{x,y,\xi}}\leq C_{m,k} \bigg[\norm{A}_{H^m \rightarrow L^2}+\norm{\bigg[\frac{1}{i}\frac{d}{dx},A\bigg]}_{H^{m+\delta} \rightarrow L^2}\bigg], k\in \n
\end{equation}
and,
\begin{equation}\label{BCH formula_def para hyperbolic flow_lem Beals seminorm estimates adapted_proof_eq7}
\norm{(1+\abs{\xi})^{-m}\partial^k_y a_0}_{L^\infty_{x,y,\xi}}\leq C_{m,k} [\norm{A}_{H^m \rightarrow L^2}+\norm{[ix,A]}_{H^{m-\rho} \rightarrow L^2}],k\in \n.
\end{equation}

Thus to conclude the proof we need to transfer the information on the amplitude $a_0$ to the symbol $a$ which is a simple application of Oscillatory Integrals. Indeed it suffices to write:
\begin{align*}
a(x,\xi)&=\int\limits_{\r  \times \r}e^{i(x-y)\cdot (\eta-\xi)}a_0(x,y,\eta)d\eta dy\\
&=\int\limits_{\r  \times \r}\frac{1}{\langle \xi-\eta \rangle^2 }(I-\Delta_y)e^{i(x-y)\cdot (\eta-\xi)}a_0(x,y,\eta)d\eta dy\\
&=\int\limits_{\r \times \r} \frac{1}{\langle \xi-\eta \rangle^2 } e^{i(x-y)\cdot (\eta-\xi)} (I-\Delta_y)^*a_0(x,y,\eta)d\eta dy,
\end{align*}
which gives the desired result as $a_0(x,y,\eta)$ and all of it's $y$ derivatives are $y$ integrable.
\end{proof}

\begin{remark}
To get a good intuition behind Beals type estimates, it is worth noting that if instead of Sobolev continuity bounds we had H\"older continuity estimates on $A$ with $m \notin \n$, then combined with the identity in the torus
\[
e^{-ix\cdot \xi}Ae^{ix\cdot \xi}=a(x,\xi).
\] we get directly the analogue of estimates \eqref{BCH formula_def para hyperbolic flow_lem Beals seminorm estimates adapted_est in x} and \eqref{BCH formula_def para hyperbolic flow_lem Beals seminorm estimates adapted_est in xi} that is
\[
\norm{(1+\abs{\xi})^{-m}a}_{L^\infty_{x,\xi}}\leq C_m \norm{A}_{W^{m.\infty} \rightarrow L^\infty}.
\]
\end{remark}

Analogously to Beals characterization of pseudodifferential operators through the continuity of the successive commutators $\op(x)$, $\frac{1}{i}\frac{d}{dx}$ with $A$ on Sobolev spaces, we give the following characterization of Paradifferential operators through estimate \eqref{BCH formula_def para hyperbolic flow_lem Beals seminorm estimates adapted_est in xi}.

\begin{corollary}\label{BCH formula_def para hyperbolic flow_characterization of pseudodifferential operators by sobolev continuity}
 Consider two real numbers $m$ and $\rho \geq 0$ and the spaces of paradifferential symbols $\Gamma_\rho^m(\d)$ equipped with the topology induced by the seminorms $M^m_\rho(\cdot;k)$ for $k\in \n$ defined in \ref{paracomposition_Notions of microlocal analysis_Paradifferential Calculus_ definition semi-norms} giving it a Fr\'echet space structure.

For $p\in \Gamma_\rho^m(\d)$ we introduce the following family of seminorms:
\[
H^m_0(p;k)=\sum^k_{j=0}\norm{\mathfrak{L}^j_{ix}T_p}_{H^m \rightarrow H^{j}},
\]
\[
H^m_n(p;k)=\sum_{l=0}^n\sum^k_{j=0}\norm{\mathfrak{L}^j_{ix}\mathfrak{L}^l_{\frac{1}{i}\frac{d}{dx}}T_p}_{H^m \rightarrow H^{j}},  \ n\in \n, n\leq \rho,
\]
and if $\rho \notin \n$:
\[
H^m_\rho(p;k)=H^m_{\lfloor \rho \rfloor}(p;k)+\sup_{n\in \n}2^{n(\rho-\lfloor \rho \rfloor)}\sum^k_{j=0}\norm{\mathfrak{L}^j_{ix}\mathfrak{L}^{\lfloor \rho \rfloor}_{\frac{1}{i}\frac{d}{dx}}T_p P_{\leq n}(D)}_{H^m \rightarrow H^{j}}.
\]
Then $H^m_\rho(p;k)_{k\in \n}$ induces an equivalent Fr\'echet topology to $M^m_\rho(\cdot;k)_{k\in \n}$ on:
\[
\psi^{B,b}\bigg(\Gamma^m_\rho(\d)\bigg)=\set{\sigma^{B,b}_p,p\in \Gamma^m_\rho(\d)}.
\]
\end{corollary}

\begin{proof}
Taking $m\in \r$, $\rho\geq 0$ and $p\in 
 \Gamma_\rho^m(\d)$ then Theorem \ref{paracomposition_Notions of microlocal analysis_Paradifferential Calculus_para continuity} gives for $\rho \in \n$
 \[ 
 H^m_\rho(p;k)\leq K M^m_\rho(p;k+1), \text{ for } k\in \n
\]
and the case $\rho \geq 0$ and $\rho \in \r \setminus \n$ is obtained from the characterization of $H^\rho$ and $W^{\rho,\infty}$ spaces by a dyadic decomposition on balls in the frequency space given by Propositions \ref{paracomposition_Notations and functional analysis_proposition Zygmund spaces on balls} and \ref{paracomposition_Notations and functional analysis_proposition Sobolev spaces on balls} and the observation that there exists an $N\in \n$ such that for all $n\in \n$
\[
T_pP_{\leq n}(D)=T_{P_{\leq n-N}(D) p}P_{\leq n}(D)
\]
where $P_{\leq n-N}(D)$ is applied to $p$ in the $x$ variable. Now by Lemma \ref{BCH formula_def para hyperbolic flow_lem Beals seminorm estimates adapted} for $\rho \in \n$
\[ 
  M^m_\rho(\sigma^{B,b}_p;k+1)\leq H^m_\rho(p;k), \text{ for } k\in \n,
\]
and the case $\rho$ non integer is treated exactly as before which gives the equivalence of the desired topologies on the subspace \[
\psi^{B,b}\bigg(\Gamma^m_\rho(\d)\bigg)=\set{\sigma^{B,b}_p,p\in \Gamma^m_\rho(\d)},
\] by noticing that $\sigma^{B,b}_{\sigma^{B,b}_p}=\sigma^{B,b}_p+\sigma^{B+\epsilon,b}_{(1-\psi^{B,b})\psi^{B,b}p}$, thus estimates on $\sigma^{B,b}_{\sigma^{B,b}_p}$ and $\sigma^{B,b}_p$ are equivalent.
\end{proof}

In order to give the key symbolic calculus results in Propositions \ref{BCH formula_def para hyperbolic flow_prop com-conjg} and \ref{BCH formula_lem est const lie derv_prop Atau symb} we need to introduce the paradifferential analogue of the H\"ormander symbol class $S^0_{1-\delta,\delta}$. For this we follow $\S$1.1 Chapter 1 of \cite{Taylor91} and introduce the space of non regular symbols:
\begin{definition-proposition}\label{BCH formula_def para hyperbolic flow_def nonreg symb}
Consider $s\in \r_+$, for $0\leq \delta, \rho <1$, we say:
\begin{equation}\label{BCH formula_def para hyperbolic flow_def nonreg symb_eq1}
p\in W^{s,\infty} S^m_{\rho,\delta}(\d) \iff 
\begin{cases}
 \norm{D^k_\xi p(\cdot,\xi)}_{W^{s,\infty}} \leq C_{k,0} \langle \xi \rangle^{m-\rho k}  \vspace*{0.2cm}\\
  \abs{D^{ \lfloor s \rfloor+n}_x D^k_\xi p(x,\xi)}\leq C_{k,n} \langle \xi \rangle^{m-\rho k+(n+\lfloor s \rfloor-s)\delta}
 \end{cases}, (x,\xi)\in \d \times \hat{\d},
\end{equation}
for $k\geq 0$ and $n\geq 1$. The best constants $C_{k,n}$ in \eqref{BCH formula_def para hyperbolic flow_def nonreg symb_eq1} define a family of seminorms denoted by ${}^{\rho,\delta}M^m_{n,s}(\cdot;k),(k,n)\in \n^2$ where $k$ is the number of derivatives we make on the frequency variable $\xi$ and $n$ is that in the $x$ variable. We also define the seminorm ${}^{\rho,\delta}M^m_s(\cdot;k)={}^{\rho,\delta}M^m_{0,s}(\cdot;k)$ and define  analogously $W^{s,\infty} S^m_{\rho,\delta}(\d^*\times  \hat{\d})$.

Analogously to Corollary \ref{BCH formula_def para hyperbolic flow_characterization of pseudodifferential operators by sobolev continuity} we introduce the following family of seminorms:
\begin{multline*}
{}^{\rho,\delta}H^m_{0,s}(p;k)=\sum_{l=0}^{\lfloor s \rfloor}\sum^k_{j=0}\norm{\mathfrak{L}^j_{ix}\mathfrak{L}^l_{\frac{1}{i}\frac{d}{dx}}\op(p)}_{H^m \rightarrow H^{j\rho}}\\+\sup_{n\in \n}2^{n(s-\lfloor s \rfloor)}\sum^k_{j=0}\norm{\mathfrak{L}^j_{ix}\mathfrak{L}^{\lfloor s \rfloor}_{\frac{1}{i}\frac{d}{dx}}[P_{n}(D)\op(p)] }_{H^m \rightarrow H^{j\rho}},
\end{multline*}
where $P_n(D)$ is applied to $p$ in the $x$ variable. And for $n\geq 1$ 
\[
{}^{\rho,\delta}H^m_{n,s}(p;k)={}^{\rho,\delta}H^m_{0,s}(p;k)+\sum_{l=1}^{n}\sum^k_{j=0}\norm{\mathfrak{L}^j_{ix}\mathfrak{L}^{\lfloor s \rfloor+l}_{\frac{1}{i}\frac{d}{dx}}\op(p)}_{H^m \rightarrow H^{j\rho-(l+\lfloor s \rfloor-s)\delta}}.
\]
Then ${}^{\rho,\delta}H^m_{n,s}(\cdot;k)_{(n,k)\in \n^2}$ induces an equivalent Fr\'echet topology to ${}^{\rho,\delta}M^m_{n,s}(\cdot;k)_{(n,k)\in \n^2}$ on $W^{s,\infty} S^m_{\rho,\delta}$.
\end{definition-proposition}

\begin{proof}
The $L^2$ continuity of such operators and control by their symbol seminorms is given in Appendix $B$ of \cite{Ayman20'} and the reciprocal is given by Lemma \ref{BCH formula_def para hyperbolic flow_lem Beals seminorm estimates adapted}.
\end{proof}

\begin{remark}\label{rem:hormander symbold classes}
We note that for the standard H\"ormander symbol classes and paradifferential symbols classes we have respectively \[S^m_{\rho,\delta}=\cap_{s\geq 0}  W^{s,\infty} S^m_{\rho,\delta}\text{ and } \Gamma^m_\rho = W^{\rho,\infty} S^m_{1,0}.\]
\end{remark}

Now all of the ingredients are in place to give the key commutation and conjugation result.
\begin{proposition}\label{BCH formula_def para hyperbolic flow_prop com-conjg}
Consider three real numbers $\delta < 1$, $s\in \r$, $\rho \geq 1$, a real valued symbol $p \in \Gamma^{\delta}_{\rho}(\d)$. Let $e^{i \tau T_{p}},\tau \in \r$ be the flow map defined by Proposition \ref{BCH formula_def para hyperbolic flow} and take a symbol $b\in \Gamma^{\beta}_{\rho}(\d)$ with $\beta \in \r$.
\begin{enumerate}
\setcounter{enumi}{4}
\item There exists $b_\tau^p \in W^{\rho,\infty} S^\beta_{1-\delta,\delta}(\d)$ such that:
\begin{align}\label{BCH formula_def para hyperbolic flow_prop com-conjg_eq def cong}
e^{i \tau T_{p}} \circ T_b \circ e^{-i\tau T_p}&=\op(b_\tau^p).
\end{align}
Moreover we have the estimates:
\begin{equation}\label{BCH formula_def para hyperbolic flow_prop com-conjg_est symbolic calc conjg}
\norm{\op(b_\tau^p)-\sum_{k=0}^{\lceil \rho-1 \rceil}\frac{\tau^k}{k!}\mathfrak{L}^k_{iT_p}T_b}_{H^s\rightarrow H^{s-\beta-\lceil \rho \rceil \delta+\rho}}\leq C_\rho e^{C \abs{\tau} M_1^\delta(p)} M_\rho^\beta(b) M_\rho^\delta(p)^{\lceil \rho \rceil},
\end{equation}
\begin{equation}\label{BCH formula_def para hyperbolic flow_prop com-conjg_est seminorm conjg}
{}^{1-\delta,\delta}H^\beta_\rho(b_\tau^p;k)\leq C_k(M_1^\delta(p))H^\beta_\rho(b;k) \left[\sum^{k-1}_{i=0}H^{\delta}_\rho(p;k)^{k-i}\right], \ k\in \n.
\end{equation}
\item There exists ${}^c b_\tau^p \in W^{\rho-1,\infty} S^{\beta+\delta-1}_{1-\delta,\delta}(\d)$ such that:
\begin{align}\label{BCH formula_def para hyperbolic flow_prop com-conjg_eq def com}
[e^{i \tau T_{p}}, T_b]&=e^{i \tau T_{p}} \op({}^c b_\tau^p) \iff \op({}^c b_\tau^p)=T_b-\op(b_{-\tau}^p).
\end{align}
Moreover we have the estimates:
\begin{equation}\label{BCH formula_def para hyperbolic flow_prop com-conjg_est symbolic calc com}
\norm{\op({}^c b_\tau^p)-\sum_{k=1}^{\lceil \rho-1 \rceil}(-1)^{k-1}\frac{\tau^k}{k!}\mathfrak{L}^k_{iT_p}T_b}_{H^s\rightarrow H^{s-\beta-\lceil \rho \rceil \delta+\rho}}\leq C_\rho e^{C \abs{\tau} M_1^\delta(p)} M_\rho^\beta(b) M_\rho^\delta(p)^{\lceil \rho \rceil},
\end{equation}
\begin{equation}\label{BCH formula_def para hyperbolic flow_prop com-conjg_est seminorm com}
{}^{1-\delta,\delta}H^{\beta+\delta-1}_{\rho-1}({}^c b_\tau^p;k)\leq C_k(M_1^\delta(p)) H^\beta_\rho(b;k) \left[\sum^{k-1}_{i=0}H^{\delta}_\rho(p;k)^{k-i}\right], k\in \n.
\end{equation}
\end{enumerate}
\end{proposition}

\begin{remark}\label{BCH formula_def para hyperbolic flow_rem on prop com-conjg}
\begin{itemize}
\item It is important to notice that the main result of this proposition is the factorization of the $e^{i \tau T_{p}}$ terms in \eqref{BCH formula_def para hyperbolic flow_prop com-conjg_eq def cong} and \eqref{BCH formula_def para hyperbolic flow_prop com-conjg_eq def com} where the right hand sides contain symbols in the usual classes modulo a more regular remainder. This was not a priori the case of the left hand sides containing $e^{i \tau T_{p}}$. In other words we prove the stability of $\Gamma^m_\rho$ under the conjugation by $e^{i \tau T_{p}}$ modulo more regular remainders.

This is crucial when studying the regularity of the flow map for: $$s>\lceil \frac{\alpha}{\alpha-1}\rceil-\frac{1}{2}.$$ Indeed if $p$ depends on a parameter $\lambda$, $D_\lambda e^{i \tau T_{p}} \circ T_b \circ e^{-i\tau T_p}$ is a priori an operator of order $\beta +\delta$ by \eqref{BCH formula_def para hyperbolic flow_est diff flow}, but $D_\lambda \op(b_\tau^p)$ is shown in proposition \ref{BCH formula_prop diff Atau conjg} to be an operator of order $\beta$. 

\item In the language of pseudodifferential operators, $\op(b_\tau^p)$ is the asymptotic sum of the series $(\frac{\tau^k}{k!}\mathfrak{L}^k_{iT_p}T_b)$ i.e the Baker-Campbell-Hausdorff formal series. Though $\op(b_\tau^p)$ is not necessarily equal to this sum, for this sum need not converge.
\end{itemize}
\end{remark}

\begin{proof}
The structure of the proof is as follows:
\begin{enumerate}[(I)]
\item We will give a proof of the estimate \eqref{BCH formula_def para hyperbolic flow_prop com-conjg_est symbolic calc conjg} assuming $b^p_\tau$ exists.
\item We will prove the existence of $b^p_\tau \in W^{\rho,\infty} S^\beta_{1-\delta,\delta}(\d)$ which is the subtle part of the proof.
\item Finally we will deduce point $(6)$ from point $(5)$.
\end{enumerate}

\paragraph*{Point $(I)$}

For point $(5)$ we compute,
\begin{align}\label{BCH formula_def para hyperbolic flow_proof prop com-conjg_point I_edo def conjg}
\partial_\tau [e^{i \tau T_{p}} \circ  T_{b} \circ e^{-i\tau T_p} ]&=iT_{p}\circ e^{i \tau T_{p}} \circ T_{b} \circ  e^{-i\tau T_p} - e^{i \tau T_{p}} \circ T_{b} \circ iT_{p} \circ  e^{-i\tau T_p}\nonumber
\intertext{Using $(2)$,}
\partial_\tau [e^{i \tau T_{p}} \circ  T_{b} \circ e^{-i\tau T_p} ]&= e^{i \tau T_{p}} \circ iT_{p}\circ T_{b} \circ  e^{-i\tau T_p} - e^{i \tau T_{p}} \circ T_{b} \circ iT_{p}\circ  e^{-i\tau T_p}\nonumber \\
&= e^{i \tau T_{p}}[iT_{p},T_{b}] e^{-i\tau T_p} .
\end{align}
As $e^{i0T_p}=Id$, integrating on $[0,\tau]$ we get: 
\[
e^{i \tau T_{p}}  \circ  T_{b} \circ e^{-i\tau T_p} =T_{b}+\int\limits_0^\tau\underbrace{e^{irT_p}  [iT_{p},T_{b}] e^{-irT_p} }_{*}dr.
\]
Iterating the computation in $*$ we get for $n \in \n^*$,
\begin{equation}\label{BCH formula_def para hyperbolic flow_proof prop com-conjg_point I_iterated int formula conjg}
e^{i \tau T_{p}}  \circ  T_{b} \circ e^{-i\tau T_p} =\sum_{k=0}^n\frac{\tau^k}{k!}\mathfrak{L}^k_{iT_p}T_b+\int\limits_0^\tau \frac{(\tau -r)^n}{n!}e^{irT_p}  \mathfrak{L}^{n+1}_{iT_p}T_be^{-irT_p} dr.
\end{equation}
Now the key point is the continuity of paradifferential operators given by Theorem \ref{paracomposition_Notions of microlocal analysis_Paradifferential Calculus_para continuity} and the symbolic calculus rules given by Theorem \ref{paracomposition_Notions of microlocal analysis_Paradifferential Calculus_symbolic calculus para precised}. By Lemma \ref{BCH formula_lem est const lie derv}: 
\begin{equation}\label{BCH formula_def para hyperbolic flow_proof prop com-conjg_point I_lie derv cont est}
\norm{\mathfrak{L}^{\lceil \rho \rceil}_{iT_p}T_b}_{H^s\rightarrow H^{s-\beta-\lceil \rho \rceil \delta+\rho}}\leq C_\rho M_\rho^\beta(b) M_\rho^\delta(p)^{\lceil \rho \rceil}, \text{ for }\rho\in \r_+,
\end{equation}
where $\lceil \rho \rceil$ is the upper integer part of $\rho$.

Thus applying point $(1)$ combined with \eqref{BCH formula_def para hyperbolic flow_proof prop com-conjg_point I_lie derv cont est} we get \eqref{BCH formula_def para hyperbolic flow_prop com-conjg_est symbolic calc conjg}.

\paragraph*{Point $(II)$}
The constant $C_\rho$ in \eqref{BCH formula_def para hyperbolic flow_proof prop com-conjg_point I_lie derv cont est} is estimated "brutally" by Lemma \ref{BCH formula_lem est const lie derv}: $O(2^\rho \times \lceil \rho \rceil!)$, thus even though one has a $\frac{1}{\lceil \rho \rceil!}$ in \eqref{BCH formula_def para hyperbolic flow_proof prop com-conjg_point I_iterated int formula conjg} the convergence result is non trivial. To get past this let us express explicitly the difficulty in the problem. Rearranging the terms in \eqref{BCH formula_def para hyperbolic flow_proof prop com-conjg_point I_edo def conjg} we see that:
\begin{equation}\label{BCH formula_def para hyperbolic flow_proof prop com-conjg_point II_edo def conjg rearenged}
\partial_\tau[e^{i \tau T_{p}} \circ  T_{b} \circ e^{-i\tau T_p}]=e^{i \tau T_{p}}[iT_{p},T_{b}] e^{-i\tau T_p}=[iT_{p},e^{i \tau T_{p}} T_{b} e^{-i\tau T_p}],
\end{equation}
thus we have to solve the following Cauchy problem in $\mathscr{L}(H^s(\d),H^{s-\beta}(\d))$:
\begin{equation}\label{BCH formula_def para hyperbolic flow_proof prop com-conjg_point II_gen edo}
\begin{cases}
\partial_\tau f(\tau)=[iT_{p},f(\tau)] \in \Gamma^{\beta}_{\rho-1}(\d),\\
f(0)=T_b\in \Gamma^{\beta}_{\rho}(\d).
\end{cases}
\end{equation}

This amounts to the non trivial problem of solving a linear ODE in the Fr\'echet space $\Gamma^{\beta}_{+\infty}(\d)$, indeed such a problem need not have a solution in general, and even if it does, it need not be unique. To solve such a problem one usually has to look at a Nash-Moser type scheme, though in our case we have an explicit ODE that can be solved with a series and a loss of derivative. Thus inspired by H\"ormander's \cite{Hormander90}, we remark another key estimate given by the continuity of paradifferential operators given by Theorem \ref{paracomposition_Notions of microlocal analysis_Paradifferential Calculus_para continuity} and the symbolic calculus in Theorem \ref{paracomposition_Notions of microlocal analysis_Paradifferential Calculus_symbolic calculus para precised}
summarised in the following lemma.
\begin{lemma}\label{BCH formula_lem est const lie derv}
Consider two real numbers $\delta, \beta$, $\rho \geq 0$, and two symbols $p \in \Gamma^{\delta}_{\rho}(\d)$ and $b\in\Gamma^{\beta}_\rho(\d)$ then there exists a constant $C>0$ such that:
\begin{equation}\label{BCH formula_lem est const lie derv_est 1}
\norm{\mathfrak{L}^{\lceil \rho \rceil}_{T_p}T_b}_{H^s\rightarrow H^{s-\beta-\lceil \rho \rceil \delta+\rho}}\leq C^{\lceil \rho \rceil+1}2^{\lceil \rho \rceil}\lceil \rho \rceil! M_\rho^\beta(b) M_\rho^\delta(p)^{\lceil \rho \rceil}, \text{ for }\rho\in \r_+,
\end{equation}
\begin{equation}\label{BCH formula_lem est const lie derv_est 2}
\norm{\mathfrak{L}^{\lceil \rho \rceil}_{T_p}T_b}_{H^s\rightarrow H^{s-\beta-\lceil \rho \rceil \delta}}\leq C^{\lceil \rho \rceil} M_0^\beta(b) M_0^\delta(p)^{\lceil \rho \rceil}, \text{ for }\rho\in \r_+.
\end{equation}
\end{lemma}
\begin{proof}[Proof of Lemma 4.2]
For \eqref{BCH formula_lem est const lie derv_est 2}, expanding $\mathfrak{L}^{\lceil \rho \rceil}_{T_p}T_b$ as polynomial in $T_p$ and $T_b$ and counting with repetitions, we see that it contains at most $2^{\lceil \rho \rceil}$ terms of the form:
\[
\pm T_p\circ\cdots \circ \underbrace{T_b}_{\text{position }i}\circ\cdots \circ T_p, \ i \in [0,\lceil \rho \rceil].
\]
Now by the continuity of paradifferential operators given in Theorem \ref{paracomposition_Notions of microlocal analysis_Paradifferential Calculus_para continuity} and \ref{paracomposition_Notions of microlocal analysis_Paradifferential Calculus_symbolic calculus para precised} we have:
\[\norm{T_p\circ\cdots \circ T_b \circ\cdots \circ T_p}_{H^s\rightarrow H^{s-\beta-\lceil \rho \rceil \delta}}\leq K^{\lceil \rho \rceil} M_0^\beta(b) M_0^\delta(p)^{\lceil \rho \rceil},
\]
which gives \eqref{BCH formula_lem est const lie derv_est 2}.

For \eqref{BCH formula_lem est const lie derv_est 1}, we start by the case $k\in \n^*$ is a an integer. We first notice that again by Theorem \ref{paracomposition_Notions of microlocal analysis_Paradifferential Calculus_para continuity} we have:
\[
\begin{cases}
\mathfrak{L}^{1}_{T_p}T_b \in \Gamma^{\beta+\delta-1}_{k-1},\\
M_{k-1}^{\beta+\delta-1}(\mathfrak{L}^{1}_{T_p}T_b)\leq C M_k^\beta(b) M_k^\delta(p).
\end{cases}
\]
Thus iterating this formula we get:
\[
\begin{cases}
\mathfrak{L}^{k}_{T_p}T_b \in \Gamma^{\beta+k\delta-k}_{0},\\
M_0^{\beta+\delta-k}(\mathfrak{L}^{k}_{T_p}T_b)\leq C_k M_k^\beta(b) \displaystyle{\prod^k_{i\geq 1}} M_i^\delta(p),
\end{cases}
\]
and $C_k$ verifies:
\[
C_k=2kC_{k-1} \Rightarrow C_k=C 2^k k!, 
\]
Thus giving the result in the case $k$ integer. For $\rho\geq 0$, it suffice to see that for $\rho \leq 1$ again by Theorem \ref{paracomposition_Notions of microlocal analysis_Paradifferential Calculus_para continuity} we have:
\[
\begin{cases}
\mathfrak{L}^{1}_{T_p}T_b \in \Gamma^{\beta+\delta-\rho}_{0},\\
M_0^{\beta+\delta-\rho}(\mathfrak{L}^{1}_{T_p}T_b)\leq C M_\rho ^\beta(b) M_\rho^\delta(p),
\end{cases}
\]
which concludes the proof.
\end{proof}
%\begin{equation}\label{BCH formula_def para hyperbolic flow_proof prop com-conjg_point II_lie derv est exact}\norm{\mathfrak{L}^{\lceil \rho \rceil}_{iT_p}T_b}_{H^s\rightarrow H^{s-\beta-\lceil \rho \rceil \delta}}\leq C^{\lceil \rho \rceil} M_0^\beta(b) M_0^\delta(p)^{\lceil \rho \rceil}, \text{ for }\rho\in \r_+.\end{equation}

This means that if we can compensate the loss of $\beta+\lceil \rho \rceil \delta$ derivatives with a cost negligible in comparison to $\lceil \rho \rceil!$, we would have a convergent series in \eqref{BCH formula_def para hyperbolic flow_proof prop com-conjg_point I_iterated int formula conjg}. A first approach would be to interpolate \eqref{BCH formula_def para hyperbolic flow_proof prop com-conjg_point I_lie derv cont est} and \eqref{BCH formula_lem est const lie derv_est 2} which gives:
\begin{align}\label{BCH formula_def para hyperbolic flow_proof prop com-conjg_point II_lie derv est after interp}
\norm{\mathfrak{L}^{\lceil \rho \rceil}_{iT_p}T_b}_{H^s\rightarrow H^{s-\beta}} &\leq C^{\frac{\rho-\lceil \rho \rceil \delta}{\rho}\lceil \rho \rceil} M_0^\beta(b)^{\frac{\rho-\lceil \rho \rceil \delta}{\rho}} M_0^\delta(p)^{\frac{\rho-\lceil \rho \rceil \delta}{\rho}\lceil \rho \rceil}\\
&\times C^{(\lceil \rho \rceil+1)\frac{\lceil \rho \rceil \delta}{\rho}}2^{\lceil \rho \rceil\frac{\lceil \rho \rceil \delta}{\rho}}\lceil \rho \rceil!^{\frac{\lceil \rho \rceil \delta}{\rho}} M_\rho^\beta(b)^\frac{\lceil \rho \rceil \delta}{\rho} M_\rho^\delta(p)^{\lceil \rho \rceil\frac{\lceil \rho \rceil \delta}{\rho}}, \nonumber
\end{align}
This indeed solves the cost $\lceil \rho \rceil!$ of \eqref{BCH formula_def para hyperbolic flow_proof prop com-conjg_point I_lie derv cont est} but depends on $M_\rho^{\cdot}$ norms of $b$ and $p$. An idea to control those norms in a cost negligible in comparison to $\lceil \rho \rceil!$ would be to mollify $p$ and $b$ using an analytic mollifier, this might work but we found it better to mollify differently. 

For this we introduce a mollification with the Gaussian function $\phi_\eps(D)$ with symbol:
\begin{equation}\label{BCH formula_def para hyperbolic flow_proof prop com-conjg_point II_def gauss moll}
\phi_\eps(\xi)=\frac{1}{\eps\sqrt{2\pi}}e^{-\frac{\xi^2}{2\eps^2}},\eps>0.
\end{equation}
Other than the standard properties of mollifiers, we have the following properties:
\begin{itemize}
\item For $h\in H^s(\d)$, $\phi_\eps(D)h$ is real analytic.
\item The moments of the Gaussian can be explicitly computed by, for $k\in \n$:
\[
\frac{1}{\sqrt{2\pi}}\int\limits_\r\abs{\xi}^k e^{-\frac{\xi^2}{2}}d\xi=\frac{2^{\frac{k}{2}}\Gamma(\frac{k+1}{2})}{\sqrt{\pi}}=(k-1)!!\begin{cases}
1 \text{ if } k \text{ is even}\\
\sqrt{\frac{2}{\pi}} \text{ if } k \text{ is odd}\\
\end{cases}.
\]
\item From the moments of the Gaussian we deduce that, for $h\in H^s(\d)$ and $k\in \n$:
\[
\norm{\partial^k_x \phi_\eps(D)h}_{H^s}\leq C_k \eps^{-k} \norm{h}_{H^s},
\]
and $C_k$ verifies for all $K>0$, $K^kC_k=o(k!)$.
\end{itemize}

Now by the symbolic calculus rules given by Theorem \ref{paracomposition_Notions of microlocal analysis_Paradifferential Calculus_symbolic calculus para precised} and the fact that the class $\Psi^0_{1,1}(\d)=S^0_{1,1}(\d)\cap \left(S^0_{1,1}(\d)\right)^*$ is closed under composition, for $\eps>0$, there exists ${}^\eps b^p_\tau \in \Psi^0_{1,1}(\d)$ such that:
\[
\sum_{k=0}^{+\infty}\frac{\tau^k}{k!}\mathfrak{L}^k_{iT_p}T_b \phi_\eps(D)=\op({}^\eps b^p_\tau), \text{ with,}
\]
\begin{equation}\label{BCH formula_def para hyperbolic flow_proof prop com-conjg_point II_lim eps seminorm est}
M_\rho^0({}^\eps b^p_\tau)\leq \sum_{k=0}^{+\infty} C^k C_k \frac{\abs{\tau}^k}{k!} \eps^{-k\delta-\beta} M_\rho^\beta(b) M_\rho^\delta(p)^{k}.
\end{equation}

In order to pass to the limit in $\eps$ we will express ${}^\eps b^p_\tau $ differently, for all $\eps>0$, $\displaystyle{\sum_{k=0}^n\frac{\tau^k}{k!}\mathfrak{L}^k_{iT_p}T_b \phi_\eps(D)}$ converges in $\mathscr{L}(H^s(\d))$, thus by uniqueness of the limit:
\begin{equation}\label{BCH formula_def para hyperbolic flow_proof prop com-conjg_point II_ eq lim eps sum}
e^{i \tau T_{p}} \circ  T_{b} \circ e^{-i\tau T_p} \phi_\eps(D)=\sum_{k=0}^{+\infty}\frac{\tau^k}{k!}\mathfrak{L}^k_{iT_p}T_b \phi_\eps(D).
\end{equation}
Thus, 
\begin{equation}\label{BCH formula_def para hyperbolic flow_proof prop com-conjg_point II_eq with new symb}
e^{i \tau T_{p}} \circ  T_{b} \circ e^{-i\tau T_p} \phi_\eps(D)=\op({}^\eps b^p_\tau).
\end{equation}

Now we estimate the ${}^{\delta-1}H_\rho^\beta(\cdot;k)_{k\in \n}$ norms of ${}^\eps b^p_\tau$. To do so we need, in the word of H\"ormander \cite{Hormander97}, a result which interpolates between information on the norm a of an operator and bounds for the derivatives of it's symbol. This was exactly the goal of Lemma \ref{BCH formula_def para hyperbolic flow_lem Beals seminorm estimates adapted}.

By commuting $\frac{1}{i}\frac{d}{dx}$ and $ix$  with \eqref{BCH formula_def para hyperbolic flow_proof prop com-conjg_point II_eq with new symb} we get: 
\begin{align*}
[\frac{1}{i}\frac{d}{dx},\op({}^\eps b^p_\tau)]=&[\frac{1}{i}\frac{d}{dx},e^{i \tau T_{p}}] \circ  T_{b} \circ e^{-i\tau T_p} \phi_\eps(D)+e^{i \tau T_{p}} \circ  [\frac{1}{i}\frac{d}{dx},T_{b} ]\circ e^{-i\tau T_p} \phi_\eps(D)\\
&+e^{i \tau T_{p}} \circ  T_{b}\circ [\frac{1}{i}\frac{d}{dx}, e^{-i\tau T_p}] \phi_\eps(D)+e^{i \tau T_{p}} \circ  T_{b} \circ e^{-i\tau T_p}[\frac{1}{i}\frac{d}{dx}, \phi_\eps(D)].
\end{align*}
and,
\begin{align*}
[ix,\op({}^\eps b^p_\tau)]&=[ix,e^{i \tau T_{p}}] \circ  T_{b} \circ e^{-i\tau T_p} \phi_\eps(D)+e^{i \tau T_{p}} \circ  [ix,T_{b} ]\circ e^{-i\tau T_p} \phi_\eps(D)\\
&+e^{i \tau T_{p}} \circ  T_{b} \circ[ix, e^{-i\tau T_p}] \phi_\eps(D)+e^{i \tau T_{p}} \circ  T_{b} \circ e^{-i\tau T_p}[ix, \phi_\eps(D)].
\end{align*}
To estimate $[\frac{1}{i}\frac{d}{dx},e^{i \tau T_{p}}]$ and $[ix,e^{i \tau T_{p}}]$ we get back to \eqref{BCH formula_def para hyperbolic flow_eq CP} and see that:
\begin{equation}\label{BCH formula_def para hyperbolic flow_proof prop com-conjg_point II_eq com derv with Atau}
\begin{cases}
[\frac{1}{i}\frac{d}{dx},e^{i \tau T_{p}}]=\int\limits_0^\tau e^{i(\tau-r)T_p}[\frac{1}{i}\frac{d}{dx},T_{ip}]e^{irT_p}dr,\\ [ix,e^{i \tau T_{p}}]=\int\limits_0^\tau e^{i(\tau-r)T_p}[ix,T_{ip}]e^{irT_p}dr.
\end{cases}
\end{equation}
Thus by iteration, the continuity of $e^{i \tau T_{p}}$ and Lemma \ref{BCH formula_def para hyperbolic flow_lem Beals seminorm estimates adapted} we get:
\[
{}^{1-\delta,\delta}H^\beta_{m,n}({}^\epsilon b_\tau^p;k)\leq C_{n,k}(M_1^\delta(p)) H^\beta_\rho(b;k) \left[\sum^{k-1}_{i=0}H^{\delta}_\rho(p;k)^{k-i}\right], \ (k,n,m)\in \n^3.
\]
Thus we can pass to the limit in $\eps$ in \eqref{BCH formula_def para hyperbolic flow_proof prop com-conjg_point II_eq with new symb}, there exist $b^p_\tau\in  W^{0,\infty} S^\beta_{1-\delta,\delta}$ such that:
 \begin{equation}\label{BCH formula_def para hyperbolic flow_proof prop com-conjg_point II_eq lim with new symb 0}
 e^{i \tau T_{p}} \circ T_b \circ e^{-i\tau T_p} =\op( b^p_\tau).
 \end{equation}
 Moreover if $\rho \in \n$ we get $b^p_\tau\in  W^{\rho,\infty} S^\beta_{1-\delta,\delta}$.
 
The ${}^{1-\delta,\delta}H^\beta_{m,\rho}(b^p_\tau;k),\rho \notin \n$ estimates are obtained by interpolation. Indeed by Proposition \ref{BCH formula_def para hyperbolic flow_def nonreg symb} the sequence of seminorms $$\displaystyle \left({}^{1-\delta,\delta}H^\beta_{m,n}(\cdot;k)\right)_{(k,m,n)\in \n^3} \text{ and }\left({}^{1-\delta,\delta}M^\beta_{m,n}(\cdot;k)\right)_{(k,m,n)\in \n^3}$$ are equivalent. Thus we deduce $\left({}^{1-\delta,\delta}M^\beta_{m,n}(\cdot;k)\right)$ estimates. Now for $\rho \in \r_+$, $\left({}^{1-\delta,\delta}M^\beta_{m,\rho}(\cdot;k)\right)$ are interpolation norms which give $\left({}^{1-\delta,\delta}M^\beta_{m,\rho}(\cdot;k)\right)$ estimates and the existence of $b^p_\tau\in  W^{\rho,\infty} S^\beta_{1-\delta,\delta}$ for $\rho \in \n$.

\paragraph*{Point $(III)$}
For point $(6)$ we compute:
\begin{align*}
\partial_\tau [e^{i \tau T_{p}}, T_b]&=[iT_{p}\circ e^{i \tau T_{p}} , T_b]=iT_{p}[ e^{i \tau T_{p}} , T_b]+[iT_{p} , T_b]e^{i \tau T_{p}} .
\end{align*}
Thus by the definition of $e^{i \tau T_{p}}$ as the flow map we get the following Duhamel formula,
\begin{align*}
[e^{i \tau T_{p}}, T_b]&=\int\limits_0^\tau e^{i(\tau-r)T_p}[iT_{p} , T_b]e^{irT_p}   dr,\\
&=e^{i \tau T_{p}}\int\limits_0^\tau \underbrace{e^{-irT_p} [iT_{p} , T_b]e^{irT_p} dr}_{\star}.
\end{align*}
Applying point $(5)$ to $\star$ we get: 
\begin{align*}
[e^{i \tau T_{p}}, T_b]&=e^{i \tau T_{p}}\sum_{k=1}^n(-1)^{k-1}\frac{\tau^k}{k!}\mathfrak{L}^k_{iT_p}T_b\\
&+(-1)^ne^{i \tau T_{p}}\int\limits_0^\tau \frac{(\tau -r)^n}{n!}e^{-irT_p}  \mathfrak{L}^{n+1}_{iT_p}T_be^{irT_p}dr.
\end{align*}
Again applying point $(1)$ combined with \eqref{BCH formula_def para hyperbolic flow_proof prop com-conjg_point I_lie derv cont est} we get \eqref{BCH formula_def para hyperbolic flow_prop com-conjg_est symbolic calc com}.

To get \eqref{BCH formula_def para hyperbolic flow_prop com-conjg_eq def com} we inject \eqref{BCH formula_def para hyperbolic flow_prop com-conjg_eq def cong} in $\star$, which concludes the proof.
\end{proof}

We give a result on the symbol of $e^{i \tau T_{p}}$
\begin{proposition}\label{BCH formula_lem est const lie derv_prop Atau symb}
Consider two real numbers $\delta < 1$, $\rho \geq 1$ and a real valued symbol $p\in \Gamma^{\delta}_{\rho}(\d)$.

 Let $e^{i \tau T_{p}},\tau \in \r$ be the flow map defined by Proposition \ref{BCH formula_def para hyperbolic flow}, then there exists  a symbol $e_\otimes^{i\tau p} \in W^{\rho,\infty}S^0_{1-\delta,\delta}(\d^*\times  \hat{\d})$ such that:
\begin{equation}\label{BCH formula_lem est const lie derv_Atau symb eq}
e^{i \tau T_{p}}=\op(e_\otimes^{i\tau p}).
\end{equation}
Moreover we have the identity: 
%\begin{equation}\label{BCH formula_lem est const lie derv_Atau symb ident 1}
%\begin{cases}
%\partial_\tau [T^{lim}_{e_\otimes^{i\tau p}}h_0]=i T_p T^{lim}_{e_\otimes^{i\tau p}}h_0,\\
%T^{lim}_{e_\otimes^{i\tau p}}h_0{ \ }_{|_{\tau=0}}=T_1 h_0.
%\end{cases}
%\text{for $h_0 \in H^s(\d),s\in\r$.}
%\end{equation}
\begin{equation}\label{BCH formula_lem est const lie derv_Atau symb ident 2}
e^{i\tau T_p}=Id+T_{e^{i\tau p}-1}+\int\limits_0^\tau e^{i(\tau-s)T_p}\big(T_{ip}T_{e^{is p}}-T_{ipe^{is p}} \big)ds.
\end{equation}
\end{proposition}

\begin{proof}
The idea is that morally $\op(e_\otimes^{i\tau p})$ should be defined by the asymptotic series: 
\[
\op(e_\otimes^{i\tau p})\sim\sum \frac{i^k\tau^k}{k!}(T_p)^k=\sum \frac{i^k\tau^k}{k!}T_{\overset{k}{\otimes} p},
\]
where $\overset{k}{\otimes} p$ is defined by Theorem \ref{paracomposition_Notions of microlocal analysis_Paradifferential Calculus_symbolic calculus para precised}. To make this series converge we again introduce the Gaussian multiplier $\phi_\eps(D)$ defined by \eqref{BCH formula_def para hyperbolic flow_proof prop com-conjg_point II_def gauss moll}, as we still have the factor $\frac{1}{k!}$ as in Proposition \ref{BCH formula_def para hyperbolic flow_prop com-conjg}. 

As in the proof of Proposition \ref{BCH formula_def para hyperbolic flow_prop com-conjg} there exists a symbol ${}^\eps e_\otimes^{i\tau p} \in \Psi^0_{1,1}(\d)$ such that:
\begin{equation}\label{BCH formula_lem est const lie derv_proof prop Atau symb_eq 1}
\sum^{+\infty}_{k=0} \frac{i^k\tau^k}{k!}(T_p)^k\phi_\eps(D)=\sum^{+\infty}_{k=0} \frac{i^k\tau^k}{k!}T_{\overset{k}{\otimes} p}\phi_\eps(D)=\op({}^\eps e_\otimes^{i\tau p}),
\end{equation}
\begin{equation}\label{BCH formula_lem est const lie derv_proof prop Atau symb_eq 2}
{}^{1-\delta,\delta}M^0_\rho({}^\eps e_\otimes^{i\tau p})\leq \sum_{k=0}^{+\infty} C^k C_k \frac{\abs{\tau}^k}{k!} \eps^{-k\delta} M_\rho^\delta(p)^{k},
\end{equation}
where $C_k$ verifies for all $K>0$, $K^kC_k=o(k!)$.

Now in order to pass to the limit in $\eps$ we need to get uniform estimates on ${}^{1-\delta,\delta}H^0_{n,s}({}^\eps e_\otimes^{i\tau p};k)$ for $(k,n)\in \n$. To do so we see that:
\begin{equation}\label{BCH formula_lem est const lie derv_proof prop Atau symb_eq 3}
\begin{cases}
\partial_\tau [\op({}^\eps e_\otimes^{i\tau p})h_0]=i T_p \op({}^\eps e_\otimes^{i\tau p})h_0,\\
\op({}^\eps e_\otimes^{i\tau p})h_0{ \ }_{|_{\tau=0}}=\phi_\eps(D)T_1h_0,
\end{cases}
\text{for $h_0 \in H^s(\d),s\in\r$.}
\end{equation} 
Thus a standard energy estimate combined with the commutation identities \eqref{BCH formula_def para hyperbolic flow_proof prop com-conjg_point II_eq com derv with Atau} and Lemma \ref{BCH formula_def para hyperbolic flow_lem Beals seminorm estimates adapted} we get:
\[
{}^{1-\delta,\delta}H^0_{m,n}({}^\eps e_\otimes^{i\tau p};k)_{k\in \n}\leq C_{k,n}(M_1^\delta(p;k))M_n^\delta(p;k), \ (k,m,n)\in \n^3.
\]
Thus if of $\rho \in \n$ we can pass to the limit in $\eps$ and get $e_\otimes^{i\tau p} \in W^{\rho,\infty}S^0_{1-\delta,\delta}(\d^*\times  \hat{\d})$. For $\rho \in \r_+\setminus\n$ follows by interpolation as in point $(II)$ of the proof of Proposition \ref{BCH formula_def para hyperbolic flow_prop com-conjg}.
%In order to get higher Zygmund estimates we see that by getting back to the sum \eqref{BCH formula_lem est const lie derv_proof prop Atau symb_eq 1}, we have:
%\begin{equation}\label{BCH formula_lem est const lie derv_proof prop Atau symb_eq 4}
%\begin{cases}
%\partial_\tau [T^{lim}_{P_{\leq k}(D)[{}^\eps e_\otimes^{i\tau p}]}h_0]=i T_{[P_{\leq k}(D)p]} T^{lim}_{[P_{\leq k}(D){}^\eps e_\otimes^{i\tau p}]}h_0,\\
%T^{lim}_{[P_{\leq k}(D)e_\otimes^{i\tau p}]}h_0{ \ }_{|_{\tau=0}}=\phi_\eps(D)h_0.
%\end{cases}
%\text{for $h_0 \in H^s(\d),s\in\r$.}
%\end{equation} 
%As previously commuting with $ix$ and $\frac{1}{i}\frac{d}{dx}$ and Proposition \ref{paracomposition_Notations and functional analysis_proposition Zygmund spaces on balls} we get:
%\[
%{}^{1-\delta,\delta}H^m_s({}^\eps e_\otimes^{i\tau p};k)_{k\in \n}\leq C_{k,n}(M_1^\delta(p;k))M_s^\delta(p;k), \ (k,n)\in \n.
%\]
%Thus we can pass to the limit in when $\eps \rightarrow 0$ and get the desired result.

%Finally for identity \eqref{BCH formula_lem est const lie derv_Atau symb ident 1} we pass to the limit in \eqref{BCH formula_lem est const lie derv_proof prop Atau symb_eq 3}. 
Identity \eqref{BCH formula_lem est const lie derv_Atau symb ident 2} comes from the following computation. Fix an $h_0\in H^s, s\in \r$, then $[e^{i \tau T_{p}}-T_{e^{i\tau p}}]h_0$ solves:
\begin{equation}\label{BCH formula_lem est const lie derv_proof prop Atau symb_CP}
\begin{cases} 
\partial_\tau \big([e^{i \tau T_{p}}-T_{e^{i\tau p}}]h_0 \big)-iT_{p} \big( [e^{i \tau T_{p}}-T_{e^{i\tau p}}]h_0 \big)=\big(T_{ip}T_{e^{i\tau p}}-T_{ipe^{i\tau p}} \big)h_0, \\
\big([e^{i \tau T_{p}}-T_{e^{i\tau p}}]h_0\big)(0,\cdot)=(Id-T_1)h_0(\cdot),
\end{cases}
\end{equation}
which by definition of $e^{i \tau T_{p}}$ gives \eqref{BCH formula_lem est const lie derv_Atau symb ident 2}.
\end{proof}

We will now compute the different Gateaux derivatives of the operators defined above.
\begin{proposition}\label{BCH formula_prop diff Atau param}
Consider two real numbers $\delta < 1$, $\rho \geq 1$, two real valued symbols $p,p' \in \Gamma^{\delta}_{\rho}(\d)$. Let $e^{i \tau T_{p}},e^{i\tau T_{p'}},\tau \in \r$ be the flow maps defined by Proposition \ref{BCH formula_def para hyperbolic flow}, then for $\tau \in \r$ we have:
\begin{equation}\label{BCH formula_prop diff Atau param_eq1}
 e^{i \tau T_{p}}-e^{i\tau T_{p'}}=\int\limits_0^\tau e^{i(\tau-r)T_p}iT_{p'- p}e^{ir T_{p'}}dr.
\end{equation}
Another way to see this is with the Gateaux derivative of $p \mapsto e^{i \tau T_{p}}$ on the Fr\'echet space $\Gamma^{\delta}_{\rho}(\d)$ is given by:
\begin{equation}\label{BCH formula_prop diff Atau param_eq2}
D_p e^{i \tau T_{p}}(h)=\int\limits_0^\tau e^{i(\tau-r)T_p}T_{ih}e^{irT_p}dr.
\end{equation}

Moreover consider an open interval $I\subset \r$, and a real valued symbols $p\in C^1(I,\Gamma^{\delta}_{\rho}(\d))$. Let $e^{i \tau T_{p}},\tau \in \r$ be the flow map defined by Proposition \ref{BCH formula_def para hyperbolic flow} then for $\tau \in \r, z\in I$ we have:
\begin{equation}\label{BCH formula_prop diff Atau param_eq3}
\partial_z e^{i \tau T_{p}}=\int\limits_0^\tau e^{i(\tau-r)T_p}T_{i\partial_z p}e^{irT_p}dr.
\end{equation}
\end{proposition}
\begin{proof}
Fix $h_0\in H^s,s\in \r$ then:
\[
\partial_\tau [e^{i \tau T_{p}} h_0]-iT_p[e^{i \tau T_{p}} h_0]=0\Rightarrow \partial_\tau [\partial_z e^{i \tau T_{p}} h_0]-iT_p[\partial_z e^{i \tau T_{p}} h_0]-T_{i\partial_z p}[e^{i \tau T_{p}} h_0]=0,
\]
which gives \eqref{BCH formula_prop diff Atau param_eq3} by the definition of $e^{i \tau T_{p}}$ and the Duhamel formula. The identities \eqref{BCH formula_prop diff Atau param_eq1} and \eqref{BCH formula_prop diff Atau param_eq2} are obtained in the same way.
\end{proof}
\begin{proposition}\label{BCH formula_prop diff Atau conjg}
Consider two real numbers $\delta < 1$, $\rho \geq 1$, two real valued symbols $p,p' \in \Gamma^{\delta}_{\rho}(\d)$. Let $e^{i \tau T_{p}},e^{i\tau T_{p'}},\tau \in \r$ be the flow maps defined by Proposition \ref{BCH formula_def para hyperbolic flow} and take a symbol $b\in \Gamma^{\beta}_{\rho}(\d)$ with $\beta \in \r$ then for $\tau \in \r$ we have:
\begin{align}
\op(b^p_\tau)-\op(b^{p'}_\tau)&=\int\limits_0^\tau e^{i(\tau-r)T_p} \mathfrak{L}_{iT_{p-p'}}\op(b^{p'}_r) e^{i(r-\tau)T_p}dr \label{BCH formula_prop diff Atau conjg_eq1}\\
&=i\int\limits_0^\tau \mathfrak{L}_{T_p-\op({p'}^p_{\tau-r})}\op((b^{p'}_r)_{\tau-r}^p) dr.\label{BCH formula_prop diff Atau conjg_eq2}
\end{align}
Another way to see this is with the Gateaux derivative of $p \mapsto \op(b^p_\tau)$ on the Fr\'echet space $\Gamma^{\delta}_{\rho}(\d)$ is given by:
\begin{equation}\label{BCH formula_prop diff Atau conjg_eq3}
D_p\op(b^p_\tau)(h)=\int\limits_0^\tau\mathfrak{L}_{i \op(h^p_{\tau-r})}\op(b^{p}_\tau)dr=\mathfrak{L}_{i\int\limits_0^\tau \op(h^p_{\tau-r})dr}\op(b^{p}_\tau).
\end{equation}
Writing,
$\op({}^c b_\tau^p)=T_b-\op(b_{-\tau}^p)$, and, $\op({}^c b_\tau^{p'})=T_b-\op(b_{-\tau}^{p'})$ we get:
\begin{align}
\op({}^c b_\tau^p)-\op({}^c b_\tau^{p'})&=-\int\limits_0^{-\tau}e^{-i(\tau+r)T_p} \mathfrak{L}_{iT_{p-p'}}\op(b^{p'}_r) e^{i(\tau+r)T_p}dr \label{BCH formula_prop diff Atau conjg_eq4}\\
&=-i\int\limits_0^{-\tau} \mathfrak{L}_{T_p-\op({p'}^p_{-\tau-r})}\op((b^{p'}_r)_{-\tau-r}^p) dr.\label{BCH formula_prop diff Atau conjg_eq5}
\end{align}
\begin{equation}\label{BCH formula_prop diff Atau conjg_eq6}
D_p\op({}^c b_\tau^p)(h)=-\int\limits_0^{-\tau}\mathfrak{L}_{i \op(h^p_{-\tau-r})}\op(b^{p}_{-\tau})dr=-\mathfrak{L}_{i\int\limits_0^{-\tau} \op(h^p_{-\tau-r}dr)}\op(b^{p}_{-\tau}).
\end{equation}
\end{proposition}
\begin{proof}
From \eqref{BCH formula_def para hyperbolic flow_proof prop com-conjg_point II_edo def conjg rearenged} and \eqref{BCH formula_def para hyperbolic flow_proof prop com-conjg_point II_gen edo} we have:
\begin{equation}\label{BCH formula_prop diff Atau conjg_proof_edo diff Atau}
\begin{cases}
\partial_\tau [\op(b^p_\tau)-\op(b^{p'}_\tau)]=\mathfrak{L}_{iT_p}(\op(b^p_\tau)-\op(b^{p'}_\tau))+\mathfrak{L}_{iT_{p-p'}}\op(b^{p'}_\tau),\\
\op(b^p_0)-\op(b^{p'}_0)=0.
\end{cases}
\end{equation}
Thus the Duhamel formula gives \eqref{BCH formula_prop diff Atau conjg_eq1} and \eqref{BCH formula_prop diff Atau conjg_eq2}.
For the Gateaux derivative passing to the limit in \eqref{BCH formula_prop diff Atau conjg_eq1} we have:
\[
D_p\op(b^p_\tau)(h)=\int\limits_0^\tau e^{i(\tau-r)T_p} \mathfrak{L}_{iT_{h}}\op(b^{p}_r) e^{i(r-\tau)T_p}dr,
\]
which gives \eqref{BCH formula_prop diff Atau conjg_eq3}.
\end{proof}
\begin{corollary}\label{BCH formula_cor reste gaue trans}
Consider three real numbers $\alpha>1$, $\beta< \alpha $ and $s\in \r$, two real valued symbols $a \in \Gamma^{\alpha}_{\lceil \frac{\alpha}{\alpha-\beta}\rceil}(\d)$ 
and $b \in \Gamma^{\beta}_{\lceil \frac{\alpha}{\alpha-\beta}\rceil-1}(\d)$.
Suppose that there exists a real valued symbol $p \in \Gamma^{\beta+1-\alpha}_{\lceil \frac{\alpha}{\alpha-\beta}\rceil}(\d)$ such that: 
\begin{equation}\label{BCH formula_cor reste gaue trans_eq1}
b=-\partial_\xi p \partial_x a+\partial_x p \partial_\xi a.
\end{equation}
Define $e^{i \tau T_{p}}(u)$ as the flow map generated by $iT_p$ from Proposition \ref{BCH formula_def para hyperbolic flow}. For $\tau \in \r$, let
\begin{equation}\label{BCH formula_cor reste gaue trans_eq2}
R_{\tau}=\tau T_{ib}+\int\limits_0^\tau e^{-isT_p} [T_{ip},T_{ia}]e^{isT_p}ds,
\end{equation} 
and,
\begin{align}\label{BCH formula_cor reste gaue trans_eq3}
\tilde{R_\tau}=e^{i \tau T_{p}} R_{\tau} e^{-i\tau T_p}
=\tau e^{i \tau T_{p}} iT_{b} e^{-i\tau T_p}+[e^{i \tau T_{p}},T_{ia}]e^{-i\tau T_p}.
\end{align}

Then $R_{\tau},\tilde{R_\tau}\in \mathscr{L}(H^{s+(\beta +1-\alpha)^+}(\d),H^s(\d))$ and
$$ \norm{(R_{\tau},\tilde{R_\tau})}_{H^{s+(\beta +1-\alpha)^+} \rightarrow H^s}\leq C \underbrace{M^{\alpha}_{\lceil \frac{\alpha}{\alpha-\beta}\rceil}(a)M^{\beta}_{\lceil \frac{\alpha}{\alpha-\beta}\rceil-1}(b)M^{\beta+1-\alpha}_{\lceil \frac{\alpha}{\alpha-\beta}\rceil}(p)}_{:=F(a,b,p)}. $$
Moreover taking three different symbols $a', b'$ and $p'$ and defining analogously $R_{\tau}',\tilde{R}_\tau'$, we have for $h\in H^s$:
\[\norm{[R_{\tau}-R_\tau']h}_{H^s}\leq C F(a,b,p) F(a',b',p') F(a-a',b-b',p-p') \norm{h}_{H^{s+(\beta +1-\alpha)^+}},
\]
where $F$ is defined in \eqref{BCH formula_cor reste gaue trans_eq3} and,
\[\norm{[\tilde{R}_{\tau}-\tilde{R}_\tau']h}_{H^s}\leq C F(a,b,p) F(a',b',p') F(a-a',b-b',p-p')  \norm{h}_{H^{s+(\beta +1-\alpha)^+}}.
\]
\end{corollary}
\begin{proof}
First we notice that by definition $R_{\tau},\tilde{R_\tau}$ are of order $\beta$. Now to show that they are actually of order $(\beta+1-\alpha)^+$ indeed as $p$, $a$ and $b$ have a regularity of $2+\frac{2-\alpha}{\alpha-1}$ and $\lceil \frac{\alpha}{\alpha-1}\rceil-1$ respectively we write by \eqref{BCH formula_def para hyperbolic flow_proof prop com-conjg_point I_iterated int formula conjg} for $n=\lceil \frac{\alpha}{\alpha-1}\rceil-1 $ :
\[
R_{\tau}=\tau \underbrace{\left(T_{ib}+[T_{ip},T_{ia}] \right)}_{:=T_{r_{\tau}}} + \underbrace{\sum_{k=1}^{n-1}\frac{(-1)^k\tau^{k+1}}{(k+1)!}\mathfrak{L}^{k+1}_{iT_p}T_{ia}}_{:=T_{r_{\tau}^{(\beta +1-\alpha)}}}\underbrace{-\int\limits_0^\tau \frac{(\tau -r)^{n+1}}{(n+1)!}e^{irT_p} \mathfrak{L}^{n+1}_{iT_p}T_{ia}e^{-irT_p} dr}_{:=R^{\alpha+n(1-\alpha)}_{\tau}}.
\]
The key observation is that by Theorem \ref{paracomposition_Notions of microlocal analysis_Paradifferential Calculus_symbolic calculus para precised}, $T_{r_{\tau}}$ and $T_{r_{\tau}^{(\beta +1-\alpha)}}$ are of order $\beta +1-\alpha$ and $R^{\alpha+n(1-\alpha)}_\tau$ is order $\alpha+n(1-\alpha)$. Now $T_{r_{\tau}}$ and $T_{r_{\tau}^{(\beta +1-\alpha)}}$ are operators in the usual paradifferential classes and thus their differential with respect to $p$ do not generate the undesired loss of $1+\beta-\alpha$ derivatives. On the other hand for $D_p R^{\alpha+n(1-\alpha)}_\tau$ is of order $\beta +1-\alpha$ by Proposition \ref{BCH formula_prop diff Atau conjg}. Which gives the desired estimates on $R_\tau$. We get the estimates on $\tilde{R}_\tau$ by writing analogously 
\begin{multline*}
\tilde{R}_{\tau}=\tau \left(T_{ib}+[T_{ip},T_{ia}] \right)+ \sum_{k=1}^{n-1}\frac{\tau^{k+1}}{k!}\mathfrak{L}^{k}_{iT_p}T_{ib}-\int\limits_0^\tau \frac{(\tau -r)^n}{n!}e^{irT_p} \mathfrak{L}^{n}_{iT_p}T_{ib}e^{-irT_p} dr\\ + \sum_{k=1}^{n-1}\frac{\tau^{k+1}}{(k+1)!}\mathfrak{L}^{k+1}_{iT_p}T_{ia}-\int\limits_0^\tau\int\limits_0^{s} \frac{(\tau-s -r)^n}{n!}e^{irT_p} \mathfrak{L}^{n+1}_{iT_p}T_{ia}e^{-irT_p} drds.
\end{multline*}

\end{proof}

\appendix
\section{Paradifferential Calculus}\label{paracomposition_Notions of microlocal analysis_Paradifferential Calculus}
In this paragraph we review classic notations and results about paradifferential and pseudodifferential calculus that we need in this paper.
We follow the presentations in \cite{Hormander71}, \cite{Hormander97}, \cite{Taylor07}, and \cite{Metivier08} which give an accessible and complete presentation. 
\begin{notation}
 In the following presentation we will use the usual definitions and standard notations for regular functions $C^k$, $C^k_b$ for bounded ones and $C^k_0$ for those with compact support,
  the distribution space $\dr'$, $\er'$ for those distribution with compact support, 
  $\dr'^k$,$\er'^k$ for distributions of order k, $L^p$ Lebesgue spaces, $H^s$ and $W^{p,q}$Sobolev spaces
  and finally $\sr$ for the Schwartz class  and it's dual $\sr'$. All of those spaces are equipped with their standard topologies. We also use the \textit{Landau notation}  $O_{\norm{ \ }}(X)$.
\end{notation}

For the definition of the periodic symbol classes we will need the following definitions and notations.
\begin{notation}
We will use $\d$ to denote $\t$ or $\r$ and $\hat{\d}$ to denote their duals that is $\z$ in the case of $\t$ and $\r$ in the case of $\r$. For concision an integral on $\z$ i.e $\displaystyle \int\limits_\z$ should be understood as $\displaystyle \sum_\z$. A function $a$ is said to be in $ C^\infty(\t \times \z)$ if for every $\xi \in \z$, $a(\cdot,\xi) \in C^\infty(\t)$. The partial derivative $\partial_\xi$ should be understood as the forward difference operator, i.e
\[\partial_\xi a(\xi)=a(\xi+1)-a(\xi),\ \xi \in \z.\]
We recall the following simple identities for the Fourier transform on the Torus:
\[
\begin{cases}
\fr_{\t}(\partial_x^\alpha f)(\xi)=\xi^\alpha \fr_{\t}(f)(\xi), \xi \in \z,\\
\fr_{\t}((e^{-2i\pi x}-1)^\alpha f)(\xi)=\xi^\alpha \fr_{\t}(f)(\xi), \xi \in \z.
\end{cases}
\]
 \end{notation}

\subsection{Littlewood-Paley Theory}\label{paracomposition_section Notations and functional analysis}
\begin{definition}[Littlewood-Paley decomposition]\label{paracomposition_section Notations and functional analysis_def LP Theory}
Pick $P_0\in C^\infty_0(\r^d)$ so that: $$P_0(\xi)=1 \text{ for } \abs{\xi}<1 \text{ and } P_0(\xi)=0 \text{ for } \abs{\xi}>2. $$ We define a dyadic decomposition of unity by:
\[ \text{for } k \geq 1, \ P_{\leq k}(\xi)=P_0(2^{-k}\xi), \ P_k(\xi)=P_{\leq k}(\xi)-P_{\leq k-1}(\xi). \]
 Thus,\[ P_{\leq k}(\xi)=\sum_{0\leq j \leq k}P_j(\xi) \text{ and } 1=\sum_{j=0}^\infty P_j(\xi). \]
 Introduce the operator acting on $\mathscr S '(\d^d)$: 
 \[P_{\leq k}u=\fr^{-1}(P_{\leq k}(\xi)u) \text{ and } u_k=\fr^{-1}(P_k(\xi)u).\]
 Thus,
 \[u=\sum_k u_k.\]
 Finally put for $k\geq 1, C_k=\supp \ P_k$ the set of rings associated to this decomposition.
\end{definition}

An interesting property of the Littlewood-Paley decomposition is that even if the decomposed function is merely a distribution the terms of the decomposition are regular, indeed they all have compact spectrum and thus are entire functions. On classical functions spaces this regularization effect can be "measured" by the following inequalities due to Bernstein.

\begin{proposition}[Bernstein's inequalities]\label{paracomposition_Notations and functional analysis_bernstein1}
Suppose that $a\in L^p(\d^d)$ has its spectrum contained in the ball $\set{\abs{\xi}\leq \lambda}$. 

Then $a\in C^\infty$ and for all $\alpha \in  \n^d$ and $1\leq p \leq q \leq +\infty$, there is $C_{\alpha,p,q}$ (independent of $\lambda$) such that 
\[\norm{\partial^{\alpha}_x a}_{L^q} \leq C_{\alpha,p,q} \lambda^{\abs{\alpha}+\frac{d}{p}-\frac{d}{q}}\norm{a}_{L^p}.\]
In particular,
\[\norm{\partial^{\alpha}_x a}_{L^q} \leq C_{\alpha} \lambda^{\abs{\alpha}}\norm{a}_{L^p}, \text{ and for $p=2$, $p=\infty$}\]
\[\norm{a}_{L^\infty}\leq C \lambda^{\frac{d}{2}} \norm{a}_{L^2}.\]
If moreover a has its spectrum is included in $ \set{0<\mu \leq \abs{\xi}\leq \lambda}$ then:
\[
 C_{\alpha,q}^{-1} \mu^{\abs{\alpha}}\norm{a}_{L^q}\leq \norm{\partial^{\alpha}_x a}_{L^q} \leq C_{\alpha,q} \lambda^{\abs{\alpha}}\norm{a}_{L^q}.
\]
\end{proposition}

\begin{proposition}\label{paracomposition_Notations and functional analysis_bernstein2}
For all $\mu >0$, there is a constant $C$ such that for all $\lambda>0$ and for all $\alpha \in W^{\mu,\infty}$ with spectrum contained in $\set{\abs{\xi}\geq \lambda}$. one has the following estimate: 
\[\norm{a}_{L^\infty}\leq C \lambda^{-\mu} \norm{a}_{W^{\mu,\infty}}.\]
\end{proposition}

\begin{definition}[Zygmund spaces on $\d^d$]\label{paracomposition_Notations and functional analysis_def Zygmund spaces on r}
For $r\in \r$ we define the space:
\[ C^r_*(\d^d) \subset \sr'(\d^d), \ C^r_*(\d^d)=\set{u\in\sr'(\d^d),\norm{u}_{C_*^r}=\sup_q 2^{qr}\norm{u_q}_{L^\infty}<\infty}\]
 equipped with its canonical topology giving it a Banach space structure.\\
 It's a classical result that for $r\notin \n$, $C^r_*(\d^d)=W^{r,\infty}(\d^d)$ the classic H{\"o}lder spaces.
\end{definition}

\begin{proposition} \label{paracomposition_Notations and functional analysis_proposition Zygmund spaces on balls}%%%%%%%
Let $\b$ be a ball with center 0. There exists a constant C such that for all $r>0$ and for all $(u_q)_{q\in \n}$ in $\sr'(\d^d)$ verifying:
\[\forall q ,\supp  \hat{u}_q \subset  2^q \b  \text{ and }  (2^{qr}\norm{u_q}_\infty)_{q\in \n} \text{ is  bounded,} \]
\[\text{then}, u=\sum_q u_q \in C^r_*(\d^d) \text{ and } \norm{u}_{C_*^r} \leq \frac{C}{1-2^{-r}} \displaystyle{\sup_{q \in \n}}\ 2^{qr}\norm{u_q}_{L^\infty}. \]
\end{proposition}

\begin{definition}[Sobolev spaces on $\d^d$]
It is also a classical result that for $s\in \r$ :
\[H^s(\d^d)=\set{u\in\sr'(\d^d),\norm{u}_s= \bigg(\sum_q 2^{2qs} {\norm{u_q}_{L^2}}^2 \bigg)^{\frac{1}{2}}<\infty}\]
 with the right hand side equipped with its canonical topology giving it a Hilbert space structure and $\norm{\ }_s$ is equivalent to the usual norm on $H^s$.
\end{definition}

\begin{proposition} \label{paracomposition_Notations and functional analysis_proposition Sobolev spaces on balls}%%%%%%%
Let $\b$ be a ball with center 0. There exists a constant C such that for all $s>0$ and for all $(u_q)_{\in \n}$ in $\sr'(\d^d)$ verifying:
\[\forall q , \ \supp \hat{u_q} \subset  2^q \b \text{ and } (2^{qs}\norm{u_q}_{L^2})_{q\in \n} \text{ is in }  L^2(\n), \]
\[\text{then}, \ u=\sum_q u_q \in H^s(\d^d) \text{ and } \norm{u}_s \leq \frac{C}{1-2^{-s}} \bigg(\sum_q 2^{2qs} {\norm{u_q}_{L^2}}^2 \bigg)^{\frac{1}{2}}. \]
\end{proposition}

We recall the usual nonlinear estimates in Sobolev spaces:
\begin{itemize}
\item If $u_j\in H^{s_j}(\d^d), j=1,2$, and $s_1+s_2>0$ then $u_1u_2 \in H^{s_0}(\d^d)$ and if
\[ s_0\leq s_j, j=1,2 \text{ and } s_0\leq s_1+s_2-\frac{d}{2}, \ \ \  \]
\[\text{then }\exists K\in \r,  \norm{u_1u_2}_{H^{s_0}}\leq K \norm{u_1}_{H^{s_1}}\norm{u_2}_{H^{s_2}},\]
where the last inequality is strict if $s_1$ or $s_2$  or $-s_0$ is equal to $\frac{d}{2}$.
\item For all $C^\infty$ function $F$ vanishing at the origin, if $u \in H^s(\d^d)$ with $s>\frac{d}{2}$, then
\[ \norm{F(u)}_{H^s} \leq C(\norm{u}_{H^s}),\]
for some non decreasing  function $C$ depending only on $F$.
\end{itemize}

\subsection{Paradifferential operators}\label{paracomposition_Notions of microlocal analysis_Paradifferential Calculus}
We start by the definition of symbols with limited spatial regularity. Let $\w\subset \sr'$ be a Banach space.
\begin{definition}\label{paracomposition_Notions of microlocal analysis_Paradifferential Calculus_ def para symbol}
Given $m \in \r$, $\Gamma^m_\w(\d)$ denotes the space of locally bounded functions $a(x,\xi)$ on $\d\times (\hat{\d} \setminus 0)$, which are $C^\infty$ with respect to $\xi$ for $\xi \neq 0$ and such that, for all $\alpha \in \n$ and for all $\xi \neq 0$, the function $x \mapsto \partial^\alpha_\xi a(x,\xi)$ belongs to $\w$ and there exists a constant $C_\alpha$ such that, for all $\eps>0$:
\begin{equation}\label{paracomposition_Notions of microlocal analysis_Paradifferential Calculus_ definition growth xi condition para} 
\forall \abs{\xi}>\eps, \norm{\partial^\alpha_\xi a(.,\xi)}_{\w}\leq C_{\alpha,\eps} (1+\abs{\xi})^{m-\abs{\alpha}}. 
\end{equation}
The spaces $\Gamma^m_\w(\d)$ are equipped with their natural Fr\'echet topology induced by the semi-norms defined by the best constants in \eqref{paracomposition_Notions of microlocal analysis_Paradifferential Calculus_ definition growth xi condition para} .
%We will essentially work with $\w=W^{\rho,\infty}$ and write $\Gamma^m_\w=\Gamma^m_\rho$, for $\rho<0$ we use $\w=C_*^\rho$.
\end{definition}

For quantitative estimates we introduce as in \cite{Metivier08}:
\begin{definition}\label{paracomposition_Notions of microlocal analysis_Paradifferential Calculus_ definition semi-norms}
For $m\in \r$ and $a \in \Gamma^m_\w(\d)$, we set
\[M^m_\w(a;n)=\sup_{\abs{\alpha}\leq n} \ \sup_{\abs{\xi}\geq\frac{1}{2}}\norm{(1+\abs{\xi})^{m-\abs{\alpha}}\partial^\alpha_\xi a(.,\xi)}_{\w}, \text{ for } n\in \n.\]
%The family of semi-norms $M^m_\w(a;n)$ induces a natural Fr\'echet space structure on $\Gamma^m_\w(\d)$.
For $\w=W^{\rho,\infty},\rho \geq 0$, we write: 
\[
\Gamma^m_{W^{\rho,\infty}}(\d)=\Gamma^m_\rho(\d) \text{ and }
M^m_\rho(a)=M^m_{W^{\rho,\infty}}(a;1).
\]
Moreover we introduce the following spaces equipped with their natural Fr\'echet space structure:
\[
C^{\infty}_b(\d)=\cap_{\rho \geq 0}W^{\rho,\infty}, \ \Gamma^m_\infty(\d)=\cap_{\rho \geq 0}\Gamma^m_\rho(\d), \ \Gamma^{-\infty}_\rho(\d)=\cap_{m\in \r}\Gamma^m_\rho(\d) \text{ and,}
\]
\[
\Gamma^{-\infty}_\infty(\d)=\cap_{\rho \geq 0}\cap_{m\in \r}\Gamma^m_\rho(\d). 
\]
\end{definition} 

In higher dimensions the $1$ in the definition of $M^m_\rho$ should be replaced by $1+\lfloor \frac{d}{2}\rfloor$.

\begin{definition}
Define an admissible cutoff function as a function $\psi^{B,b}\in C^\infty(\hat{D}^2)$,  $B>1,b>0$ that verifies:
\begin{enumerate}
\item 
\[
\psi^{B,b}(\xi,\eta)=0 \text{ when }
\abs{\xi}< B\abs{\eta}+b,
\text{ and }
\psi^{B,b}(\eta,\xi)=1 \text{ when } \abs{\xi}>B\abs{\eta}+b+1.
\]
\item for all $(\alpha,\beta)\in \n^d \times \n^d,$ there is $C_{\alpha_\beta}$, with $C_{0,0}\leq 1$, such that:
\begin{equation}\label{paracomposition_Notions of microlocal analysis_Paradifferential Calculus_definition cutoff growth hypothesis}
\forall(\xi,\eta): \abs{\partial_\xi^\alpha \partial_\eta^\beta \psi^{B,b}(\xi,\eta)}\leq C_{\alpha,\beta} (1+\abs{\xi})^{-\abs{\alpha}-\abs{\beta}}.
\end{equation}
\end{enumerate}
\end{definition}

\begin{definition}
Consider a real numbers $m\in \r$, a symbol $a\in \Gamma^m_\w$ and an admissible cutoff function $\psi^{B,b}$ define the paradifferential operator $T_a$ by:
\[\widehat{T_a u}(\xi)=(2\pi)^{-1}\int\limits_{\hat{\d}}\psi^{B,b}(\xi-\eta,\eta)\hat{a}(\xi-\eta,\eta)\hat{u}(\eta)d\eta,\]
where $\hat{a}(\eta,\xi)=\int e^{-ix\cdot\eta}a(x,\xi)dx$ is the Fourier transform of $a$ with respect to the first variable. 
In the language of pseudodifferential operators:
\[T_a u=\op(\sigma_a)u, \text{ where } \fr_x\sigma_a(\xi,\eta)=\psi^{B,b}(\xi,\eta) \fr_x a(\xi,\eta).\]
\end{definition}

An important property of paradifferential operators is their action on functions with localized spectrum.
\begin{lemma}\label{paracomposition_Notions of microlocal analysis_Paradifferential Calculus_propostion para action spectrum}
Consider two real numbers $m\in \r$, $\rho\geq 0$, a symbol $a\in \Gamma^m_0(\d)$, an admissible cutoff function $\psi^{B,b}$ and $u \in \sr(\d^d)$.
\begin{itemize}
\item For $R>>b$, if $\supp \fr u \subset \set{\abs{\xi}\leq R},$ then: 
\begin{equation}\label{paracomposition_Notions of microlocal analysis_Paradifferential Calculus_propostion para action spectrum on rings}
\supp \fr T_a u \subset \set{\abs{\xi}\leq (1+\frac{1}{B})R-\frac{b}{B}},
\end{equation}
\item For $R>>b$, if $\supp \fr u \subset \set{\abs{\xi}\geq R},$ then: 
\begin{equation}\label{paracomposition_Notions of microlocal analysis_Paradifferential Calculus_propostion para action spectrum on balls}
\supp \fr T_a u \subset \set{\abs{\xi}\geq (1-\frac{1}{B})R+\frac{b}{B}},
\end{equation}
\end{itemize}
\end{lemma}

The main features of symbolic calculus for paradifferential operators are given by the following theorems taken from \cite{Metivier08} and \cite{Ayman18}.
\begin{theorem}\label{paracomposition_Notions of microlocal analysis_Paradifferential Calculus_para continuity}
Let $m \in \r$. if $a\in \Gamma^m_0(\d)$, then $T_a$ is of order m. Moreover, for all $\mu \in \r$ there exists a constant K such that:
\[\norm{T_a}_{H^\mu \rightarrow H^{\mu-m}}\leq K M^m_0(a),\text{ and,}\]
\[\norm{T_a}_{W^{\mu,\infty} \rightarrow W^{\mu-m,\infty}}\leq K M^m_0(a), \mu \notin \n.\]
\end{theorem}
\begin{theorem} \label{paracomposition_Notions of microlocal analysis_Paradifferential Calculus_symbolic calculus para precised} %%%%%
Let $m,m' \in \r$, and $\rho>0$, $a \in \Gamma^m_\rho(\d)$and $b \in \Gamma^{m'}_\rho(\d)$. 
\begin{itemize}

\item Composition: Then $T_a T_b$ is a paradifferential operator with symbol: $$a \otimes b\in \Gamma^{m+m'}_\rho(\d),\text{ more precisely,}$$
\[
T^{\psi^{B,b}}_a T^{\psi^{B',b}}_b= T^{\psi^{\frac{BB'}{B+B'+1},b}}_{a\otimes b}.
\]
Moreover $T_a T_b- T_{a\#b}$ is of order $m+m'-\rho$ where $a \#b $ is defined by:
\[a \#b=\sum_{\abs{\alpha}<\rho }\frac{1}{i^{\abs{\alpha}}\alpha!} \partial^\alpha_\xi a \partial^\alpha_x b, \]
and there exists $r\in \Gamma^{m+m'-\rho}_0(\d)$ such that:
\[ M^{m+m'-\rho}_0(r) \leq K (M^m_\rho (a) M^{m'}_0(b)+M^m_\rho (a) M^{m'}_0(b)), \]
and we have
 \[T^{\psi^{B,b}}_a T^{\psi^{B',b}}_b- T^{\psi^{\frac{BB'}{B+B'+1},b}}_{a\#b}=T^{\psi^{\frac{BB'}{B+B'+1},b}}_r. \]

\item  Adjoint: The adjoint operator of $T_a$, $T_a^*$ is a paradifferential operator of order m with  symbol $a^*$ defined by:
\[a^*=\sum_{\abs{\alpha}<\rho} \frac{1}{i^{\abs{\alpha}}\alpha!}\partial^\alpha_\xi \partial^\alpha_x \bar{a}. \]
Moreover, for all $\mu \in \r$ there exists a constant K such that
\[ \norm{T_a^*-T_{a^*}}_{H^\mu \rightarrow H^{\mu-m+\rho}} \leq K M^m_\rho (a). \]
\end{itemize}
\end{theorem}
If $a=a(x)$ is a function of $x$ only then the paradifferential operator $T_a$ is called a paraproduct. 
It follows from Theorem \ref{paracomposition_Notions of microlocal analysis_Paradifferential Calculus_symbolic calculus para precised} and the Sobolev embedding that:
\begin{itemize}
\item If $a \in H^\alpha(\d)$ and $b \in H^\beta(\d)$ with $\alpha,\beta>\frac{d}{2}$, then
\[T_aT_b-T_{ab} \text{ is of order } -\bigg( \min\set{\alpha,\beta}-\frac{d}{2} \bigg).\]
\item If $a \in H^\alpha(\d)$ with $\alpha>\frac{d}{2}$, then
\[T_a^*-T_{a^*} \text{ is of order } -\bigg(\alpha-\frac{d}{2} \bigg).\]
\item If $a \in W^{r,\infty}(\d)$, $r\in \n$ then:
\[\norm{au-T_au}_{H^r} \leq C \norm{a}_{W^{r,\infty}} \norm{u}_{L^2}.\]
\end{itemize}
An important feature of paraproducts is that they are well defined for function $a=a(x)$ which are not $L^\infty$ but merely in some Sobolev spaces $H^r$ with $r<\frac{d}{2}$.
\begin{proposition}
Let $m>0$. If $a\in H^{\frac{d}{2}-m}(\d)$ and $u \in H^\mu(\d)$ then $T_au \in  H^{\mu-m}(\d)$. Moreover,
\[ \norm{T_a u}_{H^{\mu -m}}\leq K \norm{a}_{H^{\frac{d}{2} -m}}\norm{u}_{H^{\mu}} \]
\end{proposition}

A main feature of paraproducts is the existence of paralinearisation theorems which allow us to replace nonlinear expressions by paradifferential expressions, at the price of error terms which are smoother than the main terms.

\begin{theorem} \label{paracomposition_Notions of microlocal analysis_Paradifferential Calculus_paralinearisation para product} %%%%%
Let $\alpha, \beta \in \r $ be such that $\alpha,\beta> \frac{d}{2}$, then
\begin{itemize}
\item Bony's Linearization Theorem: For all $C^\infty$ function F, if $a \in H^\alpha (\d)$ then;%\footnote{In our recent work \cite{Ayman18} we give a generalization to this Theorem.}
\[ F(a)- F(0)-T_{F'(a)}a \in H^{2\alpha-\frac{d}{2}} (\d). \]
\item If $a\in H^\alpha(\d)$ and $b\in H^\beta(\d)$, then $ab-T_ab-T_ba \in H^{\alpha+ \beta-\frac{d}{2}} (\d)$. Moreover there exists a positive constant K independent of a and b such that:
\[\norm{ab-T_ab-T_ba}_{H^{\alpha+ \beta-\frac{d}{2}} }\leq K  \norm{a}_{H^\alpha} \norm{b}_{H^\beta}  .\]
\end{itemize}
\end{theorem}

\subsection{Paracomposition}\label{paracomposition_section Paracomposition}
We recall the main properties of the paracomposition operator first introduced by S. Alinhac in \cite{Alinhac86} to treat low regularity change of variables. Here we present the results we reviewed and generalized in some cases in \cite{Ayman18}.
\begin{theorem} \label{paracomposition_section Paracomposition_subsec paracomp on the euclidean space_theorem defintion of paracomposition}%%%%%%%%
 Let $\chi:\d^d \rightarrow \d^d$ be a $C^{1+r}$ diffeomorphism with $D\chi \in W^{r,\infty}$, $r>0, r\notin \n$ and take $s \in \r$ then the following map is continuous:  
   \begin{align*}
   H^s(\d^d) &\rightarrow H^s(\d^d)\\
  u &\mapsto \chi^* u=\sum_{k\geq 0}  \sum_{\substack{l\geq 0 \\ k-N \leq l \leq k+N}}P_l(D)u_k\circ \chi,
\end{align*} 
where $N \in \n$ is such that $2^{N}>\sup_{k,\d^d} \abs{\Phi_k D\chi}^{-1}$ and $2^{N}>\sup_{k,\d^d} \abs{\Phi_k D\chi}$.

Taking $\tilde{\chi}:\d^d \rightarrow \d^d$ a $C^{1+\tilde{r}}$ diffeomorphism with $D\chi \in W^{\tilde{r},\infty}$ map with $\tilde{r}>0$, then the previous operation has the natural fonctorial property:
\[\forall u \in H^s(\d^d) , \chi^* \tilde{\chi}^* u= ({\chi \circ \tilde{\chi}})^* u +Ru,\]   
 \[\text{with, }  R:H^{s}(\r^d) \rightarrow H^{s+min(r,\tilde{r})}(\r^d) \text{ continous}.\]
  \end{theorem}
We now give the key paralinearization theorem taking into account the paracomposition operator. 
\begin{theorem}  \label{paracomposition_section Paracomposition_subsec paracomp on the euclidean space_theorem paralinearisation of composition}%%%%%%%%
 Let $u$ be a $W^{1,\infty}(\d^d)$ map and $\chi:\d^d \rightarrow \d^d$ be a $C^{1+r}$  diffeomorphism with $D\chi \in W^{r,\infty}$, $r>0, r\notin \n$. Then:
 \[u \circ \chi(x)=\chi^* u(x)+ T_{Du\circ \chi}\chi(x)+ R_0(x)+R_1(x)+R_2(x)\]
 where the paracomposition given in the previous theorem verifies the estimates:
 \[\forall s \in \r, \norm{\chi^* u(x)}_{H^s}\leq C(\norm{D\chi}_{L^\infty})\norm{u(x)}_{H^s},\]
 \[u'\circ \chi \in  \Gamma^0_{W^{0,\infty}(\d^d)}(\d^d) \text{ for $u$  Lipchitz,}   \]
and the remainders verify the estimates: 

\[  \norm{R_0}_{H^{1+r +min(1+\rho,s-\frac{d}{2})}} \leq C\norm{D\chi}_{C_*^r}\norm{u}_{H^{1+s}} \]
	\[ \norm{R_1}_{H^{1+r+s}} \leq C(\norm{D\chi}_{L^\infty}) \norm{D\chi}_{C_*^r}\norm{u}_{H^{1+s}}. \]
	\[ \norm{R_2}_{H^{1+r+s}} \leq C(\norm{D\chi}_{L^\infty},\norm{D\chi^{-1}}_{L^\infty}) \norm{D\chi}_{C_*^r}\norm{u}_{H^{1+s}}. \]
	
	Finally the commutation between a paradifferential operator $a \in \Gamma^m_{\beta}(\d^d)$ and a paracomposition operator $\chi^*$ is given by the following
\[
	 \chi^* T_a u =T_{a^*} \chi^* u+T_{{q}^*} \chi^* u  \text{ with } q \in \Gamma^{m-\beta}_{0}(\d^d),
\]
where $a^*$ has the local expansion:
\begin{equation}\label{paracomposition_section Pull-back of pseudo and para- differential operators_theorem change of variable para eq1}%%%%%%%%
a^*(x,\xi) \sim \sum_{\substack{ \alpha \\ \abs{\alpha}\leq \lfloor min(r,\rho) \rfloor}}\frac{1}{\alpha!}\partial^\alpha a(\chi(x),D\chi^{-1}(\chi(x))^\top\xi)Q_{\alpha}(\chi(x),\xi)\in \Gamma^m_{\min(r,\beta)}(\d^d),
\end{equation}
where,
\[ P_{\alpha}(x',\xi)=D^\alpha_{y'}(e^{i(\chi^{-1}(y')-\chi^{-1}(x')-D \chi^{-1}(x')(y'-x')).\xi})_{|y'=x'} \]
and $Q_{\alpha}$ is polynomial in $\xi$ of degree $\leq \frac{\abs{\alpha}}{2}$, with $Q_{0}=1, Q_{1}=0$.
\end{theorem} 
\begin{remark} \label{paracomposition_section Paracomposition_subsec paracomp on the euclidean space_rem simplest example}%%%%%%
The simplest example for the paracomposition operator is when $\chi(x)=Ax$ is a linear operator and in that case we see that if $N$ is chosen sufficiently large in the definition of $\chi^*$:
\[u(Ax) = (Ax)^*u,\text{ and } T_{u'(Ax)}Ax = 0.\]
\end{remark}


\begin{thebibliography}{9}
\bibitem{Alazard09}
T. Alazard, G. Metivier:
\textit{Paralinearization of the Dirichlet to Neumann operator, and regularity of diamond waves}, Comm. Partial Differential Equations, 34 (2009), no. 10-12, 1632-1704.
\\
\bibitem{Alazard11}
T. Alazard, N. Burq, C. Zuily:
\textit{On the water waves equations with surface tension}, Duke Math. J. 158(3), 413-499 (2011).
\\
\bibitem{Alazard13}
T. Alazard, N. Burq, C. Zuily:
\textit{The water-waves equations: from Zakharov to Euler}, Studies in Phase Space Analysis with Applications to PDEs. Progress in Nonlinear Differential Equations and Their Applications Volume 84, 2013, pp 1-20.
\\
 \bibitem{Alazard15}
T. Alazard, P. Baldi,:
 \textit{Gravity capillary standing water waves},Arch. Ration. Mech. Anal., 217 (2015), no 3, 741-830.%
 \\
 \bibitem{Alazard18}
T. Alazard, P. Baldi, D. Han-Kwan:
\textit{Control for water waves}, J. Eur. Math. Soc., 20 (2018) 657-745.
 \\
  \bibitem{Alazard16}
T. Alazard, N. Burq, C. Zuily,:
\textit{Cauchy theory for the gravity water waves system with non localized initial data}, Ann. Inst. H. Poincar\'e Anal. Non Lin\'eaire, 33 (2016), 337-395.
\\
 \bibitem{Gerard10}
Previous results of T. Alazard, P. Baldi, P. G\'erard, Personal communication by T. Alazard.
 \\
  \bibitem{Alinhac86}
S. Alinhac \textit{Paracomposition et operateurs paradifferentiels}, 
Communications in Partial Differential Equations, 1986,11:1, 87-121.
  \\
  \bibitem{Beals 77}
 R. Beals: \textit{Characterization of pseudodifferential operators and applications}, Duke Math. J, Volume 44, Number 1 (1977), 45-57.
 \\
 \bibitem{Bony 13}
JM. Bony: \textit{On the Characterization of Pseudodifferential Operators (Old and New)}, Studies in Phase Space Analysis with Applications to PDEs. Progress in Nonlinear Differential Equations and Their Applications, vol 84. Birkhäuser, New York, NY. \url{https://doi.org/10.1007/978-1-4614-6348-1_2} 
 \\
  \bibitem{Bourgain93kdv}
J. Bourgain: \textit{Fourier transform restriction phenomena for certain lattice subsets and applications to nonlinear evolution equations,} Geom. Funct.Anal.3(1993), 3: 209. \url{https://doi.org/10.1007/BF01895688}.
 \\ 
 \bibitem{Castro10}
A. Castro, D. C\'ordoba, Francisco Gancedo,
 \textit{Singularity fornation in a surface wave model}, Nonlinearity, 2010.
 \\
 \bibitem{Coifman78}
R. Coifman, Y. Meyer,
 \textit{Au-del\`a des opérateurs pseudo-diff\'erentiels}. Ast\'erisque, no. 57 (1978), 210 p. \url{http://numdam.org/item/AST_1978__57__1_0/}
 \\
 \bibitem{Craig93}
W. Craig, C. Sulem: \textit{Numerical simulation of gravity water waves}, J. Comput. Phys. 108(1), 73-83 (1993).
 \\
 \bibitem{Gerard20}
P. G\'erard, Thomas Kappeler, \textit{On the Integrability of the Benjamin‐Ono Equation on the Torus}, Communications on Pure and Applied Mathematics, 2020.
 \\
\bibitem{Hormander71}
L. H\"ormander,  \textit{Fourier integral operators. I}, 
Acta Math. 127 (1971), 79-183.%
 \\
  \bibitem{Hormander97}
L. H{\"o}rmander:  \textit{Lectures on nonlinear hyperbolic differential equations}, 
Berlin ; New York : Springer, 1997.
 \\
\bibitem{Hormander90}
L. H\"ormander,  \textit{The Nash-Moser theorem and paradifferential operators}, 
Analysis, et cetera, 429-449,
Academic Press, Boston, MA, 1990.
 \\
  \bibitem{Hur12}
V. M. Hur,
 \textit{On the formation of singularities for surface water waves}, Communications in pure and applied analysis, volume 11, Number 4, (2012) .
 \\
 \bibitem{Hur17}
V. M. Hur,
 \textit{Wave Breaking in the Whitham equation}, Advances in Mathematics 317 (2017) 410-437 .
 \\
 \bibitem{Hur18}
V. M. Hur, L. Tao
 \textit{Wave Breaking in a Shallow Water Model}, SIAM J. Math. Anal., 50(1), 354-380.
 \\
 \bibitem{Ifrim17}
M. Ifrim, D. Tataru \textit{Well-posedness and dispersive decay of small data solutions for the Benjamin-Ono equation}, 
Annales scientifiques de l’ENS,  (4) 52 (2019), no. 2, 297-335.
 \\
 \bibitem{Ionescu07}
A. D. Ionescu, C.E. Kenig \textit{Global well posedness of the Benjamin-Ono equation in low-regularity spaces}, 
J. Amer, Math. Soc., \textbf{20} (2007), 753-798.
 \\
  \bibitem{Kappler06}
T. Kappeler, P. Topalov \textit{Global wellposedness of KdV in $H^{-1}(\t,\r)$}, 
Duke Math. J.
Volume 135, Number 2 (2006), 327-360.
 \\
 \bibitem{Killip19}
R. Killip, M. Vi\c{s}an,
 \textit{KdV is well-posed in $H^{–1}$}, Annals of Mathematics
Vol. 190, No. 1 (July 2019), pp. 249-305.
 \\
  \bibitem{Saut14}
 C. Klein, and J.-C. Saut, 
 \textit{A numerical approach to blow-up issues for dispersive
perturbations of Burgers’ equation}, Phys. D 295/296 (2015), pp. 46–65.
\\
 \bibitem{Koch03}
H. Koch and N. Tzvetkov, \textit{On the local well-posedness of the Benjamin-Ono equation in $H^s(\r)$}, Int. Math. Res. Not., 26 (2003), 1449-1464.
\\
 \bibitem{Lannes2005}
D. Lannes,  \textit{Well-posedness of the water waves equations}, 
J.  Amer. Math. Soc., 18(3):605-654 (electronic), 2005.
 \\
 \bibitem{Linares14}
 F. Linares, D. Pilod and J.-C. Saut, 
 \textit{Dispersive perturbations of Burgers and hyperbolic equations I: local theory}, SIAM J. Math. Analysis, \textbf{46} (2014), 1505-1537.
\\
\bibitem{Metivier08}
G. Metivier,  \textit{Para-differential calculus and applications to the Cauchy problem for non linear systems}, Ennio de Giorgi Math. res. Center Publ., Edizione della Normale, 2008.%
\\
  \bibitem{Saut02}
 L. Molinet, J. C. Saut and N. Tzvetkov
\textit{Well-posedness and ill-posedness results for the Kadomtsev-Petviashvili-I equation}, Duke Math. J. Volume 115, Number 2 (2002), 353-384.
\\
\bibitem{Molinet08}
  L. Molinet,
\textit{Global Well-Posedness in $L^2$ for the Periodic Benjamin-Ono Equation}, American Journal of Mathematics, Johns Hopkins University Press, 2008, 130 (3), pp.635-683.%
\\ 
\bibitem{Molinet12}
  L. Molinet,
\textit{Sharp ill-posedness results for the KdV and mKdV
equations on the torus}, Advances in Mathematics
Volume 230, Issues 4-6, July-August 2012, Pages 1895-1930.
\\
\bibitem{Molinet18}
 L. Molinet, D. Pilod, S. Vento,
\textit{On well-posedness for some dispersive perturbations of Burgers' equation}, Annales de l'Institut Henri Poincare\'e C, Analyse non lin\'eaire,
Volume 35, Issue 7, November 2018, Pages 1719-1756.
\\
\bibitem{Pasqualotto21}
 F. Pasqualotto, Sung-Jin Oh 
\textit{Gradient blow-up for dispersive and dissipative perturbations of the Burgers equation}, Preprint: arXiv:2107.07172, 2021.
 \\ 
  \bibitem{Ayman18}
A. R. Said:
\textit{On Paracompisition and change of variables in Paradifferential operators}, arXiv preprint, arXiv:2002.02943.
 \\ 
 \bibitem{Ayman19}
A. R. Said:
\textit{A geometric proof of the Quasi-linearity of the Water-Waves system and the incompressible Euler equations}, To appear in SIAM Journal on Mathematical Analysis.
\\
 \bibitem{Ayman20'}
A. R. Said:
\textit{On the Cauchy problem of dispersive Burgers type equations}, to appear in Indiana University Mathematics Journal.
\\
\bibitem{Saut13}
 J. C. Saut 
\textit{Asymptotic Models for Surface and Internal waves}, 29 Brazilian Mathematical Colloquia, IMPA Mathematical Publications ,2013.
 \\
 \bibitem{Saut18}
 J. C. Saut 
\textit{Benjamin-Ono and Intermediate Long Wave equation : modeling, IST and PDE}, arXiv preprint, arXiv:1811.08652, 2018.
 \\
 \bibitem{Saut20'}
 J. C. Saut, Y. Wang 
\textit{The Wave Breaking for Whitham-Type
Equations Revisited}, arXiv preprint, arXiv:2006.03803.
 \\
  \bibitem{Shnirelman05}
A. Shnirelman: \textit{Microglobal Analysis of the Euler Equations}, J. math. fluid mech. (2005) 7(Suppl 3): S387. \url{https://doi.org/10.1007/s00021-005-0167-5}.
\\
\bibitem{Tao04}
T. Tao:
\textit{Global well-posedness of the Benjamin-Ono equation in $H^1(\r)$, j. Hyperbolic Differ, Equ \textbf{1} (2004), 27-49}.%
 \\ 
\bibitem{Taylor07}
  M. E. Taylor,
  \textit{Tools for PDE: Pseudodifferential Operators, Paradifferential Operators, and Layer Potentials},
 American Mathematical Soc., 2007.%
 \\
 \bibitem{Taylor91}
  M. E. Taylor,
  \textit{Pseudodifferential Operators and Nonlinear PDE},
Brickhauser, Boston, 1991.%
 \\
   \bibitem{Zakharov68}
V.E. Zakharov:
\textit{Stability of periodic waves of finite amplitude on the surface of a deep fluid}, J. Appl. Mech. Techn. Phys. 9(2), 190-194 (1968).
 \\ 
\end{thebibliography}
\end{document}